\documentclass[oneside,leqno,10pt]{article}
\usepackage{amssymb,amsmath,latexsym,amsthm,stmaryrd,amscd}

\def\ga{\mathfrak{a}}
\def\gb{\mathfrak{b}}
\def\gc{\mathfrak{c}}

\def\ge{\mathfrak{e}}

\def\gg{\mathfrak{g}}
\def\gh{\mathfrak{h}}

\def\gk{\mathfrak{k}}
\def\gl{\mathfrak{l}}
\def\gm{\mathfrak{m}}
\def\gn{\mathfrak{n}}
\def\go{\mathfrak{o}}

\def\gq{\mathfrak{q}}
\def\gr{\mathfrak{r}}
\def\gs{\mathfrak{s}}
\def\gt{\mathfrak{t}}
\def\gu{\mathfrak{u}}
\def\gv{\mathfrak{v}}
\def\gw{\mathfrak{w}}

\def\gz{\mathfrak{z}}

\def\ggg{> \hskip -5 pt >}

\def\C{\mathbb{C}}

\def\H{\mathbb{H}}

\def\R{\mathbb{R}}

\def\cC{\mathcal{C}}

\def\cF{\mathcal{F}}

\def\cH{\mathcal{H}}

\def\cO{\mathcal{O}}

\def\cU{\mathcal{U}}

\def\Ad{{\rm Ad}}
\def\ad{{\rm ad}\,}
\def\Pf{{\rm Pf}}

\def\Ind{{\rm Ind\,}}

\def\tr{{\rm trace\,}}

\def\Det{\rm Det}

\newtheorem{theorem}[equation]{Theorem}

\newtheorem{lemma}[equation]{Lemma}

\newtheorem{corollary}[equation]{Corollary}

\newtheorem{proposition}[equation]{Proposition}

\newtheorem{definition}[equation]{Definition}

\newtheorem{remark}[equation]{Remark}

\newtheorem{construction}[equation]{Construction}

\def\sideremark#1{\ifvmode\leavevmode\fi\vadjust{\vbox to0pt{\vss
 \hbox to 0pt{\hskip\hsize\hskip1em
\vbox{\hsize2cm\tiny\raggedright\pretolerance10000 
 \noindent #1\hfill}\hss}\vbox to8pt{\vfil}\vss}}} 

\title{Stepwise Square Integrability for Nilradicals of Parabolic Subgroups\\
	and Maximal Amenable Subgroups}

\author{Joseph A. Wolf\footnote{Research partially supported by the Simons
Foundation and by the Dickson Emeriti Professorship.\newline
}}
\date{}

\begin{document}

\maketitle

\abstract{In a series of recent papers (\cite{W2012}, \cite{W2013}, 
\cite{W2014}, \cite{W2015}) we extended the notion of
square integrability, for representations of nilpotent Lie groups, to that
of stepwise square integrability.  There we discussed a number of applications
based on the fact that nilradicals of minimal parabolic subgroups of
real reductive Lie groups are stepwise square integrable.  In Part I we prove 
stepwise square integrability for nilradicals of arbitrary parabolic 
subgroups of real reductive Lie groups.  This is technically more delicate
than the case of minimal parabolics.  We further discuss
applications to Plancherel formulae and Fourier inversion formulae for
maximal exponential solvable subgroups of parabolics and maximal amenable
subgroups of real reductive Lie groups.  Finally, in Part II,
we extend a number of those results to (infinite dimensional) direct limit 
parabolics.  These extensions involve an infinite dimensional version of the
Peter-Weyl Theorem, construction of a direct limit Schwartz space, and
realization of that Schwartz space as a dense subspace of the corresponding
$L^2$ space.
}
\smallskip

{\footnotesize
{\textrm
\tableofcontents
}
}
\newpage

\centerline{\Large \bf Part I: Finite Dimensional Theory}

\section{Stepwise Square Integrable Representations}
\label{sec1}
\setcounter{equation}{0}
There is a very precise theory of square integrable representations of
nilpotent Lie groups due to Moore and the author \cite{MW1973}.  It is 
based on the Kirillov's general representation theory \cite{K1962} for 
nilpotent Lie groups, in which he introduced coadjoint orbit theory to
the subject.  When a nilpotent Lie group has square integrable
representations its representation theory, Plancherel and Fourier inversion 
formulae, and other aspects of real analysis, become explicit and transparent.
\medskip

Somewhat later it turned out that many familiar nilpotent
Lie groups have foliations, in fact semidirect product towers composed of
subgroups that have square integrable representations.  These include
nilradicals of minimal parabolic subgroups, e.g. the group of strictly 
upper triangular real or complex matrices.  All the analytic benefits of
square integrability carry over to stepwise square integrable nilpotent
Lie groups.
\medskip

In order to indicate our results here we must recall the notions of square 
integrability and stepwise square integrability in sufficient detail
to carry them over to nilradicals of arbitrary parabolic subgroups of
real reductive Lie groups.
\medskip

A connected simply connected Lie group $N$
with center $Z$ is called {\em square integrable}, or is said to
{\em have square integrable representations}, if it has unitary
representations $\pi$ whose coefficients $f_{u,v}(x) = 
\langle u, \pi(x)v\rangle$ satisfy $|f_{u,v}| \in L^2(N/Z)$.  
C.C. Moore and the author worked out the structure and representation
theory of these groups \cite{MW1973}.  If $N$ has one 
such square integrable representation then there is a certain polynomial
function $\Pf(\lambda)$ on the linear dual space $\gz^*$ of the Lie algebra of
$Z$ that is key to harmonic analysis on $N$.  Here $\Pf(\lambda)$ is the
Pfaffian of the antisymmetric bilinear form on $\gn / \gz$ given by
$b_\lambda(x,y) = \lambda([x,y])$.  The square integrable
representations of $N$ are the $\pi_\lambda$ (corresponding to coadjoint
orbits $\Ad^*(N)\lambda$) where $\lambda \in \gz^*$ with $\Pf(\lambda) \ne 0$,
Plancherel  almost irreducible unitary representations of $N$ are square
integrable, and, up to an explicit constant, 
$|\Pf(\lambda)|$ is the Plancherel density on the unitary
dual $\widehat{N}$ at $\pi_\lambda$.  Concretely,
\medskip

\begin{theorem}{\rm \cite{MW1973}}\label{sqint-plancherel}
Let $N$ be a connected simply connected nilpotent Lie group that has
square integrable representations.  Let $Z$ be its center and $\gv$
a vector space complement to $\gz$ in $\gn$, so
$\gv^* = \{\gamma \in \gn^* \mid \gamma|_\gz = 0\}$.
If $f$ is a Schwartz class
function $N \to \C$ and $x \in N$ then
\begin{equation}\label{inversion}
f(x) = c\int_{\gz^*} \Theta_{\pi_\lambda}(r_xf) |\Pf(\lambda)|d\lambda
\end{equation}
where $c = d!2^d$ with $2d = \dim \gn/\gz$\,, $r_xf$ is the
right translate $(r_xf)(y) = f(yx)$, and $\Theta$ is the distribution
character
\begin{equation}\label{def-dist-char}
\Theta_{\pi_\lambda}(f) = c^{-1}|\Pf(\lambda)|^{-1}\int_{\cO(\lambda)}
        \widehat{f_1}(\xi)d\nu_\lambda(\xi) \text{ for } f \in \cC(N).
\end{equation}
Here $f_1$ is the lift
$f_1(\xi) = f(\exp(\xi))$ of $f$ from $N$ to $\gn$, 
$\widehat{f_1}$ is its classical Fourier transform,
$\cO(\lambda)$ is the coadjoint orbit $\Ad^*(N)\lambda = \gv^* + \lambda$,
and $d\nu_\lambda$ is the translate of normalized Lebesgue measure from
$\gv^*$ to $\Ad^*(N)\lambda$.  
\end{theorem}
\medskip

\noindent More generally, we will consider the situation where
\begin{equation}\label{setup}
\begin{aligned}
N = &L_1L_2\dots L_{m-1}L_m \text{ where }\\
 &\text{(a) each factor $L_r$ has unitary representations with coefficients in 
$L^2(L_r/Z_r)$,} \\
 &\text{(b) each } N_r := L_1L_2\dots L_r \text{ is a normal subgroup of } N
   \text{ with } N_r = N_{r-1}\rtimes L_r \text{ semidirect,}\\
 &\text{(c) if $r \geqq s$ then $[\gl_r,\gz_s] = 0$}
\end{aligned}
\end{equation}
The conditions of (\ref{setup}) are sufficient to construct the representations
of interest to us here,  but not sufficient to compute the Pfaffian that is
the Plancherel density.  For that, in the past we used the
{\em strong computability condition}
\begin{equation}\label{c-strong} 
\begin{aligned}
&\text{Decompose }\gl_r = \gz_r + \gv_r \text{ and } \gn = \gs + \gv
	\text{ (vector space direct) where }
 	\gs = \oplus\, \gz_r \text{ and } \gv = \oplus\, \gv_r; 
	\text{ then}  \\
&\text{\phantom{XXX(d)} }[\gl_r,\gl_s] \subset \gv_s \text{ for } r > s.
\end{aligned}
\end{equation} 
The problem is that the strong computability condition 
(\ref{c-strong}) can fail for some non-minimal real
parabolics, but we will see that, for the Plancherel density, we only need
the {\em weak computability condition}
\begin{equation}\label{c-weak}
\begin{aligned}
&\text{Decompose }\gl_r = \gl_r' \oplus \gl_r''\text{, direct sum of ideals,
	where } \gl_r'' \subset \gz_r \text{ and } \gv_r \subset \gl_r'; 
	\text{ then \phantom{XXXXXXXXXX}}  \\
&\text{\phantom{XXX(e)} }[\gl_r,\gl_s] \subset \gl_s'' + \gv_s 
	\text{ for } r > s.
\end{aligned}
\end{equation} 
where we retain $\gl_r = \gz_r + \gv_r \text{ and } \gn = \gs + \gv$.
\medskip

In the setting of (\ref{setup}), (\ref{c-strong}) and
(\ref{c-weak})  it is useful to denote
\begin{equation}\label{c-d}
\begin{aligned}
&\text{(a) }d_r = \tfrac{1}{2}\dim(\gl_r/\gz_r) \text{ so }
        \tfrac{1}{2} \dim(\gn/\gs) = d_1 + \dots + d_m\,,
        \text{ and } c = 2^{d_1 + \dots + d_m} d_1! d_2! \dots d_m!\\
&\text{(b) }b_{\lambda_r}: (x,y) \mapsto \lambda_r([x,y]) 
        \text{ viewed as a bilinear form on } \gl_r/\gz_r \\
&\text{(c) }S = Z_1Z_2\dots Z_m = Z_1 \times \dots \times Z_m \text{ where } Z_r
        \text{ is the center of } L_r \\
&\text{(d) }P: \text{ polynomial } P(\lambda) = \Pf(b_{\lambda_1})
        \Pf(b_{\lambda_2})\dots \Pf(b_{\lambda_m}) \text{ on } \gs^* \\
&\text{(e) }\gt^* = \{\lambda \in \gs^* \mid P(\lambda) \ne 0\} \\
&\text{(f) } \pi_\lambda \in \widehat{N} \text{ where } \lambda \in \gt^*: 
    \text{ irreducible unitary representation of } N = L_1L_2\dots L_m
        \text{ as follows.} 
\end{aligned}
\end{equation}

\begin{construction}\label{construction}
{\rm \cite{W2013} Given $\lambda \in \gt^*$, in other words 
$\lambda = \lambda_1 + \dots + \lambda_m$ where $\lambda_r \in \gz_r$ with
each $\Pf(b_{\lambda_r}) \ne 0$, we construct 
$\pi_\lambda \in \widehat{N}$ by recursion on $m$.  If $m = 1$ then
$\pi_\lambda$ is a square integrable representation of $N = L_1$\,.  Now
assume $m > 1$.  Then we have the irreducible unitary representation 
$\pi_{\lambda_1 + \dots + \lambda_{m-1}}$ of $L_1L_2\dots L_{m-1}$\,.
and (\ref{setup}(c)) shows that $L_m$ stabilizes the unitary equivalence 
class of $\pi_{\lambda_1 + \dots + \lambda_{m-1}}$.  Since $L_m$ is
topologically contractible the Mackey obstruction vanishes and 
$\pi_{\lambda_1 + \dots + \lambda_{m-1}}$ extends to an irreducible unitary
representation $\pi^\dagger_{\lambda_1 + \dots + \lambda_{m-1}}$ on $N$ on the
same Hilbert space.  View the square integrable representation 
$\pi_{\lambda_m}$ of $L_m$ as a representation of $N$ whose kernel contains
$L_1L_2\dots L_{m-1}$\,.  Then we define
$\pi_\lambda = \pi^\dagger_{\lambda_1 + \dots + \lambda_{m-1}} \widehat\otimes
\pi_{\lambda_m}$\,.}
\hfill$\diamondsuit$
\end{construction}

\begin{definition}\label{stepwise2}
{\rm The representations $\pi_\lambda$ of (\ref{c-d}(f)), constructed
just above, are the
{\it stepwise square integrable} representations of $N$ relative to
the decomposition (\ref{setup}).  If $N$ has stepwise square integrable
representations relative to (\ref{setup}) we will say that
$N$ is {\it stepwise square integrable}.}\hfill $\diamondsuit$
\end{definition}

\begin{remark}\label{2-steps}
{\rm Construction \ref{construction} of the stepwise square integrable
representations $\pi_\lambda$ uses (\ref{setup}(c)),
$[\gl_r,\gz_s] = 0$ for $r > s$, so that $L_r$ stabilizes the unitary
equivalence class of $\pi_{\lambda_1 + \dots + \lambda_{r-1}}$.  The
condition (\ref{c-strong}), $[\gl_r,\gl_s] \subset \gv$ for $r > s$, enters the
picture in proving that the polynomial $P$ of (\ref{c-d}(d)) is the
Pfaffian $\Pf = \Pf_\gn$ of $b_\lambda$ on $\gn/\gs$. However we don't need 
that, and the weaker (\ref{c-weak}) is sufficient to show that $P$
is the Plancherel density.  See Theorem \ref{plancherel-general} below.}
\hfill$\diamondsuit$
\end{remark}

\begin{lemma}\label{concentrated} {\rm \cite{W2013}}
Assume that $N$ has stepwise square integrable representations.
Then Plancherel measure is concentrated on the set
$\{\pi_\lambda \mid \lambda \in \gt^*\}$ of all stepwise square
integrable representations.
\end{lemma}

Theorem \ref{sqint-plancherel} extends to the stepwise square
integrable setting, as follows.

\begin{theorem}\label{plancherel-general}
Let $N$ be a connected simply connected nilpotent Lie group that
satisfies {\rm (\ref{setup})} and {\rm (\ref{c-weak})}.  Then Plancherel 
measure for $N$ is
concentrated on $\{\pi_\lambda \mid \lambda \in \gt^*\}$.
If $\lambda \in \gt^*$, and if $u$ and $v$ belong to the
representation space $\cH_{\pi_\lambda}$ of $\pi_\lambda$,  then
the coefficient $f_{u,v}(x) = \langle u, \pi_\nu(x)v\rangle$
satisfies
\begin{equation}\label{frob1}
||f_{u,v}||^2_{L^2(N / S)} = \frac{||u||^2||v||^2}{|P(\lambda)|}\,.
\end{equation}
The distribution character $\Theta_{\pi_\lambda}$ of $\pi_{\lambda}$ satisfies
\begin{equation}\label{frob2}
\Theta_{\pi_\lambda}(f) = c^{-1}|P(\lambda)|^{-1}\int_{\cO(\lambda)}
        \widehat{f_1}(\xi)d\nu_\lambda(\xi) \text{ for } f \in \cC(N)
\end{equation}
where $\cC(N)$ is the Schwartz space, $f_1$ is the lift 
$f_1(\xi) = f(\exp(\xi))$, $\widehat{f_1}$ is its classical Fourier transform,
$\cO(\lambda)$ is the coadjoint orbit $\Ad^*(N)\lambda = \gv^* + \lambda$,
and $d\nu_\lambda$ is the translate of normalized Lebesgue measure from
$\gv^*$ to $\Ad^*(N)\lambda$.  The Plancherel formula on $N$ is
\begin{equation}\label{frob3}
f(x) = c\int_{\gt^*} \Theta_{\pi_\lambda}(r_xf) |P(\lambda)|d\lambda
        \text{ for } f \in \cC(N).
\end{equation}
\end{theorem}
\medskip

Theorem \ref{plancherel-general} is proved in \cite{W2013} for groups
$N$ that satisfy (\ref{setup}) together with (\ref{c-strong}).  We will
need it for (\ref{setup}) together with the somewhat less restrictive
(\ref{c-weak}).  The only point where the argument needs a slight
modification is in the proof of (\ref{frob1}).  
The action of $L_m$ on $\gl_1 + \dots + \gl_{m-1}$ 
is unipotent, so there is an $L_m$-invariant measure preserving decomposition
$N_m/S_m = \left ( L_1/Z_1 \right ) \times \dots \times 
\left ( N_m/Z_m \right )$.  The case $m = 1$ is the property 
$|f_{u,v}|^2_{L^2(L_1/Z_1)} = \frac{||u||^2||v||^2}{|\Pf(\lambda)|} < \infty$ 
of coefficients
of square integrable representations.  By induction on $m$,
$|f_{u,v}|^2_{L^2(N_{m-1}/S_{m-1}))} = 
\frac{||u||^2||v||^2}{|\Pf(\lambda_1)\dots\Pf(\lambda_{m-1})|}$ for 
$N_{m-1}$\,.  Let $\pi^\dagger$
be the extension of $\pi \in \widehat{N_{m-1}}$ to $N_m$\,.  Let 
$u, v \in \cH_{\pi_{\lambda_1+\dots\lambda_{m-1}}}$ and write $v_y$ for 
$\pi^\dagger_{\lambda_1+\dots + \lambda_{m-1}}(y)v$.  Let 
$u',v'\in\cH_{\pi_{\lambda_m}}$\,. 
$$
\begin{aligned}
||f_{u\otimes u', v\otimes v'}&||^2_{L^2(N/S)} =
        \int_{N/S}
          |\langle u,\pi^\dagger_{\lambda_1+\dots\lambda_{m-1}}(xy)v\rangle|^2
          |\langle u', \pi_{\lambda_m}(y)v'\rangle|^2
          d(xyS_m) \\
&= \int_{L_m/Z_m} |\langle u', \pi_{\lambda_m}(y)v'\rangle|^2
   \left ( \int_{N_{m-1}/S_{m-1}}
        |\langle u,\pi^\dagger_{\lambda_1+\dots\lambda_{m-1}}(xy)v\rangle|^2
        d(xS_{m-1}) \right ) d(yZ_m) \\
&= \int_{L_m/Z_m} |\langle u', \pi_{\lambda_m}(y)v'\rangle|^2
   \left ( \int_{N_{m-1}/S_{m-1}}
        |\langle u, \pi^\dagger_{\lambda_1+\dots\lambda_{m-1}}(x)v_y\rangle|^2
        d(xS_{m-1}) \right ) d(yZ_m) \\
&= \int_{L_m/Z_m} |\langle u', \pi_{\lambda_m}(y)v'\rangle|^2
   \left ( \int_{N_{m-1}/S_{m-1}}
        |\langle u, \pi_{\lambda_1+\dots\lambda_{m-1}}(x)v_y\rangle|^2
        d(xS_{m-1}) \right ) d(yZ_m) \\
&= \tfrac{||u||^2||v_y||^2}{|\Pf(\lambda_1)\dots\Pf(\lambda_{m-1})|}
	\int_{N_m/Z_m} |\langle u', \pi_{\lambda_m}(y)v'\rangle|^2d(yZ_m) \\ 
&= \tfrac{||u||^2||v||^2}{|\Pf(\lambda_1)\dots\Pf(\lambda_{m-1})|}
	\int_{N_m/Z_m} |\langle u', \pi_{\lambda_m}(y)v'\rangle|^2d(yZ_m) 
= \frac{||u\otimes u'||^2 ||v \otimes v'||^2}
	{|\Pf(\lambda_1)\dots\Pf(\lambda_m)|} < \infty.
\end{aligned}
$$
Thus Theorem \ref{plancherel-general} is valid as stated.
\medskip

The first goal of this note is to show that if $N$ is the nilradical of a 
parabolic subgroup $Q$ of a real reductive Lie group, then $N$ is stepwise
square integrable, specifically that it satisfies (\ref{setup}) and
(\ref{c-weak}), so that Theorem \ref{plancherel-general} applies to it.
That is Theorem \ref{gen-setup}.  The second goal is to examine applications
to Fourier analysis on the parabolic $Q$ and several important subgroups, such 
as the maximal split solvable subgroups and the maximal amenable subgroup
of $Q$.  The third goal is to extend all these results to direct limit 
parabolics in a certain class of infinite dimensional real reductive Lie 
groups.
\medskip

In Section \ref{sec2} we recall the restricted root machinery used
in \cite{W2013} to show that nilradicals of minimal parabolics are
stepwise square integrable.  In Section \ref{sec3} we make a first
approximation to refine that machinery to apply it to general parabolics.
That is enough to see that those parabolics satisfy (\ref{setup}), and 
to construct their stepwise square integrable representations.  But 
it not quite enough to compute the Plancherel density.  Then in 
Section \ref{sec4} we introduce an appropriate modification of the 
earlier stepwise square integrable machinery.  We prove (\ref{c-weak}) in
general and use the result to compute the Plancherel density and verify the
estimates and inversion formula of Theorem \ref{plancherel-general} 
for arbitrary parabolic
subgroups of real reductive Lie groups.  The main result is
Theorem \ref{gen-setup}.
\medskip

In Section \ref{sec5} we apply Theorem \ref{gen-setup} to obtain
explicit Plancherel and Fourier inversion formulae for the maximal
exponential solvable subgroups $AN$ in real parabolic
subgroups $Q = MAN$, following the lines of the minimal parabolic 
case studied in \cite{W2014}.  The key point here is computation of
the Dixmier-Puk\'anszky operator $D$ for
the group $AN$.  Recall that $D$ is a pseudo-differential operator
that compensates lack of unimodularity in $AN$.
\medskip

There are technical obstacles to extending our results to non-minimal
parabolics $Q = MAN$, many involving the orbit types for noncompact 
reductive groups $M$, but in Section \ref{sec6} we do carry out the extension
to the maximal amenable subgroups $(M \cap K)AN$.  This covers all the
maximal amenable subgroups of $G$ that satisfy a certain technical condition
\cite{M1979}.
\medskip

That ends Part I: Finite Dimensional Theory.  We go on to Part II: Infinite
Dimensional Theory.
\medskip

In Section \ref{sec7} we discuss infinite dimensional direct limits of nilpotent
Lie groups and the setup for studying direct limits of stepwise square integrable
representations.  Then in Section \ref{sec8} we introduce the machinery of 
propagation, which will allow us to deal with nilradicals of direct limit
parabolics.
\medskip

In Section \ref{sec9} we apply this machinery to an $L^2$ space for the
direct limit nilradicals.  This $L^2$ space is formed using the formal degree
inherent in stepwise square integrable representations, and it is not immediate
that its elements are functions.  But we also introduce a limit Schwartz space,
based on matrix coefficients of $C^\infty$ vectors for stepwise square
integrable representations.  It is a well defined $LF$ (limit of Fr\' echet) space,
sitting naturally in the $L^2$ space, and we can view that $L^2$ space as its
Hilbert space completion.  That is Proposition \ref{dense-embedding}.  We
follow it with a fairly explicit Fourier Inversion Formula, 
Theorem \ref{limit-inversion}.
\medskip

In Section \ref{sec10} we work out the corresponding results for the maximal
exponential locally solvable subgroup $AN$ of the direct limit parabolic
$Q = MAN$.  We have to be careful about the Schwartz space and the lack of a 
Dixmier-Puk\'anszky operator in the limit, but the results of Section \ref{sec9}
to extend from $N$ to $AN$.  See Proposition \ref{dense-embedding-an} and
Theorem \ref{limit-inversion-an}.  In Section \ref{sec11} we develop similar
results for the maximal lim-compact subgroup $U$ of $M$, carefully avoiding 
the analytic complications that would result from certain classes of Type II
and Type III representations.  
\medskip

In Section \ref{sec12} we fit the results of Sections \ref{sec9} and \ref{sec11}
together for an analysis of the $L^2$ space, the Schwartz space, and the Fourier
Inversion formula, for the limit group $UN$ in the parabolic $Q = MAN$.  Finally,
in Section \ref{sec13}, we combine the results of Sections \ref{sec10} and
\ref{sec12} for the corresponding results on the maximal amenable subgroup
$UAN$ of the limit parabolic $Q$.  See Proposition \ref{dense-embedding-uan}
and Theorem \ref{limit-inversion-uan}.

\section{Specialization to Minimal Parabolics} \label{sec2}
\setcounter{equation}{0}
In order to prove our result for nilradicals of arbitrary parabolics we
need to study the construction that gives the decomposition 
$N = L_1L_2\dots L_m$ of \ref{setup} and the form of the Pfaffian polynomials 
for the individual the square integrable layers $L_r$\,.
\medskip

Let $G$ be a connected real reductive Lie group, $G = KAN$ an Iwasawa
decompsition, and $Q = MAN$ the corresponding minimal parabolic subgroup.
Complete $\ga$ to a Cartan subalgebra $\gh$ of $\gg$.
Then $\gh = \gt + \ga$ with $\gt = \gh \cap \gk$.  Now we have root systems
\begin{itemize}
\item $\Delta(\gg_\C,\gh_\C)$: roots of $\gg_\C$ relative to $\gh_\C$
(ordinary roots),

\item $\Delta(\gg,\ga)$: roots of $\gg$ relative to $\ga$ (restricted roots),

\item $\Delta_0(\gg,\ga) = \{\alpha \in \Delta(\gg,\ga) \mid 
        2\alpha \notin \Delta(\gg,\ga)\}$ (nonmultipliable restricted roots).
\end{itemize}
The choice of $\gn$ is the same as the choice of a positive restricted
root systen $\Delta^+(\gg,\ga)$.  Define
\begin{equation}\label{cascade}
\begin{aligned}
&\beta_1 \in \Delta^+(\gg,\ga) \text{ is a maximal positive restricted root
and }\\
& \beta_{r+1} \in \Delta^+(\gg,\ga) \text{ is a maximum among the roots of }
\Delta^+(\gg,\ga) \text{ orthogonal to all } \beta_i \text{ with } i \leqq r
\end{aligned}
\end{equation}
The resulting roots (we usually say {\em root} for {\em restricted root})
$\beta_r$\,, $1 \leqq r \leqq m$, are mutually strongly 
orthogonal, in particular mutually 
orthogonal, and each $\beta_r \in \Delta_0(\gg,\ga)$.  
For $1\leqq r \leqq m$ define
\begin{equation}\label{layers}
\begin{aligned}
&\Delta^+_1 = \{\alpha \in \Delta^+(\gg,\ga) \mid \beta_1 - \alpha \in \Delta^+(\gg,\ga)\} 
\text{ and }\\
&\Delta^+_{r+1} = \{\alpha \in \Delta^+(\gg,\ga) \setminus (\Delta^+_1 \cup \dots \cup \Delta^+_r)
        \mid \beta_{r+1} - \alpha \in \Delta^+(\gg,\ga)\}.
\end{aligned}
\end{equation}
We know 
\cite[Lemma 6.1]{W2013} that if $\alpha \in \Delta^+(\gg,\ga)$ then either
$\alpha \in \{\beta_1, \dots , \beta_m\}$
or $\alpha$ belongs to exactly one of the sets $\Delta^+_r$\,.  
\medskip

The layers are are the
\begin{equation}\label{def-l}
\gl_r = \gg_{\beta_r} + {\sum}_{\Delta^+_r}\, \gg_\alpha 
\text{ for } 1\leqq r\leqq m
\end{equation}
Denote
\begin{equation}
s_{\beta_r} \text{ is the Weyl group reflection in } \beta_r
\text{ and } \sigma_r: \Delta(\gg,\ga) \to \Delta(\gg,\ga) \text{ by }
\sigma_r(\alpha) = -s_{\beta_r}(\alpha).
\end{equation}
Then $\sigma_r$ leaves $\beta_r$ fixed and preserves $\Delta^+_r$.
Further, if $\alpha, \alpha' \in \Delta^+_r$ then $\alpha + \alpha'$
is a (restricted) root if and only if $\alpha' = \sigma_r(\alpha)$, and in 
that case $\alpha + \alpha' = \beta_r$\,.
\medskip

From this it follows \cite[Theorem 6.11]{W2013} that $N = L_1L_2\dots L_m$
satisfies (\ref{setup}) and (\ref{c-strong}), so it has stepwise square 
integrable representations.
Further \cite[Lemma 6.4]{W2013} the $L_r$ are Heisenberg groups in the sense 
that if $\lambda_r \in \gz_r^*$ with $\Pf_{\gl_r}(\lambda_r) \ne 0$ then
$\gl_r/\ker\lambda_r$ is an ordinary Heisenberg group of dimension
$\dim\gv_r + 1$.
\medskip

\section{Intersection with an Arbitrary Real Parabolic}
\label{sec3}
\setcounter{equation}{0}
Every parabolic subgroup of $G$ is conjugate to a parabolic that 
contains the minimal parabolic $Q = MAN$.  Let $\Psi$ denote
the set of simple roots for the positive system $\Delta^+(\gg,\ga)$.  Then
the parabolic subgroups of $G$ that contain $Q$ are in one to one
correspondence with the subsets $\Phi \subset \Psi$, say $Q_\Phi
\leftrightarrow \Phi$, as follows.  Denote $\Psi = \{\psi_i\}$ and set
\begin{equation}\label{para-roots} 
\begin{aligned}
\Phi^{red} &= \left \{\alpha = {\sum}_{\psi_i \in \Psi} 
	n_i\psi_i \in \Delta(\gg,\ga) \mid
	n_i = 0 \text{ whenever } \psi_i \notin \Phi \right \} \\
\Phi^{nil} &= \left \{\alpha = {\sum}_{\psi_i \in \Psi} 
        n_i\psi_i \in \Delta^+(\gg,\ga) \mid n_i > 0 \text{ for some }
	\psi_i \notin \Phi \right \}.
\end{aligned}
\end{equation}
Then, on the Lie algebra level, $\gq_\Phi = \gm_\Phi + \ga_\Phi + \gn_\Phi$
where
\begin{equation}\label{para-pieces}
\begin{aligned}
&\ga_\Phi = \{ \xi \in \ga \mid \psi(\xi) = 0 \text{ for all } \psi \in
	\Phi \} = \Phi^\perp\,, \\
&\gm_\Phi + \ga_\Phi \text{ is the centralizer of } \ga_\Phi \text{ in }
	\gg \text{, so } \gm_\Phi \text{ has root system } \Phi^{red},
	\text{ and }\\
&\gn_\Phi = {\sum}_{\alpha \in \Phi^{nil}} \gg_\alpha\,, \text{ nilradical
	of } \gq_\Phi\,, \text{ sum of the positive } \ga_\Phi\text{-root
	spaces.}
\end{aligned}
\end{equation}
Since $\gn = \sum_r \gl_r$, as given in (\ref{def-l}) we have
\begin{equation}\label{intersec-1}
\gn_{\Phi} = {\sum}_r (\gn_\Phi \cap \gl_r) = 
  {\sum}_r \left ( (\gg_{\beta_r} \cap \gn_\Phi) 
	+ {\sum}_{\Delta_r^+} (\gg_\alpha \cap \gn_\Phi) \right ).
\end{equation}
As $\ad(\gm)$ is irreducible on each restricted root space, 
if $\alpha \in \{\beta_r\} \cup \Delta_r^+$ then
$\gg_\alpha \cap \gn_\Phi$ is $0$ or all of $\gg_\alpha$\,.
\medskip

\begin{lemma}\label{inter-center}
Suppose $\gg_{\beta_r} \cap \gn_\Phi = 0$. 
Then $\gl_r \cap \gn_\Phi = 0$.
\end{lemma}
\begin{proof} Since $\gg_{\beta_r} \cap \gn_\Phi = 0$, the root $\beta_r$
has form $\sum_{\psi \in \Phi} n_\psi\psi$ with each $n_\psi \geqq 0$ and
$n_\psi = 0$ for $\psi \notin \Phi$.  If $\alpha \in \Delta_r^+$ it has form 
$\sum_{\psi \in \Psi} \ell_\psi\psi$ with $0 \leqq \ell_\psi \leqq n_\psi$ for
each $\psi \in \Psi$.  In particular $\ell_\psi = 0$ for 
$\psi \notin \Phi$.  Now every root space of $\gl_r$ is contained in
$\gm_\Psi$\,.  In particular $\gl_r \cap \gn_\Phi = 0$.  \hfill
\end{proof}

\begin{remark}{\rm We can define a partial order on $\{\beta_i\}$ by:
$\beta_{i+1} \succ \beta_i$ when the set of positive roots of which
$\beta_{i+1}$ is a maximum is contained in the corresponding set for 
$\beta_i$\,.  This is only a consideration when one further disconnects the 
Dynkin diagram by deleting a node at which $-\beta_i$ attaches, which
doesn't happen for type $A$.  If $\beta_s \succ \beta_r$ in this partial
order, and $\gg_{\beta_r} \cap \gn_\Phi = 0$, then 
$\gg_{\beta_s} \cap \gn_\Phi = 0$ as well, so $\gl_s \cap \gn_\Phi = 0$.}
\hfill $\diamondsuit$
\end{remark}

\begin{lemma}\label{inter-compl}
Suppose $\gg_{\beta_r} \cap \gn_\Phi \ne 0$.  Define $J_r \subset \Delta_r^+$
by $\gl_r \cap \gn_\Phi = \gg_{\beta_r} + \sum_{J_r} \gg_\alpha$\,.  Decompose
$J_r = J'_r \cup J''_r$ $($disjoint$)$ where 
  $J'_r = \{\alpha \in J_r \mid \sigma_r\alpha \in J_r\}$ and
  $J''_r = \{\alpha \in J_r \mid \sigma_r\alpha \notin J_r\}$.
Then $\gg_{\beta_r} + \sum_{J''_r} \gg_\alpha$ belongs to a single 
	$\ga_\Phi$-root space in $\gn_\Phi$\,,  i.e.
	$\alpha|_{\ga_\Phi} = \beta_r|_{\ga_\Phi}$\,, for every 
	$\alpha \in J''_r$\,.
\end{lemma}
\begin{proof}  Two restricted roots $\alpha = \sum_{\Psi} n_i\psi_i$
and $\alpha' = \sum_{\Psi} \ell_i\psi_i$ have the same restriction to
$\ga_\Phi$ if and only if $n_i = \ell_i$ for all $\psi_i \notin \Phi$.
Now suppose $\alpha \in J''_r$ and $\alpha' = \sigma_r\alpha$\,.  
Then $n_i > 0$ for some $\psi_i \notin \Phi$
but $\ell_i = 0$ for all $\psi_i \notin \Phi$.  Thus $\alpha$ and
$\beta_r = \alpha + \sigma_r\alpha$ have the same $\psi_i$-coefficient 
$n_i = n_i + \ell_i$ for
every $\psi_i \notin \Phi$.  In other words the corresponding restricted
root spaces are contained in the same $\ga_\Phi$-root space. \hfill
\end{proof}

\begin{lemma} \label{semidirect}
Suppose $\gl_r \cap \gn_\Phi \ne 0$.  Then the algebra
$\gl_r \cap \gn_\Phi$ has center 
$\gg_{\beta_r} + \sum_{J''_r} \gg_\alpha$\,, and
$\gl_r \cap \gn_\Phi = (\gg_{\beta_r} + \sum_{J''_r} \gg_\alpha)
+ (\sum_{J'_r} \gg_\alpha))$.
Further, $\gl_r\cap \gn_\Phi
= \left ({\sum}_{J''_r} \gg_\alpha\right ) \oplus
\left ( \gg_{\beta_r} + \left ({\sum}_{J'_r} \gg_\alpha \right )\right )$
direct sum of ideals.
\end{lemma}
\begin{proof}
This is immediate from the statements and proofs of 
Lemmas \ref{inter-center} and \ref{inter-compl}.  \hfill
\end{proof}

Following the cascade construction (\ref{cascade}) it will be 
convenient to define sets of simple restricted roots
\begin{equation}\label{cascade-simple}
\Psi_1 = \Psi \text{ and } \Psi_{s+1} = 
        \{\psi \in \Psi \mid  \langle \psi,\beta_i\rangle = 0 
        \text{ for } 1 \leqq i \leqq s\}.
\end{equation}
Note that $\Psi_r$ is the simple root system for
$\{\alpha \in \Delta^+(\gg,\ga) \mid \alpha \perp \beta_i 
\text{ for } i < r\}$.
\medskip

\begin{lemma}\label{part-c}
If $r > s$ then 
$[\gl_r \cap \gn_\Phi\,,\, \gg_{\beta_s} + {\sum}_{J''_s} \gg_\alpha] = 0$.
\end{lemma}

\begin{proof} Suppose that $\alpha \in J''_s$\,.  Express
$\alpha$ and $\sigma_s\alpha$ as sums of simple roots, say
$\alpha  = \sum n_i\psi_i$ and $\sigma_s\alpha = \sum \ell_i\psi_i$\,.
Then, $\ell_i = 0$ for all $\psi_i \in \Psi_s \cap \Phi^{nil}$
and $\beta_s = \sum(n_i+\ell_i)\psi_i$\,.  In other words the coefficient
of $\psi_i$ is the same for $\alpha$ and $\beta_s$ whenever
$\psi_i \in \Psi_s \cap \Phi^{nil}$.  Now let
$\gamma \in (\{\beta_r\} \cup \Delta_r^+) \cap \Phi^{nil}$
where $r > s$, and express $\gamma = \sum c_i\psi_i$\,.  Then
$c_{i_0} > 0$ for some $\beta_{i_0} \in (\Psi_r \cap \Phi^{nil})$.  Note
$\Psi_r \subset \Psi_s$\,, so $c_{i_0} > 0$ for some
$\beta_{i_0} \in (\Psi_s \cap \Phi^{nil})$\,.  Also, $[\gl_r,\gl_s] \subset
\gl_s$ because $r > s$.  If $\gamma + \alpha$ is a root then its
$\psi_{i_0}$-coefficient is greater than that of $\beta_s$\,, which is
impossible.  Thus $\gamma + \alpha$ is not a root.  The lemma follows.
\hfill
\end{proof}
\medskip

We look at a particular sort of linear functional on 
$\sum_r \bigl ( \gg_{\beta_s} + {\sum}_{J''_s} \gg_\alpha \bigr )$.
 Choose
$\lambda_r \in \gg_{\beta_r}^*$ such that $b_{\lambda_r}$ is nondegenerate
on $\sum_r \sum_{J'_r} \gg_\alpha$\,.  
Set $\lambda = \sum \lambda_r$\,.  
We know that (\ref{setup}(c)) holds for the nilradical of the
minimal parabolic $\gq$ that contains $\gq_\Phi$\,. In view of Lemma
\ref{part-c} it follows that
$b_\lambda(\gl_r,\gl_s) = \lambda([\gl_r,\gl_s] = 0$ for $r > s$.
For this particular type of $\lambda$, the bilinear form $b_\lambda$
has kernel $\sum_r \bigl ( \gg_{\beta_s} + 
{\sum}_{J''_s} \gg_\alpha \bigr )$ and is nondegenerate on 
$\sum_r \sum_{J'_r}\gg_\alpha$\,.
\medskip

At this point, the decomposition
$N_\Phi = (L_1\cap N_\Phi)(L_2\cap N_\Phi)\dots (L_m\cap N_\Phi)$
satisfies the first two conditions of (\ref{setup}): 
$$
\begin{aligned}
 &\text{{\rm (a)} each factor $L_r\cap N_\Phi$ has unitary representations 
with coefficients in $L^2((L_r\cap N_\Phi)/(center))$, and} \\
 &\text{{\rm (b)} each } N_r\cap N_\Phi := 
	(L_1\cap N_\Phi)\dots (L_r\cap N_\Phi)
        \text{ is a normal subgroup of } N_\Phi\\
	&\phantom{XXXXXXXXXXXXXX}
   \text{ with } N_r\cap N_\Phi = (N_{r-1}\cap N_\Phi)\rtimes (L_r\cap N_\Phi)
	 \text{ semidirect.}
\end{aligned}
$$
With Lemma \ref{part-c} this is enough to carry out 
Construction \ref{construction}
of our representations $\pi_\lambda$ of $N_\Phi$\,.  However it is not
enough for (\ref{setup}(c)) and (\ref{c-weak}).  For that we will group
the $L_r\cap N_\Phi$ in such a way that (\ref{c-weak}) is immediate and 
(\ref{setup}(c)) follows from Lemma \ref{part-c}.
This will be done in the next section.

\section{Extension to Arbitrary Parabolic Nilradicals}
\label{sec4}
\setcounter{equation}{0}
In this section we address (\ref{setup}(c)) and (\ref{c-weak}), 
completing the proof that
$N_\Phi$ has a decomposition that leads to stepwise square integrable
representations.  
\medskip

We start with some combinatorics.  Denote sets of indices as follows.
$q_1$ is the first index of (\ref{setup}) (usually $1$) such that 
$\beta_{q_1}|_{\ga_\Phi} \ne 0$;  define 
$$
I_1 = \{i \mid \beta_i|_{\ga_\Phi} = \beta_{q_1}|_{\ga_\Phi}\}.
$$
Then $q_2$ is the first index of (\ref{setup}) such that
$q_2 \notin I_1$ and $\beta_{q_2}|_{\ga_\Phi} \ne 0$; define
$$
I_2 = \{i \mid \beta_i|_{\ga_\Phi} = \beta_{q_2}|_{\ga_\Phi}\}.
$$
Continuing, $q_k$ is the first index of (\ref{setup}) such that
$q_k \notin (I_1\cup\dots\cup I_{k -1})$ and $\beta_{q_k}|_{\ga_\Phi} \ne 0$;
define
$$
I_k = \{i \mid \beta_i|_{\ga_\Phi} = \beta_{q_k}|_{\ga_\Phi}\}
$$
as long as possible.  Write $\ell$ for the last index $k$ that leads to
a nonempty set $I_k$\,.  Then, in terms of the index set of
(\ref{setup}), $I_1 \cup \dots \cup I_\ell$ consists of all the
indices $i$ for which $\beta_i|_{\ga_\Phi} \ne 0$.
\medskip

For $1 \leqq j \leqq \ell$ define
\begin{equation}\label{big-summands}
\gl_{\Phi,j} = {\sum}_{i \in I_j} (\gl_i\cap \gn_\Phi) =
	\Bigl ( {\sum}_{i \in I_j} \gl_i\Bigr ) \cap \gn_\Phi
	\text{ and } \gl^{\dagger}_{\Phi,j} = 
	{\sum}_{k \geqq j} \gl_{\Phi,k}\,.
\end{equation}

\begin{lemma}\label{some-brackets}
If $k \geqq j$ then $[\gl_{\Phi,k} , \gl_{\Phi,j}] \subset \gl_{\Phi,j}$\,.
For each index $j$,
$\gl_{\Phi,j}$ and $\gl^{\dagger}_{\Phi,j}$ are subalgebras of $\gn_\Phi$ and
$\gl_{\Phi,j}$ is an ideal in $\gl^{\dagger}_{\Phi,j}$\,.
\end{lemma}

\begin{proof}
As we run along the sequence $\{\beta_1,\beta_2,\dots\}$ the coefficients
of the simple roots are weakly decreasing, so in particular the coefficients
of the roots in $\Psi \setminus \Phi$ are weakly decreasing.  If
$r \in I_k$, $s \in I_j$ and $k > j$ now $r > s$.  Using
$[\gl_r,\gl_s] \subset \gl_s$ (and thus
$[(\gl_r \cap \gn_\Phi),(\gl_s \cap \gn_\Phi)] \subset 
\gl_s \cap \gn_\Phi$) for $r > s$ it follows that
$[\gl_{\Phi,k} , \gl_{\Phi,j}] \subset \gl_{\Phi,j}$ for $k > j$.
\medskip

Now suppose $k = j$.  If $r = s$ then $[\gl_r,\gl_r] = \gg_{\beta_r}$, so
we may assume $r > s$, and thus 
$[\gl_r,\gl_s] \subset \gl_s \subset \gl_{\Phi,j}$.  It follows that
$[\gl_{\Phi,k} , \gl_{\Phi,j}] \subset \gl_{\Phi,j}$ for $k = j$.
\medskip

Now it is immediate that 
$\gl_{\Phi,j}$ and $\gl^{\dagger}_{\Phi,j}$ are subalgebras of $\gn_\Phi$ and
$\gl_{\Phi,j}$ is an ideal in $\gl^{\dagger}_{\Phi,j}$\,.
\end{proof}

\begin{lemma}\label{not-beta}
If $k > j$ then 
$[\gl_{\Phi,k} \,, \gl_{\Phi,j}] \cap \sum_{i \in I_j}\gg_{\beta_i} = 0$.
\end{lemma}

\begin{proof} This is implicit in Theorem \ref{plancherel-general}, which 
gives (\ref{c-weak}), but we give a direct 
proof for the convenience of the reader.
Let $\gg_\gamma \subset \gl_{\Phi,k}$ and $\gg_\alpha \subset \gl_j$
with $[\gg_\gamma , \gg_\alpha] \cap \sum_{i \in I_j}\gg_{\beta_i} \ne 0$.
Then $[\gg_\gamma , \gg_\alpha] = \gg_{\beta_i}$ where 
$\gg_\gamma \subset \gl_r$ and $\gg_\alpha \subset \gl_i$\,, so
$\gg_\gamma = \gg_{\beta_i - \alpha} \subset \gl_r \cap \gl_i = 0$.
That contradiction proves the lemma.
\end{proof}

Given $r \in I_j$ we use the notation of Lemma \ref{inter-compl}
to decompose
\begin{equation}\label{split-lr}
\gl_r \cap \gn_\Phi = \gl'_r + \gl''_r \text{ where }
	\gl'_r = \gg_{\beta_r} + {\sum}_{J'_r}\gg_\alpha \text{ and }
	\gl''_r = {\sum}_{J''_r}\gg_\alpha\,.
\end{equation}
Here $J'_r$ consists of roots $\alpha \in \Delta_r^+$
such that $\gg_\alpha + \gg_{\beta_r - \alpha}
\subset \gn_\Phi$\,, and $J''_r$ consists of roots $\alpha \in \Delta_r^+$
such that $\gg_\alpha \subset \gn_\Phi$ but $\gg_{\beta_r - \alpha}
\not\subset \gn_\Phi$\,.
For $1 \leqq j \leqq \ell$ define 
\begin{equation}\label{split-lphi}
\gz_{\Phi,j} = {\sum}_{i \in I_j} (\gg_{\beta_i} +  \gl''_i)
\end{equation} 
and decompose
\begin{equation}\label{big-split}
\gl_{\Phi,j} = \gl'_{\Phi,j} + \gl''_{\Phi,j} \text{ where }
	\gl'_{\Phi,j} = {\sum}_{i \in I_j} \gl'_i \text{ and }
	\gl''_{\Phi,j} = {\sum}_{i \in I_j} \gl''_i\,.
\end{equation}

\begin{lemma}\label{central-ideal}
Recall $\gl^{\dagger}_{\Phi,j} = {\sum}_{k \geqq j} \gl_{\Phi,k}$ from 
{\rm (\ref{big-summands})}.
For each $j$, both $\gz_{\Phi,j}$ and $\gl''_{\Phi,j}$ are 
central ideals in $\gl^{\dagger}_{\Phi,j}$\,, and $\gz_{\Phi,j}$ is the center
of $\gl_{\Phi,j}$.
\end{lemma}

\begin{proof}
Lemma \ref{inter-compl} shows that $\alpha|_{\ga_\Phi} = \beta_i|_{\ga_\Phi}$
whenever $i \in I_j$ and $\gg_\alpha \subset \gl''_{\Phi,j}$\,.
If $[\gl_{\Phi,k},\gl''_i] \ne 0$ it contains some $\gg_{\delta}$ 
such that $\gg_{\delta} \subset \gl_{\Phi,j}$ and 
at least one of the coefficients of $\delta$ along roots of 
$\Psi\setminus\Phi$ is greater than that of $\beta_i$.  As
$\gg_\delta \subset \gl_i$ that is impossible.  
Thus $\gl''_{\Phi,j}$ is a
central ideal in $\gl^{\dagger}_{\Phi,j}$\,.   The same is immediate for
$\gz_{\Phi,j} = {\sum}_{i \in I_j} (\gg_{\beta_i} +  \gl''_i)$\,. 
In particular $\gz_{\Phi,j}$ is central in $\gl_{\Phi,j}$.  But the
center of $\gl_{\Phi,j}$ can't be any larger, by definition of
$\gl'_{\Phi,j}$\,.
\end{proof}

Decompose 
\begin{equation}\label{nzv}
\gn_\Phi = \gz_\Phi + \gv_\Phi \text{ where }
\gz_\Phi = \sum_j \gz_{\Phi,j}\,\,,\,\, 
\gv_\Phi = \sum_j \gv_{\Phi,j} \text{ and }
\gv_{\Phi,j} = \sum_{i \in I_j} \sum_{\alpha \in J'_i} \gg_\alpha\,.
\end{equation}
Then Lemma \ref{central-ideal} gives us (\ref{c-weak}) for the $\gl_{\Phi,j}$:
$\gl_{\Phi,j} =  \gl'_{\Phi,j} \oplus \gl''_{\Phi,j}$ with 
$\gl''_{\Phi,j} \subset \gz_{\Phi,j}$ and $\gv_{\Phi,j} \subset \gl'_{\Phi,j}$.

\begin{lemma} \label{stepwise-nondegen}
For generic $\lambda_j \in \gz_{\Phi,j}^*$ the kernel of $b_{\lambda_j}$ on 
$\gl_{\Phi,j}$ is just $\gz_{\Phi,j}$, in other words $b_{\lambda_j}$ is
is nondegenerate on $\gv_{\Phi,j} \simeq \gl_{\Phi,j}/\gz_{\Phi,j}$.
In particular $L_{\Phi,j}$ has square integrable representations.
\end{lemma}

\begin{proof}  From the definition of $\gl'_{\Phi,j}$\,, the bilinear 
form $b_{\lambda_j}$ on $\gl_{\Phi,j}$
annihilates the center $\gz_{\Phi,j}$ and is nondegenerate on
$\gv_{\Phi,j}$\,.  Thus the corresponding representation
$\pi_{\lambda_j}$ of $L_{\Phi,j}$ has coefficients that are
square integrable modulo its center.
\end{proof}

Now we come to our first main result:

\begin{theorem}\label{gen-setup}
Let $G$ be a real reductive Lie group and $Q$ a real parabolic subgroup.
Express $Q = Q_\Phi$ in the notation of {\rm (\ref{para-roots})} and
{\rm (\ref{para-pieces})}.  Then its nilradical $N_\Phi$ has
decomposition $N_\Phi = L_{\Phi,1}L_{\Phi,2}\dots L_{\Phi,\ell}$
that satisfies the conditions of {\rm (\ref{setup})} and {\rm (\ref{c-weak})}
as follows.
The center $Z_{\Phi,j}$ of $L_{\Phi,j}$ is the analytic subgroup
for $\gz_{\Phi,j}$ and
\begin{equation}\label{gen-setup-list}
\begin{aligned}
 &\text{{\rm (a)} each factor $L_{\Phi,j}$ has unitary representations
	with coefficients in $L^2(L_{\Phi,j}/Z_{\Phi,j})$, and} \\
 &\text{{\rm (b)} each } N_{\Phi,j} :=
        L_{\Phi,1}L_{\Phi,2}\dots L_{\Phi,j}
        \text{ is a normal subgroup of } N_\Phi\\
        &\phantom{XXXXXXXXXXXXXX}
   	\text{ with } N_{\Phi,j} = N_{\Phi,j-1}\rtimes L_{\Phi,j}
         \text{ semidirect,}\\
 &\text{{\rm (c)} } [\gl_{\Phi,k},\gz_{\Phi,j}] = 0 \text{ and }
	[\gl_{\Phi,k}, \gl_{\Phi,j}] \subset \gv_{\Phi,j} + \gl_{\Phi,j}'' 
	\text{ for } k > j.
\end{aligned}
\end{equation}
In particular $N_\Phi$ has stepwise square integrable representations
relative to the decomposition 
$N_\Phi = L_{\Phi,1}L_{\Phi,2}\dots L_{\Phi,\ell}$\,, and the results of
{\rm Theorem \ref{plancherel-general}}, specifically {\rm (\ref{frob1}),
(\ref{frob2})} and {\rm (\ref{frob3})}, hold for $N_\Phi$\,.
\end{theorem}

\begin{proof} Statement (a) is the content of Lemma \ref{stepwise-nondegen},
and statement (b) follows from Lemma \ref{some-brackets}.  The first part
of (c), $[\gl_{\Phi,k},\gz_{\Phi,j}] = 0$ for $k > j$, is contained in 
Lemma \ref{central-ideal}.  The second part,
$[\gl_{\Phi,k}, \gl_{\Phi,j}] \subset \gv_\Phi + \gl_{\Phi,j}''$ for $k > j$,
follows from Lemma \ref{not-beta}.  Now Theorem \ref{plancherel-general}
applies.
\end{proof}

\section{The Maximal Exponential-Solvable Subgroup $A_\Phi N_\Phi$}
\label{sec5}
\setcounter{equation}{0}

In this section we extend the considerations of \cite[\S 4]{W2014} 
from minimal parabolics to the exponential-solvable subgroups $A_\Phi N_\Phi$
of real parabolics $Q_\Phi = M_\Phi A_\Phi N_\Phi$.  It turns out that 
the of Plancherel and Fourier inversion formulae of $N_\Phi$ go through,
with only small changes, to the non-unimodular solvable group 
$A_\Phi N_\Phi$\,.  We follow the development in \cite[\S 4]{W2014}.
\medskip

Let $H$ be a separable locally compact group of type I.  Then
\cite[\S1]{LW1978} the Fourier inversion formula for $H$ has form
\begin{equation}\label{nonunimod-planch-gen}
f(x) = \int_{\widehat{H}} \tr \pi(D(r_xf))d\mu_H(\pi)
\end{equation}
where $D$ is an invertible positive self adjoint operator on $L^2(H)$,
conjugation semi-invariant of weight equal to that of the modular function
$\delta_H$\,, $r_xf$ is the right translate $y \mapsto f(yx)$, 
and $\mu$ is a positive Borel measure on the unitary dual
$\widehat{H}$.  When $H$ is unimodular, $D$ is the identity and 
(\ref{nonunimod-planch-gen}) reduces to the usual Fourier inversion formula
for $H$.  In general the semi-invariance of $D$ compensates any
lack of unimodularity.  See \cite[\S1]{LW1978} for a detailed discussion
including a discussion of the domains of $D$ and $D^{1/2}$.  Here 
$D\otimes{\mu}$ is unique up to normalization of Haar measure, 
but $(D,\mu)$ is not unique, except of course when
we fix one of them, such as in the unimodular case when we take $D = 1$.
Given such a pair $(D,\mu)$ we refer to $D$ as a 
{\em Dixmier-Puk\'anszky operator} and to $\mu$ as the associated
Plancherel measure.
\medskip

One goal of this section is to describe a ``best'' choice of the
Dixmier-Puk\'anszky operator for $A_\Phi N_\Phi$ in terms of the 
decomposition $N_\Phi = L_{\Phi,1}L_{\Phi,2}\dots L_{\Phi,\ell}$ that
gives stepwise square integrable representations of $N_\Phi$\,.
\medskip

Let $\delta = \delta_{Q_\Phi}$ denote the modular function of $Q_\Phi$\,.  
Its kernel
contains $M_\Phi N_\Phi$ because $\Ad(M_\Phi)$ is reductive with compact
center and $\Ad(N_\Phi)$ is unipotent. Thus
$\delta(man) = \delta(a)$, and if $\xi \in \ga_\Phi$ then
$\delta(\exp(\xi)) = \exp(\tr(\ad(\xi)))$.  Note that $\delta$ also is the 
modular function for $A_\Phi N_\Phi$\,.
\medskip

\begin{lemma}\label{traces}  Let $\xi \in \ga_\Phi$\,.  Then each
$\dim \gl_{\Phi,j} + \dim \gz_{\Phi,j}$ is even, and

\noindent\phantom{ii}{\rm (i)} the trace of $\ad(\xi)$ on $\gl_{\Phi,j}$
	is $\frac{1}{2}\dim(\gl_{\Phi,j} + \dim \gz_{\Phi,j})\beta_{j_0}(\xi)$
	for any $j_0 \in I_j$,

\noindent\phantom{i}{\rm (ii)} the trace of $\ad(\xi)$ on $\gn_\Phi$,
	on $\ga_\Phi + \gn_\Phi$ and on $\gq_\Phi$ is
	$\frac{1}{2}\sum_j(\dim \gl_{\Phi,j} + 	
	\dim \gz_{\Phi,j})\beta_{j_0}(\xi)$, \text{ and}

\noindent {\rm (iii)} the determinant of $\Ad(\exp(\xi))$ on $\gn_\Phi$\,,
        on $\ga_\Phi + \gn_\Phi$\,, and on $\gq_\Phi$\,, is
	$\prod_j \exp(\beta_{j_0}(\xi))^{\frac{1}{2}
	(\dim \gl_{\Phi,j} + \dim \gz_{\Phi,j})}$.
\end{lemma}

\begin{proof}
We use the notation of (\ref{split-lr}), (\ref{split-lphi}) and 
(\ref{big-split}).  It is immediate that 
$\dim \gl_r + \dim (\gg_{\beta_r} + \gl_r'')$ is even.  Sum over $r \in I_j$
to see that $\dim \gl_{\Phi,j} + \dim \gz_{\Phi,j}$ is even. 
\medskip

The trace of $\ad(\xi)$ on $\gl_r \cap \gn_\Phi$ is
$(\dim \gg_{\beta_r})\beta_r(\xi)$ on $\gg_{\beta_r}$, 
plus $\frac{1}{2}\sum_{\alpha \in J'_r}(\dim \gg_\alpha)\beta_r(\xi)$ 
(for the pairs $\gg_\alpha, \gg_\alpha' \in \Delta_r^+\cap \Phi^{nil}$
that pair into $\gg_{\beta_r}$), plus
$\sum_{\alpha \in J''_r}(\dim \gg_\alpha)\beta_r(\xi)$ (since 
$\alpha \in J''_r$ implies 
$\alpha|_{\ga_\Phi} = \beta_r|_{\ga_\Phi}$).  Now the trace of
$\ad(\xi)$ on $\gl_r \cap \gn_\Phi$ is \hfill\newline
\phantom{XXXXX}$(\frac{1}{2}\dim \gg_{\beta_r} + 
\frac{1}{2}\dim \gl'_r + \dim \gl_r'') \beta_r(\xi)$
= $\frac{1}{2}\dim(\gl_r \cap \gn_\Phi) + \dim(\gg_{\beta_r} + \gl''_r)
\beta_r(\xi)$.\hfill\newline
Summing over $r \in I_j$ we arrive at assertion (i).
Then sum over $j$ for (ii) and exponentiate for (iii).
\end{proof}

We reformulate Lemma \ref{traces} as
\begin{lemma}\label{modfunction}
The modular function $\delta = 
\delta_{Q_\Phi}$ of $Q_\Phi = M_\Phi A_\Phi N_\Phi$ is 
$$
\delta(man) = {\prod}_j \exp(\beta_{j_0}(\log a))^{\frac{1}{2}
        (\dim \gl_{\Phi,j} + \dim \gz_{\Phi,j})}\,.
$$
The modular function $\delta_{A_\Phi N_\Phi} = \delta|_{A_\Phi N_\Phi}$\,,
and $\delta_{U_\Phi A_\Phi N_\Phi} = \delta|_{U_\Phi A_\Phi N_\Phi}$ where
$U_\Phi$ is a maximal compact subgroup of $M_\Phi$\,.
\end{lemma}

Consider semi-invariance of the polynomial $P$ of (\ref{c-d}(d)), which
by definition is the product of factors $\Pf_{\gl_{\Phi,j}}$\,.  
Using
(\ref{nzv}) and Lemma \ref{stepwise-nondegen}, calculate with bases of
the $\gv_{\Phi,j}$ as in \cite[Lemma 4.4]{W2014} to arrive at

\begin{lemma}\label{invar-P}
Let $\xi \in \ga_\Phi$ and $a = \exp(\xi) \in A_\Phi\,$.  Then
$\ad(\xi)P = \left (\frac{1}{2}\sum_j 
\dim(\gl_{\Phi,j}/\gz_{\Phi,j})\beta_{j_0}(\xi)\right ) P$ 
and $\Ad(a)P = \left ( \prod_j(\exp (\beta_{j_0}(\xi)))^{\frac{1}{2}
\sum_j \dim(\gl_{\Phi,j}/\gz_{\Phi,j})}\right ) P$.
\end{lemma}

\begin{definition}\label{quasicenter}{\rm
The {\em quasi-center} of $\gn_\Phi$ is $\gs_\Phi = \sum_j\gz_{\Phi,j}$\,.
Fix a basis $\{e_t\}$ of $\gs_\Phi$ consisting of ordinary root vectors,
$e_t \in \gg_{\alpha_t}$.  The {\em quasi-center determinant} relative to
the choice of $\{e_t\}$ is the polynomial function
$\Det_{\gs_\Phi}(\lambda)$ = $\prod_t \lambda(e_t)$
on $\gs^*_\Phi$\,.}\hfill$\diamondsuit$
\end{definition}

Let $a \in A_\Phi$ and compute $(\Ad(a)\Det_{\gs_\Phi})(\lambda)$
= $\Det_{\gs_\Phi}(\Ad^*($$a$$)^{-1}\lambda)$
= $\prod_t \lambda(\Ad(a)e_t)$.  Each $e_t \in \gz_{\Phi,j}$ is multiplied
by $\exp(\beta_{j_0}(\log a))$.  So $(\Ad(a)\Det_{\gs_\Phi})(\lambda)$
= $\bigl (\prod_j \exp(\beta_{j_0}(\log a))^{\dim \gz_{\Phi,j}} \bigr )
\Det_{\gs_\Phi}(\lambda)$.  Now

\begin{lemma}\label{pf-d}
If $\xi \in \ga_\Phi$ then $\Ad(\exp(\xi))\Det_{\gs_\Phi}$ = 
$\Bigl (\prod_j \exp(\beta_{j_0}(\xi))^{\dim\gz_{\Phi,j}}\Bigr )
\Det_{\gs_\Phi}$ where $j_0 \in I_j$\,.
\end{lemma}

Combining Lemmas \ref{traces}, \ref{modfunction} and \ref{pf-d} we have

\begin{proposition}\label{pf-mod}
The product $P\cdot\Det_{\gs_\Phi}$ is an $\Ad(Q_\Phi)$-semi-invariant
polynomial on $\gs^*_\Phi$ of degree 
$\frac{1}{2}(\dim \gn_\Phi + \dim \gs_\Phi)$ and of weight equal to the 
weight of the modular function $\delta_{Q_\Phi}$\,.
\end{proposition}

Denote $V_\Phi = \exp(\gv_\Phi)$ and $S_\Phi = \exp(\gs_\Phi)$.  Then
$V_\Phi \times S_\Phi \to N_\Phi$\,, by $(v,s) \mapsto vs$, is an
analytic diffeomorphism.  Define

\begin{equation}\label{def-dp}
D_0: \text{ Fourier transform of } P\cdot\Det_{\gs_\Phi}
\text{ acting on $A_\Phi N_\Phi = A_\Phi V_\Phi S_\Phi$
 by acting on the $S_\Phi$ variable.}
\end{equation}

\begin{theorem}\label{dix-for-an}
The operator $D_0$ of {\rm (\ref{def-dp})} is an invertible self-adjoint
differential operator of degree $\frac{1}{2}(\dim\gn_\Phi + \dim\gs_\Phi)$
on $L^2(A_\Phi N_\Phi)$ with dense domain the Schwartz space
$\cC(A_\Phi N_\Phi)$, and $D := D_0^{1,2} (D_0^{1/2})^*$ is a well defined
invertible positive self-adjoint operator of the same
degree $\frac{1}{2}(\dim\gn_\Phi + \dim\gs_\Phi)$
on $L^2(A_\Phi N_\Phi)$ with dense domain $\cC(A_\Phi N_\Phi)$.
In particular $D$
is a  Dixmier-Puk\'anszky operator on $A_\Phi N_\Phi$ with domain equal
to the space of rapidly decreasing $C^\infty$ functions.
\end{theorem}

\begin{proof} Since it is the Fourier transform of a real polynomial, $D_0$
is a differential operator that is self-adjoint on $L^2(A_\Phi N_\Phi)$
with dense domain $\cC(A_\Phi N_\Phi)$.  Thus $D$ is well defined, and
is positive and self-adjoint as asserted.  Now it remains only
to see that $D$ (and thus $D_0$) are invertible.
\smallskip

Invertibility of $D$ comes out of Dixmier's theory of quasi-Hilbert
algebras \cite{D1952} as applied by Kleppner and Lipsman to group extensions.
Specifically, \cite[\S 6]{KL1972} leads to a
Dixmier-Puk\'anszky operator, there called $M$.  The quasi-Hilbert algebra in
question is defined on \cite[pp. 481--482]{KL1972}, the relevant transformations
$M$ and $\Upsilon$ are specified in \cite[Theorem 1]{P1955}, and invertibility
of $M$ is shown in \cite[pp. 293--294]{D1952}.  Unwinding the definitions of
$M$ and $\Upsilon$ in \cite[\S 6]{KL1972} one sees that the
Dixmier-Puk\'anszky operator $M$ of \cite{KL1972} is the same as our
operator $D$.  That completes the proof.
\end{proof}

The action of $\ga_\Phi$ on $\gz_{\Phi,j}$ is scalar,
$\ad(\alpha)\zeta = \beta_{j_0}(\alpha)\zeta$ where (as before)
$j_0 \in I_j$\,.  So the isotropy
algebra $(\ga_\Phi)_\lambda$ is the same at every $\lambda \in \gt_\Phi^*$\,,
given by $(\ga_\Phi)_\lambda = \{\alpha \in \ga_\Phi \mid \text{ every }
\beta_{j_0}(\alpha) = 0\}$.  Thus the $(A_\Phi)$-stabilizer on $\gt_\Phi^*$
is
\begin{equation}\label{a-stab}
A'_\Phi := \{\exp(\alpha) \mid \text{ every } 
	\beta_{j_0}(\alpha) = 0\}, \text{ independent of choice of }
	\lambda \in \gt_\Phi^*\,.
\end{equation}

Given $\lambda \in \gt_\Phi^*$, in other words given a stepwise square
integrable representation $\pi_\lambda$ where $\lambda \in \gs_\Phi^*$\,,
we write $\pi_\lambda^\dagger$ for the extension of $\pi_\lambda$ to a
representation of $A'_\Phi N_\Phi$ on the same Hilbert space.  That
extension exists because $A'_\Phi$ is a vector group, thus contractible to 
a point, so $H^2(A'_\Phi;\C') = H^2(point;\C') = \{1\}$, and
the Mackey obstruction vanishes.  Now the representations of $A_\Phi N_\Phi$
corresponding to $\pi_\lambda$ are the
\begin{equation}\label{rep-an-lambda}
\pi_{\lambda,\xi} := \Ind_{A'_\Phi N_\Phi}^{A_\Phi N_\Phi}
	(\exp(i \xi) \otimes \pi_\lambda^\dagger) \text{ where }
	\xi \in (\ga'_\Phi)^* \text{ and }
	\exp(i \xi): \exp(\alpha) := \exp(i\xi(\alpha)) \text{ for }
	\alpha \in \ga'_\Phi\,.
\end{equation}
Note also that
\begin{equation}\label{rep-an-conj}
\pi_{\lambda,\xi}\cdot \Ad(an) = \pi_{\Ad^*(a)\lambda,\xi}
	\text{ for } a \in A_\Phi \text{ and } n \in N_\Phi\, .
\end{equation} 
The resulting Plancherel
formula (\ref{nonunimod-planch-gen}), 
$f(x) = \int_{\widehat{H}} \tr \pi(D(r_xf))d\mu_H(\pi)$,
$H = A_\Phi N_\Phi$\,, is

\begin{theorem}\label{planch-an}
Let $Q_\Phi = M_\Phi A_\Phi N_\Phi$ be a parabolic subgroup of the 
real reductive Lie group $G$.  Let $D$ denote the Dixmier-Puk\'anszky
operator of {\rm (\ref{def-dp})}.  Let
$\pi_{\lambda,\xi} \in \widehat{A_\Phi N_\Phi}$ as
described in {\rm (\ref{rep-an-lambda})} and let
$\Theta_{\pi_{\lambda,\xi}}: h \mapsto
    \tr \pi_{\lambda,\phi}(h)$
denote its distribution character.  Then
$\Theta_{\pi_{\lambda,\xi}}$ is a tempered distribution.
If $f \in \cC(A_\Phi N_\Phi)$ and $x \in A_\Phi N_\Phi$ then
$$
f(x) = c \int_{(\ga'_\Phi)^*} 
	\left ( \int_{\gs_\Phi^*/\Ad^*(A_\Phi)}
        \Theta_{\pi_{\lambda,\xi}}(D(r_xf)) |\Pf(\lambda)| d\lambda
	\right ) d\phi
$$
where $c = 2^{d_1 + \dots + d_m} d_1! d_2! \dots d_m!$ as in {\rm (\ref{c-d}a)}
and $m$ is the number of factors $L_r$ in $N_n$\,.
\end{theorem}

\begin{proof}
We compute along the lines of the computation of 
\cite[Theorem 2.7]{LW1982} and \cite[Theorem 3.2]{KL1972}.
$$
\begin{aligned}
\tr &\pi_{\lambda,\phi}(Dh) \\
        &=\int_{x \in A_\Phi/ A'_\Phi}\delta(x)^{-1}
                \tr \int_{N_\Phi A'_\Phi}(Dh)(x^{-1}nax)\cdot
                (\pi_{\lambda}^\dagger\otimes \exp(i\phi))
                   (na)\,dn\, da\, dx\\
&= \int_{x \in A_\Phi/ A'_\Phi}
        \tr \int_{N_\Phi A'_\Phi}(Dh)(nx^{-1}ax)\cdot
                (\pi_{\lambda}^\dagger\otimes \exp(i\phi))
                (xnx^{-1}a)\,
                dn\, da\, dx.
\end{aligned}
$$
Now
\begin{equation}\label{one-orbit}
\begin{aligned}
\int_{(\ga'_\Phi)^*}
  &\tr \pi_{\lambda,\phi}(Dh) \,d\phi \\
&= \int_{\widehat{A'_\Phi}}
  \int_{x \in A_\Phi/ A'_\Phi}
        \tr \int_{N_\Phi A'_\Phi}(Dh)(nx^{-1}ax) 
                (\pi^\dagger_{\lambda}\otimes \exp(i\phi))
                (xnx^{-1}a)\,
                dn\,  da\, dx\, \,d\phi\\
&= \int_{x\in  A_\Phi/ A'_\Phi}
        \int_{\widehat{A'_\Phi}}
        \tr \int_{N_\Phi A'_\Phi}
        (Dh)(nx^{-1}ax) 
                (\pi^\dagger_{\lambda}\otimes \exp(i\phi))
                (xnx^{-1}a)\, dn\, da\, \,d\phi\,\, dx \\
&= \int_{x \in A_\Phi/ A'_\Phi} \tr \int_{N_\Phi} (Dh)(n)
        \pi^\dagger_{\lambda}(xnx^{-1}) dn\, dx \\
&= \int_{x \in A_\Phi/ A'_\Phi} \tr \int_{N_\Phi} (Dh)(n)
        (\Ad(x^{-1})\cdot \pi^\dagger_{\lambda})(n) dn\, dx \\
&= \int_{x \in A_\Phi/ A'_\Phi} \tr (\Ad(x^{-1})\cdot
        \pi^\dagger_{\lambda})(Dh))\, dx\\
&= \int_{x \in A_\Phi/ A'_\Phi}
        (\Ad(x^{-1})\cdot\pi^\dagger_{\lambda})_*(D)
        \,\,\tr (\Ad(x^{-1})\cdot\pi^\dagger_{\lambda})(h) dx\\
&= \int_{x \in A_\Phi/ A'_\Phi}(\pi^\dagger_{\lambda})_*(\Ad(x)\cdot D)
        \,\tr (\Ad(x^{-1})\cdot\pi^\dagger_{\lambda})(h)\, dx \\
&= \int_{x \in A_\Phi/ A'_\Phi} \delta_{A_\Phi N_\Phi} (x)
        \,\tr (\Ad(x^{-1})\cdot\pi^\dagger_{\lambda})(h)\, dx
= \int_{\lambda' \in \Ad^*(A_\Phi)\lambda} 
	\tr \pi^\dagger_{\lambda'}(h)|\Pf(\lambda')|d\lambda' .
\end{aligned}
\end{equation}
Summing over $\overline{\lambda} = \Ad^*(A_\Phi)(\lambda) \in 
\gt^*/\Ad^*(A_\Phi)$ we now have
\begin{equation}\label{all-orbits}
\begin{aligned}
\int_{\overline{\lambda} \in 
	\gt_\Phi^*/\Ad^*(A_\Phi)} & \left ( \int_{(\ga'_\Phi)^*}
        \tr \pi_{\lambda,\phi}(Dh) \,d\phi\right ) d\overline{\lambda} \\
&=   \int_{\overline{\lambda} \in \gt_\Phi^*/\Ad^*(A_\Phi)} \left (
	\int_{\lambda' \in \Ad^*(A_\Phi)\lambda} 
        \tr \pi^\dagger_{\lambda'}(h)|\Pf(\lambda')|d\lambda'\right ) 
	d\overline{\lambda} \\
& =  \int_{\lambda \in \gs_\Phi^*} \tr \pi_\lambda(h) |\Pf(\lambda)|d\lambda 
	= h(1).
\end{aligned}
\end{equation}
If $h = r_xf$ then $h(1) = f(x)$ and the theorem follows.
\end{proof}

\section{The Maximal Amenable Subgroup $U_\Phi A_\Phi N_\Phi$}
\label{sec6}
\setcounter{equation}{0}

In this section we extend our results on $N_\Phi$ and $A_\Phi N_\Phi$ to
the maximal amenable subgroups 
$$
E_\Phi := U_\Phi A_\Phi N_\Phi \text{ where }
	U_\Phi \text{ is a maximal compact subgroup of }M_\Phi\,.  
$$
Of course
if $\Phi = \emptyset$, i.e. if $Q_\Phi$ is a minimal parabolic, then 
$U_\Phi = M_\Phi$\,.
We start by recalling the classification of maximal amenable subgroups in
real reductive Lie groups.
\medskip

Recall the definition.  A {\em mean} on a locally compact group $H$ is a
linear functional $\mu$ on $L^\infty(H)$ of norm $1$ and such that
$\mu(f) \geqq 0$ for all real-valued $f \geqq 0$.  $H$ is {\em amenable} 
if it has a
left-invariant mean.  There are more than a dozen useful equivalent conditions.
Solvable groups and compact groups are amenable, as are extensions of
amenable groups by amenable subgroups.
In particular if $U_\Phi$ is a maximal compact
subgroup of $M_\Phi$ then $E_\Phi := U_\Phi A_\Phi N_\Phi$ is amenable.
\medskip

We'll need a technical condition \cite[p. 132]{M1979}.  Let $H$ be the group
of real points in a linear algebraic group whose rational points are Zariski 
dense, let $A$ be a maximal $\R$-split torus in $H$, let $Z_H(A)$ denote
the centralizer of $A$ in $H$, and let $H^0$ be the algebraic connected 
component of the identity in $H$.  Then $H$ is {\em isotropically connected}
if $H = H^0\cdot Z_H(A)$.  More generally we will say that a subgroup 
$H \subset G$ is {\em isotropically connected} if the algebraic hull of 
$\Ad_G(H)$ is isotropically connected.  The point is Moore's theorem

\begin{proposition}\label{moore}{\rm \cite[Theorem 3.2]{M1979}.}
The groups $E_\Phi := U_\Phi A_\Phi N_\Phi$ are maximal amenable subgroups
of $G$.  They are isotropically connected and self-normalizing.  
As $\Phi$ runs over the $2^{|\Psi|}$ subsets of $\Psi$ the $E_\Phi$ are
mutually non-conjugate.  An amenable subgroup $H \subset G$ is contained in
some $E_\Phi$ if and only if it is isotropically connected.
\end{proposition}

Now we need some notation and definitions. 
\begin{equation}\label{def-mult}
\text{ if } \alpha \in \Delta^+(\gg,\ga) \text{ then }
	[\alpha]_\Phi = \{\gamma \in \Delta^+(\gg,\ga)
	\mid \gamma|_{\ga_\Phi} = \alpha |_{\ga_\Phi}\} \text{ and } 
	\gg_{[\Phi, \alpha]} = {\sum}_{\gamma \in [\alpha]_\Phi} \gg_\gamma\,.
\end{equation}
Recall \cite[Theorem 8.3.13]{W1966} that the various $\gg_{[\Phi, \alpha]}$,
$\alpha \notin \Phi^{red}$, are $\ad(\gm_\Phi)$-invariant and are
absolutely irreducible as $\ad(\gm_\Phi)$-modules.
\begin{definition}\label{invariant}{\rm
The decomposition $N_\Phi = 
L_{\Phi,1}L_{\Phi,2}\dots L_{\Phi,\ell}$ of Theorem \ref{gen-setup} is
{\em invariant} if each \newline $\ad(\gm_\Phi)\gz_{\Phi,j} \subset \gz_{\Phi,j}$,
equivalently if each $\Ad(M_\Phi)\gz_{\Phi,j} = \gz_{\Phi,j}$,
in other words whenever $\gz_{\Phi,j} = \gg_{[\Phi, \beta_{j_0}]}$\,.
The decomposition $N_\Phi = L_{\Phi,1}L_{\Phi,2}\dots L_{\Phi,\ell}$
is {\em weakly invariant} if each $\Ad(U_\Phi)\gz_{\Phi,j} = \gz_{\Phi,j}$\,.
} \hfill $\diamondsuit$
\end{definition}

Here are some special cases.  \hfill\newline
\smallskip
\noindent (1) If $\Phi$ is empty, i.e. if $Q_\Phi$ is a minimal
parabolic, then the decomposition $N_\Phi = 
L_{\Phi,1}L_{\Phi,2}\dots L_{\Phi,\ell}$ is invariant.  \hfill\newline
\smallskip
\noindent (2) If $|\Psi \setminus \Phi| = 1$, i.e. if $Q_\Phi$ is a maximal
parabolic, then $N_\Phi = L_{\Phi,1}$ is invariant.\hfill\newline
\smallskip
\noindent (3) Let $G = SL(6;\R)$ with simple roots 
$\Psi = \{\psi_1\,, \dots , \psi_5\}$ in 
the usual order and $\Phi = \{\psi_1, \psi_4, \psi_5\}$.  Then
$\beta_1 = \psi_1 + \dots + \psi_5$, $\beta_2 = \psi_2 + \psi_3 + \psi_4$
and $\beta_3 = \psi_3$\,.  Note $\beta_1|_{\ga_\Phi} = \beta_2|_{\ga_\Phi}
\ne \beta_3|_{\ga_\Phi} = (\psi_3 + \psi_4)|_{\ga_\Phi}$\,. 
Thus $\gn_\Phi = \gl_{\Phi,1} + \gl_{\Phi_2}$ with
$\gl_{\Phi,1} = (\gl_1 + \gl_2)\cap \gn_\Phi$ and 
$\gl_{\Phi_2} = \gg_{\beta_3}$\,.  Now $\gg_{[\Phi, \beta_3]} \ne \gz_{\Phi,2}$
so the decomposition $N_\Phi =
L_{\Phi,1}L_{\Phi,2}\dots L_{\Phi,\ell}$ is not invariant.
\hfill\newline
\smallskip
\noindent (4)  In the example just above, 
$[\beta_3] = \{\psi_3, \psi_3 + \psi_4, \psi_3 + \psi_4 + \psi_5\}$.
The semisimple part $[\gm_\Phi,\gm_\Phi]$ of $\gm_\Phi$ is direct sum of 
$\gm_1 = \gs\gl(2;\R)$ with simple root $\psi_1$ and 
$\gm_{4,5} = \gs\gl(3;\R)$ with simple roots $\psi_4$ and $\psi_5$\,.
The action of $[\gm_\Phi,\gm_\Phi]$ on $\gg_{[\beta_3]}$ is trivial on
$\gm_1$ and the usual (vector) representation of $\gm_{4,5}$.  That
remains irreducible on the maximal compact $\gs\go(3)$ in $\gm_{4,5}$.
It follows that here the decomposition $N_\Phi =
L_{\Phi,1}L_{\Phi,2}\dots L_{\Phi,\ell}$ is not weakly invariant.
\medskip

\begin{lemma}\label{components}
Let $F = \exp(i\ga)\cap K$.  Then $F$ is an elementary abelian
$2$-group of cardinality $\leqq 2^{\dim \ga}$.  In particular, 
$F$ is finite, and if $x \in F$ then $x^2  = 1$.  Further,
$F$ is central in $M_\Phi$ $($thus also in $U_\Phi)$,
$U_\Phi = F U_\Phi^0$, $E_\Phi = FE_\Phi^0$ and  $M_\Phi = FM_\Phi^0$\,.
\end{lemma}
\begin{proof}
Let $\theta$ be the Cartan involution of $G$ for which $K = G^\theta$.
If $x \in F$ then $x = \theta(x) = x^{-1}$ so $x^2 = 1$.  Now $F$ is an
elementary abelian $2$-group of cardinality $\leqq 2^{\dim \ga}$, in
particular $F$ is finite.  
\medskip

Let $G_u$ denote the compact real form of $G_\C$ such that $G \cap G_u = K$,
and let $A_{\Phi,u}$ denote the torus subgroup $\exp(i\ga_\Phi)$.
The centralizer $M_{\Phi,u}A_{\Phi,u} = Z_{G_u}(A_{\Phi,u})$ is connected.  
It has a maximal torus $C_{\Phi,u}B_{\Phi,u}A_{\Phi,u}$ corresponding
to 
\begin{equation}\label{chop-0}
\gh_\C = \gc_{\Phi,\C}+\gb_{\Phi,\C}+\ga_{\Phi,\C}
\end{equation} 
where $\gc_\Phi$
is a Cartan subalgebra of $\gu_\Phi$\,, $\gc_\Phi + \gb_\Phi$ is a
Cartan subalgebra of $\gm_\Phi$ and $\gb_\Phi + \ga_\Phi = \ga$.
The complexification $M_{\Phi,\C}A_{\Phi,\C} = Z_{G_\C}(A_{\Phi,\C})$
is connected and has connected Cartan subgroup 
$C_{\Phi,\C}B_{\Phi,\C}A_{\Phi,\C}$\,.  Now every component of $M_\Phi A_\Phi 
= (M_{\Phi,\C}A_{\Phi,\C})\cap G$ contains an element of
$\exp(\gc_\Phi +i\gb_\Phi +i\ga_\Phi)$.  Thus every
component of its maximal compact subgroup $U_\Phi$ contains an element
of $\exp(i\gb_\Phi +i\ga_\Phi) = \exp(i\ga)$.  This proves
$U_\Phi \subset F U_\Phi^0$\,.  But $F \subset M_\Phi A_\Phi$\,, and is
finite and central there, so $F \subset U_\Phi$.  Now $U_\Phi = F U_\Phi^0$\,.
It follows that $M_\Phi = FM_\Phi^0$\,.  As $E_\Phi$ is the semidirect
product of $U_\Phi$ with an exponential solvable (thus topologically
contractible) group it also follows that $E_\Phi = FE_\Phi^0$\,.
\end{proof}

Notice that the parabolic $Q_\Phi$ is cuspidal (in the sense of 
Harish-Chandra) if and only if $\gb_\Phi = 0$, in other words if
and only if $M_\Phi$ has discrete series representations.  The cuspidal
parabolics are the ones used to construct standard tempered representations
of real reductive Lie groups.

\begin{lemma}\label{F-triv}
The action of $F$ on $\gs_\Phi^*$ is trivial.
\end{lemma}
\begin{proof} We know that the action of $F$ is trivial on each $\gz_j^*$
\cite[Proposition 3.6]{W2014}.  The action of $M_\Phi$ is absolutely 
irreducible on every $\ga_\Phi$-root space \cite[Theorem 8.13.3]{W1966}.
Recall $\gz_{\Phi,j} = \sum_{I_j} (\gg_{\beta_i} + \gl_i'')$ where 
$\gl_i'' = \sum_{J_i''} \gg_\alpha$ from (\ref{split-lr}) and 
(\ref{split-lphi}).  Using Lemma \ref{inter-compl} we see that the action 
of $F$ is trivial on each $\gg_{\beta_i} + \gl_i''$\,, thus trivial on 
$\gz_{\Phi,j}$\,, and thus trivial on their sum $\gs_\Phi$\,, and finally 
by duality is trivial on $\gs_\Phi^*$\,.
\end{proof}

When $N_\Phi = L_{\Phi,1}L_{\Phi,2}\dots L_{\Phi,\ell}$
is weakly invariant we can proceed more or less as in \cite{W2014}.
Set
\begin{equation}\label{regset}
\gr_\Phi^* = \{\lambda \in \gs_\Phi^* \mid P(\lambda) \ne 0 \text{ and }
	\Ad(U_\Phi)\lambda \text{ is a principal } U_\Phi\text{-orbit on }
	\gs_\Phi^*\}.
\end{equation}
If $c \ne 0$ and $\lambda \in \gr_\Phi^*$ then $c\lambda \in \gr_\Phi^*$\,.
Thus we obtain $\gr_\Phi^*$ by scaling the set of all $\lambda$ in a
unit sphere $s$ of $\gs_\Phi^*$ (for any norm) such that 
$\Ad(U_\Phi)\lambda \text{ is a principal } U_\Phi\text{-orbit on } s$.
Thus, as in the case of compact group actions on compact spaces, 
$\gr_\Phi^*$ is dense, open and $U_\Phi$-invariant in $\gs_\Phi^*$\,.
By definition of principal orbit the isotropy subgroups of $U_\Phi$ at the 
various points of $\gr_\Phi^*$ are conjugate, and we take a measurable 
section $\sigma$ to $\gr_\Phi^* \to \Ad^*(U_\Phi)\backslash \gr^*_\Phi$ on whose
image all the isotropy subgroups are the same,
\begin{equation}\label{m-iso-regset}
U'_\Phi: \text{ isotropy subgroup of } U_\Phi \text{ at }
	\sigma(U_\Phi(\lambda)), \text{ independent of } 
	\lambda \in \gr_\Phi^*\,.
\end{equation}
Lemma \ref{F-triv} says that $U'_\Phi = F(U'_\Phi \cap U_\Phi^0)$
In view of Lemma \ref{F-triv} the principal isotropy subgroups 
$U'_\Phi$ are specified by the
work of W.-C. and W.-Y. Hsiang \cite{HH1970} on the structure and
classification of principal orbits of compact connected linear groups.
With a glance back at (\ref{a-stab}) we have
\begin{equation}\label{ma-iso-regset}
U'_\Phi A'_\Phi: \text{ isotropy subgroup of } U_\Phi A_\Phi \text{ at }
        \sigma(U_\Phi A_\Phi (\lambda)), 
	\text{ independent of } \lambda \in \gr_\Phi^*\,.
\end{equation}
The first consequence, as in \cite[Proposition 3.3]{W2014}, is
\begin{theorem}\label{planch-mn}
Suppose that  $N_\Phi = L_{\Phi,1}L_{\Phi,2}\dots L_{\Phi,\ell}$
is weakly invariant.  Let $f \in \cC(U_\Phi N_\Phi)$ 
Given $\lambda \in \gr_\Phi^*$ let $\pi_\lambda^\dagger$ denote the 
extension of $\pi_\lambda$ to a representation of $U'_\Phi N_\Phi$ on
the space of $\pi_\lambda$.  Then
the Plancherel density at $\Ind_{U'_\Phi N_\Phi}^{U_\Phi N_\Phi}
(\pi_\lambda^\dagger \otimes \mu')$, $\mu' \in \widehat{U'_\Phi}$\,,
is $(\dim \mu')|P(\lambda)|$
and the Plancherel Formula for $U_\Phi N_\Phi$ is 
$$
f(un) = c\int_{\gr_\Phi^*/\Ad^*(U_\Phi)}\,
        {\sum}_{\mu' \in \widehat{U'_\Phi}}\,\,
        \tr \Bigl ( \bigl ( \Ind_{U'_\Phi N_\Phi}^{U_\Phi N_\Phi}\,
	(\pi_\lambda^\dagger \otimes \mu')\bigr )
		\bigl ( r_{un}f\bigr )\Bigr ) \cdot
        \dim(\mu')\cdot |P(\lambda)|d\lambda
$$
where $c = 2^{d_1 + \dots + d_m} d_1! d_2! \dots d_m!$\,, from
{\rm (\ref{c-d})}.
\end{theorem}

Combining Theorems \ref{planch-an} and \ref{planch-mn} we come to
\begin{theorem}\label{planch-uan}
Let $Q_\Phi = M_\Phi A_\Phi N_\Phi$ be a parabolic subgroup of the
real reductive Lie group $G$.  Let $U_\Phi$ be a maximal compact subgroup
of $M_\Phi$\,, so $E_\Phi := U_\Phi A_\Phi N_\Phi$ is a maximal
amenable subgroup of $Q_\Phi$\,.  Suppose that the decomposition
$N_\Phi = L_{\Phi,1}L_{\Phi,2}\dots L_{\Phi,\ell}$
is weakly invariant.  Given $\lambda \in \gr_\Phi^*$, $\phi \in \ga'_\Phi$
and $\mu' \in \widehat{U'_\Phi}$ denote
$$
\pi_{\lambda,\phi,\mu'} = 
	\Ind_{U'_\Phi A'_\Phi N_\Phi}^{U_\Phi A_\Phi N_\Phi}
	(\pi_\lambda^\dagger \otimes e^{i\phi} \otimes \mu')
	\in \widehat{E_\Phi}\,.
$$ 
Let
$\Theta_{\pi_{\lambda,\phi,\mu'}}: h \mapsto
    \tr \pi_{\lambda,\phi,\mu'}(h)$
denote its distribution character.  Then
$\Theta_{\pi_{\lambda,\phi,\mu'}}$ is a tempered distribution
on the maximal amenable subgroup $E_\Phi$\,.
If $f \in \cC(E_\Phi)$ then
$$
f(x) = c \int_{(\ga'_\Phi)^*}
        \left ( \int_{\gr_\Phi^*/\Ad^*(U_\Phi A_\Phi)}
	\left ( {\sum}_{\mu' \in \in \widehat{U'_\Phi}}\,\,
        \Theta_{\pi_{\lambda,\phi,\mu'}}(D(r_xf)) \cdot \dim(\mu') \right )
	\, |P(\lambda)| d\lambda \right ) d\phi
$$
where $c = (\frac{1}{2\pi})^{\dim \ga'_\Phi /2}\,
	2^{d_1 + \dots + d_m} d_1! d_2! \dots d_m!$\,.
\end{theorem}

\begin{proof} Theorem \ref{limit-inversion-uan} extends this result to
certain direct limit parabolics, and the calculation in the proof of 
Theorem \ref{limit-inversion-uan} specializes to give the proof of 
Theorem \ref{planch-uan}.
\end{proof}

When weak invariance fails we replace the $\gz_{\Phi,j}$ by the larger
\begin{equation}\label{bigger-z}
\gg_{[\Phi,\beta_j]} = {\sum}_{\alpha \in [\beta_j]_\Phi}\gg_\alpha
	\text{ where } [\beta_j]_\Phi = \{\alpha \in \Delta^+(\gg,\ga) \mid
	\alpha|_{\ga_\Phi} = \beta_{j_0}|_{\ga_\Phi}\}, \text{ any } j_0 \in I_j\,,
\end{equation}
as in (\ref{def-mult}).  Note that $\gg_{[\Phi,\beta_j]}$ is an 
irreducible $\Ad(M_\Phi^0)$-module.
We need to show that we can replace $\gs_\Phi = \sum \gz_{\Phi,j}$ by
$$
\widetilde{\gs_\Phi} := {\sum}_j\,\, \gg_{[\Phi,\beta_j]}
$$
in our Plancherel formulae.  The key is
\begin{lemma}\label{easy-tilde}
Let $\lambda_j \in \gg_{[\Phi,\beta_j]}^*$\,.  Split
$\gg_{[\Phi,\beta_j]} = \gz_{\Phi,j} + \gw_{\Phi,j}$ 
where $\gw_{\Phi,j} = \gg_{[\Phi,\beta_j]} \cap \gv_\Phi$ 
is the sum of the $\gg_\alpha$ that occur in
$\gg_{[\Phi,\beta_j]}$ but not in $\gz_{\Phi,j}$\,.  Then the Pfaffian
$\Pf_j(\lambda_j) = \Pf_j(\lambda_j|_{\gz_{\Phi,j}})$.
\end{lemma}
\begin{proof}
Write $\lambda_j = \lambda_{\gz,j} + \lambda_{\gw,j}$ where 
$\lambda_{\gz,j}(\gw_{\Phi,j}) = 0 = \lambda_{\gw,j}(\gz_{\Phi,j})$.
Let $\gg_\gamma, \gg_\delta \subset \gl_{\Phi,j}$ with
$[\gg_\gamma , \gg_\delta] \ne 0$.  Then 
$[\gg_\gamma , \gg_\delta] \subset \gl_{\Phi,j}$, so
$[\gg_\gamma , \gg_\delta] \cap \gw_{\Phi,j} = 0$, in
particular $\lambda_{\gw,j}([\gg_\gamma , \gg_\delta]) = 0$.
In other words $\lambda_j([\gg_\gamma , \gg_\delta]) =
\lambda_j|_{\gz_{\Phi,j}}([\gg_\gamma , \gg_\delta])$.
Now $b_{\lambda_j|_{\gz_{\Phi,j}}} = b_{\lambda_j}$, so their
Pfaffians are the same.
\end{proof}

In order to extend Theorems \ref{planch-mn} and \ref{planch-uan} we now need
only make some trivial changes to (\ref{regset}), (\ref{m-iso-regset}), 
(\ref{ma-iso-regset}) and the measurable section:
\begin{itemize}
\item $ \widetilde{\gr_\Phi}^* = \{\lambda \in \widetilde{\gs_\Phi}^* 
	\mid P(\lambda) \ne 0 \text{ and }
        \Ad(U_\Phi)\lambda \text{ is a principal } U_\Phi\text{-orbit on }
        \widetilde{\gs_\Phi}^*\}$.
\item $\widetilde{\sigma}$: measurable section to 
	$\widetilde{\gr_\Phi}^* \to 
	\widetilde{\gr_\Phi}^* \backslash U_\Phi$ on whose
	image all the isotropy subgroups are the same.
\item $U'_\Phi: \text{ isotropy subgroup of } U_\Phi \text{ at }
        \widetilde{\sigma}(U_\Phi(\lambda)), \text{ independent of }
        \lambda \in \widetilde{\gr_\Phi}^*$\,.
\item $U'_\Phi A'_\Phi: \text{ isotropy subgroup of } U_\Phi A_\Phi \text{ at }
        \widetilde{\sigma}(U_\Phi A_\Phi (\lambda)),
        \text{ independent of } \lambda \in \widetilde{\gr_\Phi}^*$\,.
\end{itemize}
The result is

\begin{theorem}\label{w-o-weakly-invar}
In Theorems {\rm \ref{planch-mn}} and {\rm \ref{planch-uan}} 
one can omit the requirement that 
$N_\Phi = L_{\Phi,1}L_{\Phi,2}\dots L_{\Phi,\ell}$ be weakly invariant.
\end{theorem}
\pagebreak

\centerline{\Large \bf Part II: Infinite Dimensional Theory}

\section{Direct Limits of Nilpotent Lie Groups}
\label{sec7}
\setcounter{equation}{0}

In this section we describe the basic outline for direct limits of
stepwise square integrable representations of simply connected nilpotent
Lie groups.  Later we will specialize these constructions to nilradicals 
$N_{\Phi,\infty} = \varinjlim N_{\Phi,n}$ of parabolic subgroups 
$Q_{\Phi,\infty} = \varinjlim Q_{\Phi,n}$ in our real reductive Lie
groups $G_\infty = \varinjlim G_n$\,.  In order to do that we
will need to adjust the ordering in the decompositions (\ref{setup}) 
so that they fit together as $n$ increases.  We do that by reversing the 
indices and keeping the $L_r$ constant as $n$ goes to infinity.  Thus, 
we suppose that
\begin{equation}\label{nil-direct-system}
\{N_n\} \text{ is a}\text{ strict direct system of connected nilpotent Lie groups,}
\end{equation}
in other words  the connected simply connected nilpotent Lie groups
$N_n$ have the property that $N_n$ is a closed analytic subgroup of $N_\ell$
for all $\ell \geqq n$.  As usual, $Z_r$ denotes the center of $L_r$\,.
                For each $n$, we require that
{\small
\begin{equation}\label{newsetup}
\begin{aligned}
N_n = L_1L_2&\cdots L_{m_n} \text{ where }\\
\phantom{X} \text{(a) }&L_r \text{ is a closed analytic subgroup of } N_n \text{ for }
        1 \leqq r \leqq m_n\text{ and} \\
\phantom{X} \text{(b) }&\text{each $L_r$ has unitary representations with
        coefficients in $L^2(L_r/Z_r)$.}\\
\phantom{X} \text{(x) }&L_{p,q} = L_{p+1}L_{p+2}\cdots L_q (p < q)
        \text{ and } N_{\ell,n} = L_{m_\ell +1}L_{m_\ell +2}\cdots L_{m_n}
        = L_{m_\ell,m_n} (\ell < n);\\
\phantom{X} \text{(c) }&N_{\ell ,n}
        \text{ is normal in } N_n
   \text{ and } N_n = N_r \ltimes N_{r,n} \text{ semidirect product,} \\
\phantom{X} \text{(d) }&\text{decompose }\gl_r = \gz_r + \gv_r \text{ and }
        \gn_n = \gs_n + {\bigoplus}_{r \leqq m_n}\, \gv_r \text{ where } 
	\gs_n = {\bigoplus}_{r \leqq m_n}\, \gz_r ; \\
    &\text{then } [\mathfrak{l}_r,{\mathfrak z}_s] = 0 \text{ and } 
	[\mathfrak{l}_r,\mathfrak{l}_s] \subset \gl''_s + \mathfrak{v}
        \text{ for } r < s \text{ where } \\
	&\gl_r = \gl'_r \oplus \gl''_r
	\text{ direct sum of ideals with } \gl''_r \subset \gz_r \text{ and } 
	\gv_r \subset \gl'_r
\end{aligned}
\end{equation}
}

With this setup we can follow the lines of the constructions in
\cite[Section 5]{W2013} as indicated in \S \ref{sec1} above.  Denote
\begin{equation}
P_n(\gamma_n) = \Pf_1(\lambda_1)\Pf_2(\lambda_2) \cdots \Pf_{m_n}(\lambda_{m_n})
\text{ where } \lambda_r \in \gz_r^* \text{ and } \gamma_n = \lambda_1 + \dots
	+ \lambda_{m_n}
\end{equation}
and the nonsingular set
\begin{equation}
\mathfrak{t}_n^* = \{\gamma_n \in \mathfrak{s}_n^* \mid P_n(\gamma_n) \ne 0\}.
\end{equation}
When $\gamma_n \in \gt_n^*$ the stepwise square integrable representation 
$\pi_{\gamma_n} \in \widehat{N_n}$ is defined as in Construction 
\ref{construction}, but with the indices reversed: 
$\pi_{\lambda_1 + \dots + \lambda_{m+1}} = \pi^\dagger_{\lambda_1 + \dots + \lambda_m}
\widehat{\otimes} \pi_{\lambda_{m+1}}$ with representation space 
$\cH_{\pi_{\lambda_1 + \dots + \lambda_{m+1}}} = 
\cH_{\pi_{\lambda_1 + \dots + \lambda_m}} \widehat{\otimes}
\cH_{\pi_{\lambda_{m+1}}}$\,.
\medskip

The parameter space for our representations of the direct limit Lie group
$N = \varinjlim N_n$ is
\begin{equation}\label{paramspacegen}
\gt^* = \left \{ \left . \gamma = (\gamma_\ell) \in \gs^* = 
	\varprojlim \gs_\ell^* \right |
	\gamma_\ell = \lambda_1 + \dots + \lambda_{m_\ell} 
	\in \gt_\ell^* \text{ for all } \ell \right \}.
\end{equation}
The closed normal subgroups $N_{n,n+1}$ and $N_{n,\infty}$ satisfy
$N_n \cong N_{n+1}/N_{n,n+1} \cong N/N_{n,\infty}$.  Let $\gamma \in \gt^*$ 
and denote
\begin{equation} \label{def-sq-lim-reps1}
\begin{aligned}
\pi_{\gamma,n} &\text{: the stepwise square integrable }
\pi_{\lambda_1 + \cdots + \lambda_{m_n}} \in \widehat{N_n}\\
\pi_{\gamma,n,n+1} &\text{: the stepwise square integrable }
\pi_{\lambda_{m_n+1} + \cdots + \lambda_{m_{n+1}}} \in \widehat{N_{n,n+1}}
\end{aligned}
\end{equation}
Using $N_n \cong N_{n+1}/N_{n,n+1}$ we lift $\pi_{\gamma,n}$ to a
representation $\pi^\dagger_{\gamma,n}$ of $N_{n+1}$ whose kernel contains
$N_{n,n+1}$ and we extend $\pi_{\gamma,n,n+1}$ to a representation
$\pi^\dagger_{\gamma,n,n+1}$
of $N_{n+1}$ on the same representation space $\cH_{\pi_{\gamma,n,n+1}}$.
Then we define
\begin{equation} \label{def-sq-lim-reps2}
\pi_{\gamma,{n+1}} = \pi^\dagger_{\gamma,n} \otimes \pi^\dagger_{\gamma,{n,n+1}} \in
	\widehat{N_{n+1}}\,.
\end{equation}
The representation space is the projective (jointly continuous) tensor product
$\mathcal{H}_{\pi_{\gamma,{n+1}}}
= \mathcal{H}_{\pi_{\gamma,n}} \widehat{\otimes} \cH_{\pi_{\gamma,{n,n+1}}}$
where $\cH_{\pi_{\gamma,{n,n+1}}} = 
\cH_{\pi_{\lambda,{m_n + 1}}} \widehat{\otimes} \cdots
        \widehat{\otimes} \cH_{\pi_{\lambda,{m_{n+1}}}}$.
Choose a $C^\infty$ unit vector 
$e_{n+1} \in \cH_{\pi_{\gamma,{n,n+1}}}$\,.  Then
\begin{equation}\label{tensor-inj-N}
v \mapsto v \otimes e_{n+1} \text{ is an } N_n \text{-equivariant isometry } 
\mathcal{H}_{\pi_{\gamma,n}} \hookrightarrow \mathcal{H}_{\pi_{\gamma,{n+1}}}
\end{equation}
exhibits $\pi_{\gamma,n}$ as the restriction $\pi_{\gamma,{n+1}}|_{N_n}$ on
the subspace $(\mathcal{H}_{\pi_{\gamma,n}} \otimes e_{n+1})$ of
$\mathcal{H}_{\pi_{\gamma,{n+1}}}$\,.

\begin{lemma}\label{lim-rep-N}
The maps just described, define direct system 
$\{(\pi_{\gamma,n},\cH_{\pi_{\gamma,n}})\}$ of irreducible stepwise square 
integrable unitary representations, and thus define an irreducible unitary 
representation
$\pi_\gamma = \varinjlim \pi_{\gamma,n}$ of $N = \varinjlim N_n$
on the Hilbert space 
$\cH_{\pi_\gamma} =\varinjlim \mathcal{H}_{\pi_{\gamma,n}}$\,.
\end{lemma}

The representations $\pi_\gamma$ described in Lemma \ref{lim-rep-N} 
are the {\sl limit stepwise square integrable} representations of $N$.  
Corollary \ref{l2-well-def} will show that the unitary 
equivalence class of $\pi_\gamma$ is independent of the choice of the 
$C^\infty$ unit vectors
$e_n$\,.

\section{Direct Limit Structure of Parabolics and some Subgroups}\label{sec8}
\setcounter{equation}{0}

We adapt the constructions Section \ref{sec7}
to limits of nilradicals of parabolic subgroups.  That requires some alignment
of root systems so that the direct limit respects the restricted root 
structures, in particular the strongly orthogonal root structures,
of the $N_n$\,.  We enumerate the set $\Psi_n = 
\Psi(\mathfrak{g}_n, \mathfrak{a}_n)$ of nonmultipliable simple restricted
roots so that, in the Dynkin diagram, for type $A$ we spread from the
center of the diagram.  For types $B$, $C$ and $D$, \,
$\psi_1$ is the \textit{right} endpoint.
In other words for $\ell \geqq n$ $\Psi_\ell$
is constructed from $\Psi_n$ adding simple roots to the \textit{left} end
of their Dynkin diagrams.  Thus

{\small
\begin{equation}\label{rootorderA}
\begin{aligned}
&\begin{tabular}{|c|l|c|}\hline
$A_{2\ell+1}$ &
\setlength{\unitlength}{.4 mm}
\begin{picture}(180,18)
\put(10,2){\circle{2}}
\put(5,5){$\psi_{-\ell}$}
\put(11,2){\line(1,0){13}}
\put(27,2){\circle*{1}}
\put(30,2){\circle*{1}}
\put(33,2){\circle*{1}}
\put(36,2){\line(1,0){13}}
\put(50,2){\circle{2}}
\put(45,5){$\psi_{-n}$}
\put(51,2){\line(1,0){13}}
\put(67,2){\circle*{1}}
\put(70,2){\circle*{1}}
\put(73,2){\circle*{1}}
\put(76,2){\line(1,0){13}}
\put(90,2){\circle{2}}
\put(87,5){$\psi_0$}
\put(91,2){\line(1,0){13}}
\put(107,2){\circle*{1}}
\put(110,2){\circle*{1}}
\put(113,2){\circle*{1}}
\put(116,2){\line(1,0){13}}
\put(130,2){\circle{2}}
\put(128,5){$\psi_n$}
\put(131,2){\line(1,0){13}}
\put(147,2){\circle*{1}}
\put(150,2){\circle*{1}}
\put(153,2){\circle*{1}}
\put(156,2){\line(1,0){13}}
\put(170,2){\circle{2}}
\put(167,5){$\psi_\ell$}
\end{picture}
&$\ell \geqq n \geqq 0$
\\
\hline
\end{tabular}\\
&\begin{tabular}{|c|l|c|}\hline
\setlength{\unitlength}{.4 mm}
$A_{2\ell}\phantom{i.}$ &
\setlength{\unitlength}{.4 mm}
\begin{picture}(180,18)
\put(1,2){\circle{2}}
\put(-4,5){$\psi_{-\ell}$}
\put(2,2){\line(1,0){13}}
\put(18,2){\circle*{1}}
\put(21,2){\circle*{1}}
\put(24,2){\circle*{1}}
\put(27,2){\line(1,0){13}}
\put(41,2){\circle{2}}
\put(36,5){$\psi_{-n}$}
\put(42,2){\line(1,0){13}}
\put(58,2){\circle*{1}}
\put(61,2){\circle*{1}}
\put(64,2){\circle*{1}}
\put(67,2){\line(1,0){13}}
\put(81,2){\circle{2}}
\put(78,5){$\psi_{-1}$}
\put(82,2){\line(1,0){13}}
\put(96,2){\circle{2}}
\put(93,5){$\psi_1$}
\put(97,2){\line(1,0){13}}
\put(113,2){\circle*{1}}
\put(116,2){\circle*{1}}
\put(119,2){\circle*{1}}
\put(122,2){\line(1,0){13}}
\put(136,2){\circle{2}}
\put(134,5){$\psi_n$}
\put(137,2){\line(1,0){13}}
\put(153,2){\circle*{1}}
\put(156,2){\circle*{1}}
\put(159,2){\circle*{1}}
\put(162,2){\line(1,0){13}}
\put(176,2){\circle{2}}
\put(173,5){$\psi_\ell$}
\end{picture}
&$\ell \geqq n \geqq 1$
\\
\hline
\end{tabular}
\end{aligned}
\end{equation}
\begin{equation}\label{rootorderBCD}
\begin{aligned}
&\begin{tabular}{|c|l|c|}\hline
$B_\ell$&
\setlength{\unitlength}{.5 mm}
\begin{picture}(155,13)
\put(5,2){\circle{2}}
\put(2,5){$\psi_{\ell}$}
\put(6,2){\line(1,0){13}}
\put(24,2){\circle*{1}}
\put(27,2){\circle*{1}}
\put(30,2){\circle*{1}}
\put(34,2){\line(1,0){13}}
\put(48,2){\circle{2}}
\put(45,5){$\psi_n$}
\put(49,2){\line(1,0){23}}
\put(73,2){\circle{2}}
\put(70,5){$\psi_{n-1}$}
\put(74,2){\line(1,0){13}}
\put(93,2){\circle*{1}}
\put(96,2){\circle*{1}}
\put(99,2){\circle*{1}}
\put(104,2){\line(1,0){13}}
\put(118,2){\circle{2}}
\put(115,5){$\psi_2$}
\put(119,2.5){\line(1,0){23}}
\put(119,1.5){\line(1,0){23}}
\put(143,2){\circle*{2}}
\put(140,5){$\psi_1$}
\end{picture}
&$\ell\geqq n \geqq 2$\\
\hline
\end{tabular} \\
&\begin{tabular}{|c|l|c|}\hline
$C_\ell$ &
\setlength{\unitlength}{.5 mm}
\begin{picture}(155,13)
\put(5,2){\circle*{2}}
\put(2,5){$\psi_{\ell}$}
\put(6,2){\line(1,0){13}}
\put(24,2){\circle*{1}}
\put(27,2){\circle*{1}}
\put(30,2){\circle*{1}}
\put(34,2){\line(1,0){13}}
\put(48,2){\circle*{2}}
\put(45,5){$\psi_n$}
\put(49,2){\line(1,0){23}}
\put(73,2){\circle*{2}}
\put(70,5){$\psi_{n-1}$}
\put(74,2){\line(1,0){13}}
\put(93,2){\circle*{1}}
\put(96,2){\circle*{1}}
\put(99,2){\circle*{1}}
\put(104,2){\line(1,0){13}}
\put(118,2){\circle*{2}}
\put(115,5){$\psi_2$}
\put(119,2.5){\line(1,0){23}}
\put(119,1.5){\line(1,0){23}}
\put(143,2){\circle{2}}
\put(140,5){$\psi_1$}
\end{picture}
& $\ell\geqq n \geqq 3$
\\
\hline
\end{tabular}\\
&\begin{tabular}{|c|l|c|}\hline
$D_\ell$ &
\setlength{\unitlength}{.5 mm}
\begin{picture}(155,20)
\put(5,9){\circle{2}}
\put(2,12){$\psi_{\ell}$}
\put(6,9){\line(1,0){13}}
\put(24,9){\circle*{1}}
\put(27,9){\circle*{1}}
\put(30,9){\circle*{1}}
\put(34,9){\line(1,0){13}}
\put(48,9){\circle{2}}
\put(45,12){$\psi_n$}
\put(49,9){\line(1,0){23}}
\put(73,9){\circle{2}}
\put(70,12){$\psi_{n-1}$}
\put(74,9){\line(1,0){13}}
\put(93,9){\circle*{1}}
\put(96,9){\circle*{1}}
\put(99,9){\circle*{1}}
\put(104,9){\line(1,0){13}}
\put(118,9){\circle{2}}
\put(113,12){$\psi_3$}
\put(119,8.5){\line(2,-1){13}}
\put(133,2){\circle{2}}
\put(136,0){$\psi_1$}
\put(119,9.5){\line(2,1){13}}
\put(133,16){\circle{2}}
\put(136,14){$\psi_2$}
\end{picture}
& $\ell\geqq n \geqq 4$
\\
\hline
\end{tabular}
\end{aligned}
\end{equation}
}

\noindent We describe this by saying that $G_\ell$ {\em propagates} $G_n$\,.
For types $B$, $C$ and $D$ this is the same as the notion of propagation in
\cite{OW2011} and \cite{OW2014}.  
\medskip

The direct limit groups obtained
this way are $SL(\infty;\C)$, $SO(\infty;\C)$, $Sp(\infty;\C)$, $SL(\infty;\R)$,
$SL(\infty;\H)$, $SU(\infty,q)$ with $q \leqq \infty$, $SO(\infty,q)$
with $q \leqq \infty$, $Sp(\infty,q)$ with $q \leqq \infty$, 
$Sp(\infty;\R)$ and $SO^*(2\infty)$.
\medskip

Let $\{G_n\}$ be a direct system of real semisimple Lie groups in which
$G_\ell$ propagates $G_n$ for $\ell \geqq n$.  Then the corresponding simple
restricted root systems satisfy $\Psi_n \subset \Psi_\ell$ as indicated in
(\ref{rootorderA}) and (\ref{rootorderBCD}).  Consider conditions on a 
family $\Phi = \{\Phi_n\}$ of subsets
$\Phi_n \subset \Psi_n$ such that $G_n \hookrightarrow G_\ell$ maps the
corresponding parabolics $Q_{\Phi,n} \hookrightarrow Q_{\Phi,\ell}$. Then
we have
\begin{equation}\label{limq}
Q_{\Phi,\infty} := \varinjlim Q_{\Phi,n} \text{ inside } 
G_\infty := \varinjlim G_n\,.
\end{equation}
Express $Q_{\Phi,n} = M_{\Phi,n} A_{\Phi,n} N_{\Phi,n}$ and
$Q_{\Phi,\ell} = M_{\Phi,\ell} A_{\Phi,\ell} N_{\Phi,\ell}$\,.  Then
$M_{\Phi,n} \hookrightarrow M_{\Phi,\ell}$ is equivalent to 
$\Phi_n \subset \Phi_\ell$\,, $A_{\Phi,n} \hookrightarrow A_{\Phi,\ell}$ is
implicit in the condition that $G_\ell$ propagates $G_n$\,, 
and $N_{\Phi,n} \hookrightarrow N_{\Phi,\ell}$ is equivalent to 
$(\Psi_n \setminus \Phi_n) \subset (\Psi_\ell \setminus \Phi_\ell)$\,.  As
before let $U_{\Phi,n}$ denote a maximal compact subgroup of $M_{\Phi,n}$\,;
we implicitly assume that $U_{\Phi,n} \hookrightarrow U_{\Phi,\ell}$
whenever $M_{\Phi,n} \hookrightarrow M_{\Phi,\ell}$\,.  
\medskip

We will extend some of our results from the finite dimensional
setting to these subgroups of $Q_{\Phi,\infty}$\,.
\begin{equation}\label{lim-subgroups1}
\begin{aligned}
{\rm (a)}\,
&N_{\Phi,\infty} := \varinjlim N_{\Phi,n} \text{ maximal locally unipotent 
	subgroup, requiring } 
	(\Psi_n \setminus \Phi_n) \subset (\Psi_\ell \setminus \Phi_\ell), \\
{\rm (b)}\,
&A_{\Phi,\infty} := \varinjlim A_{\Phi,n}\,, \\
{\rm (c)}\,
&U_{\Phi,\infty} := \varinjlim U_{\Phi,n} \text{ maximal lim-compact subgroup, 
	requiring } \Phi_n \subset \Phi_\ell,\\
{\rm (d)}\,
&U_{\Phi,\infty} N_{\Phi,\infty} := \varinjlim U_{\Phi,n}N_{\Phi,n} 
	\text{, requiring } \Phi_n \subset \Phi_\ell \text{ and } 
	  (\Psi_n \setminus \Phi_n) \subset (\Psi_\ell \setminus \Phi_\ell).
\end{aligned}
\end{equation}  
\medskip

To study these we will need to extend some notation from the finite 
dimensional setting to the system $\{\gg_n\}$.  For 
$\alpha \in \Delta^+(\gg_n,\ga_n)$ we denote
\begin{equation}
 [\alpha]_{\Phi,n} = \{\delta \in \Delta^+(\gg_n,\ga_n) \,\mid\,
	\delta|_{\ga_{\Phi,n}} = \alpha|_{\ga_{\Phi,n}}\} \text{ and }
 \gg_{\Phi,n,\alpha} = 
	{\sum}_{\delta \in [\alpha]_{\Phi,n}} \gg_\delta \,.
\end{equation}
The adjoint action of $\gm_{\Phi,n}$ on
$\gg_{\Phi,n,\alpha}$ is absolutely irreducible \cite[Theorem 8.3.13]{W1966}.
Note that $\gg_{\Phi,n,\alpha}$ is the sum of the root spaces for roots 
$\delta = \sum_{\psi \in \Psi_n} n_\psi(\delta)\psi \in \Delta^+(\gg_n,\ga_n)$ 
such that $n_\psi(\delta) = n_\psi(\alpha)$ for all $\psi \in \Psi_n \setminus
\Phi_n$\,, in other words the same coefficients along $\Psi_n \setminus \Phi_n$ 
in $\sum_{\psi \in \Psi_n} n_\psi(\cdot)\psi$.  The following lemma is immediate. 
\medskip

\begin{lemma}\label{very-restricted}
Let $n \leqq \ell$ and assume the condition 
$(\Psi_n \setminus \Phi_n) \subset (\Psi_\ell \setminus \Phi_\ell)$
of {\rm (\ref{lim-subgroups1})(a)} for $N_{\Phi,\infty}$\,.
Then $\gg_{\Phi,n,\alpha} \subset \gg_{\Phi,\ell,\alpha}$\,.  
In particular we have the joint $\ga_{\Phi,\infty}$-eigenspaces 
$\gg_{\Phi,\infty,\alpha}
= \varinjlim_n\, \gg_{\Phi,n,\alpha}$ in $\gn_{\Phi,\infty}$\,.
\end{lemma}
\medskip

We will also say something about representations, but not about
Fourier inversion, for the
\begin{equation}\label{lim-subgroups2}
\begin{aligned}
&A_{\Phi,\infty} N_{\Phi,\infty} := \varinjlim A_{\Phi,n}N_{\Phi,n} 
	\text{ max. exponential solvable subgroup where } 
	(\Psi_n \setminus \Phi_n) \subset (\Psi_\ell \setminus \Phi_\ell),
	\text{ and for the } \\
&E_{\Phi,\infty}:= \varinjlim E_{\Phi,n} 
	\text{ maximal amenable subgroup where } \Phi_n \subset \Phi_\ell
	\text{ and }
	(\Psi_n \setminus \Phi_n) \subset (\Psi_\ell \setminus \Phi_\ell)\,.
\end{aligned}
\end{equation}
Here  $E_{\Phi,n} = U_{\Phi,n} A_{\Phi,n} N_{\Phi,n}$\,, so
$E_{\Phi,\infty} = U_{\Phi,\infty} A_{\Phi,\infty} N_{\Phi,\infty}$\,.
The difficulty with Fourier inversion for the two limit groups of 
(\ref{lim-subgroups2}) is that we don't have an explicit 
Dixmier-Puk\'anszky operator.
\medskip

Start with $N_{\Phi,\infty}$\,.  For that we must assume
$(\Psi_n \setminus \Phi_n) \subset (\Psi_\ell \setminus \Phi_\ell)$.
In view of the propagation assumption on the $G_n$ the maximal set of
strongly orthogonal non-multipliable roots in $\Delta^+(\gg_n,\ga_n)$
is increasing in $n$.  It is
obtained by cascading up (we reversed the indexing from the finite
dimensional setting) has form $\{\beta_1, \dots , \beta_{r_n}\}$
in $\Delta^+(\gg_n,\ga_n)$.  Following ideas of Section \ref{sec4} we 
partition 
$\{\beta_1, \dots , \beta_{r_n}\} = \bigcup_k \bigcup_{I_{n,k}} \beta_i$
where $I_{n,k}$ consists of the indices $i$ for which the $\beta_i$ have a 
given restriction to $\ga_{\Phi,n}$ and belong to $\Delta^+(\gg_n,\ga_n)$.  
Note that $I_{n,k}$ can increase as $n$ increases. This happens in some cases 
where the $\Phi$ stop growing,
i.e. where there is an index $n_0$ such that $\Phi_n = \Phi_{n_0}
\ne \emptyset$ for $n \geqq n_0$\,.  That is the case when
$\Delta(\gg_n,\ga_n)$ is of type $A_n$ with each $\Psi = \{\psi_1\}$.
Thus we also denote $I_{\infty,k} = \bigcup_n I_{n,k}$\,.
\medskip

As in (\ref{big-summands}), following the idea of $\gl_j = \gz_j + \gv_j$\,, 
we define
\begin{equation}\label{big-summands-n}
\begin{aligned}
&\gl_{\Phi,n,j} \phantom{i}= {\sum}_{i \in I_{n,j}} (\gl_i \cap \gn_{\Phi,n})
	\text{, the } \beta_j\text{ part of } \gn_{\Phi,n}\,, \\
&\gl_{\Phi,\infty,j} = {\sum}_{i \in I_{\infty,j}} (\gl_i \cap \gn_\Phi)
	\text{, the } \beta_j\text{ part of } \gn_{\Phi,\infty}\,, \\
&\gz_{\Phi,\infty,j} = {\sum}_n \,\gz_{\Phi,n,j}\,,\,\, 
	\gs_{\Phi,\infty} = {\sum}_j \,\gz_{\Phi,\infty,j} \text{ and }
\gv_{\Phi,\infty} = {\sum}_{n,j} \,\gv_{\Phi,n,j}\,,
\end{aligned}
\end{equation}
so $\gn_{\Phi,\infty} = \gs_{\Phi,\infty} + \gv_{\Phi,\infty}$\,.
We'll also use $\gs_{\Phi,n} = 
{\sum}_j \,\gz_{\Phi,n,j}$ and $\gv_{\Phi,n} = {\sum}_j \,\gv_{\Phi,n,j}$\,,
so $\gn_{\Phi,n} = \gs_{\Phi,n} + \gv_{\Phi,n}$\,.
\medskip

$L_{\Phi,n,j}$ denotes the analytic subgroup with Lie algebra $\gl_{\Phi,n,j}$
and $L_{\Phi,\infty,j} = \varinjlim_n L_{\Phi,n,j}$ has Lie algebra 
$\gl_{\Phi,\infty,j}$\,.  We have this set up so that 
\begin{equation}\label{big-summands-N}
N_{\Phi,\infty} = {\varinjlim}_n\, N_{\Phi,n} = 
{\varinjlim}_j\, L_{\Phi,\infty,j} = {\varinjlim}_j\, {\varinjlim}_n\, L_{\Phi,n,j}\,.
\end{equation}

\section{Representations of the Limit Groups I: the $N_{\Phi,\infty}$}
\label{sec9}
\setcounter{equation}{0}

In this section we indicate the limit stepwise square integrable 
representations $\pi_{\Phi,\gamma} = \varinjlim \pi_{\Phi,\gamma_n}$ of the 
direct limit group $N_{\Phi,\infty} = \varinjlim N_{\Phi,n}$\,.  The 
parameter space for the stepwise square integrable representations of the 
$N_{\Phi,n}$ is 
$\gt_{\Phi,n}^* = \{\gamma_n \in \gs_{\Phi,n}^* \,\mid\, P(\gamma_n) \ne 0 \}$
where $\gamma_n = \sum_1^{m_n}\lambda_j$ and $P(\gamma_n)$ is the product of
the Pfaffians $P_j(\lambda_j)$. Note that $\gamma_\ell|_{\gs_{\Phi,n}} =
\gamma_n$ for $\ell > n$. The parameter space for the $\pi_{\Phi,\gamma}$ 
is $\gt_{\Phi,\infty}^* = \{(\gamma_n) \in \gs_{\Phi,\infty}^* \mid
\text{ each } \gamma_n \in \gt_{\Phi,n}^*\}$ where
$\gs_{\Phi,\infty}^* = \varprojlim \gs_{\Phi,n}$.  The stepwise square 
integrable representations $\pi_{\gamma_n}$ were obtained recursively in
Construction \ref{construction}, from square integrable representations of
the $L_r$, $r \leqq m_n$\,, and in Lemma \ref{lim-rep-N} we described method 
of construction of the direct limit representations $\pi_{\Phi,\gamma}$\,.
\medskip

As noted before we must assume the condition
$(\Psi_n \setminus \Phi_n) \subset (\Psi_\ell \setminus \Phi_\ell)$
of (\ref{lim-subgroups1})(a), so that $\{N_{\Phi,n}\}$ is a direct system, 
in order to work with $N_{\Phi,\infty}$\,.
Then we have  the decompositions (\ref{big-summands-n}) and
(\ref{big-summands-N}).  With those in mind we will build up the
parameter space for direct limits of stepwise square integrable 
representations of $N_{\Phi,\infty}$ in two steps.  First,

\begin{lemma}\label{restr}
If $\lambda \in \gg_{\Phi,\infty,\beta_j}^*$ 
then the bilinear form $b_\lambda$ on
$\gn_{\Phi,\infty,j}/\gz_{\Phi,\infty,j}$ satisfies 
$b_\lambda = b_{\lambda|_{\gz_{\Phi,\infty,j}}}$\,.
\end{lemma}
\begin{proof} 
Let $n$ be sufficiently large that  $\gg_{\beta_j} \subset \gl_{\Phi,n,j}$\,.
Apply Lemma \ref{very-restricted} to each $\gg_{\Phi,\ell,\beta_j}$ with
$\ell \geqq n$.  That gives $(b_\lambda)|_{\gn_{\Phi,\ell}/\gs_{\Phi,\ell}}
= (b_{\lambda|_{\gz_{\Phi,\ell,j}}})|_{\gn_{\Phi,\ell}/\gs_{\Phi,\ell}}$\,.
As $\ell$ increases the additional brackets go into $\gn_{\Phi,\ell}$ and
thus into the kernel of $b_\lambda$\,.
\end{proof}

Second, we define the $\beta_j$ part of the parameter space.  In view of 
Lemma \ref{restr} we need only look at
\begin{equation}
\gt_{\Phi,\infty,j}^* = 
	\left \{ \left . \lambda_j = (\lambda_{n,j}) \in {\varprojlim}_n\, 
	\gz_{\Phi,n,j}^* \right |
	\lambda_{n,j} \in \gz_{\Phi,n,j}^* \text{ with }
	P_{\gl_{\Phi,n,j}}(\lambda_{n,j}) \ne 0 \text{ for } n \geqq n(\lambda_j)
	\right \}
\end{equation}	
where $n(\lambda_j)$ is the first index $n$ such that 
$\lambda_{n,j} \in \gz_{\Phi,n,j}^*$. 
We start this way because of the possibility that the 
$\gz_{\Phi,n,j}$ 
could grow, for fixed $j$, if the multiplicity of $\beta_j$ as a joint
eigenvalue of $\ad(\ga_{\Phi,n})$, increases as $n$ increases.  Third, 
\begin{equation}
\gt_{\Phi,\infty}^* = 
	\left \{ \left . \gamma = (\gamma_n) \in {\varprojlim}_n \gs^*_{\Phi,n} 
	\right | \text{ every } \gamma_n \in \gt_{\Phi,n}^* \right \}.
\end{equation}

Fix $\gamma = (\gamma_j) \in \gt_{\Phi,\infty}^*$\,.  As in Construction
\ref{construction} and Lemma \ref{lim-rep-N} we have the limit stepwise square 
integrable representation
$\pi_{\Phi,\lambda_j,\infty}$ of $L_{\Phi,\infty,j}$.  
Apply Construction \ref{construction} and Lemma \ref{lim-rep-N} to the 
$\pi_{\Phi,\lambda_j,\infty}$ as $j$ increases, obtaining the limit 
stepwise square integrable representation  
$\pi_{\Phi,\gamma,\infty} \in \widehat{N_{\Phi,\infty}}$. 

\begin{theorem} \label{l2-lim}
Assume the condition
$(\Psi_n \setminus \Phi_n) \subset (\Psi_\ell \setminus \Phi_\ell)$
of {\rm (\ref{lim-subgroups1})(a)}, so that $\{N_{\Phi,n}\}$ is a direct
system and $N_{\Phi,\infty} = \varinjlim N_{\Phi,n}$ is well defined.  
Let $\gamma = (\gamma_n) \in \gt_{\Phi,\infty}^*$ and $\pi_{\Phi,\gamma,\infty}
= \varinjlim \pi_{\Phi,\gamma,n}$ as in {\rm Lemma \ref{lim-rep-N}}.  View 
$\cH_{\pi_{\Phi,\gamma,\infty}}$ as the direct limit of the
$\cH_{\pi_{\Phi,\gamma,n}}$ in the category of Hilbert spaces and
partial isometries.  Let
$u,v \in \cH_{\pi_{\Phi,\gamma,\ell}} \subset 
\cH_{\pi_{\Phi,\gamma,\infty}}$.  Then the coefficient function
$f_{\pi_{\Phi,\gamma,\infty};\,u,v}(x) 
= \langle u, \pi_{\Phi,\gamma,\infty}(x)v\rangle$ satisfies
\begin{equation}\label{frob-schur-n-inf}
||f_{\pi_{\Phi,\gamma,\infty};\,u,v}|_{N_{\Phi,\ell}}||^2_{L^2(N_{\Phi,\ell}/
	S_{\Phi,\ell})} = \tfrac{||u||^2||v||^2}{|P_\ell(\gamma_\ell)|}
\end{equation}
\end{theorem}

\begin{proof}
Let $u = \bigotimes u_j$ and $v = \bigotimes v_j$ where 
$u_j, v_j \in \cH_{\pi_{\Phi,\gamma_j,\infty}}$\,, the representation
spaces of the $\pi_{\Phi,\gamma_j,\infty}$\,.  We know from stepwise
square integrability that the coefficients satisfy
$$
||f_{\pi_{\Phi,\gamma_j,\infty};\,u_j,v_j}|_{N_n}||_{
	L^2(L_{\Phi,n,j}/Z_{\Phi,n,j})}^2
	= \tfrac{||u_j||^2 ||v_j||^2}{|P_{\gl_{\Phi,n,j}}(\gamma_j)|}
	\text{ for } n \ggg 0.
$$
In other words,
$$
||f_{\pi_{\Phi,\gamma_j,\infty};\,u_j,v_j}|_{N_n}||^2_{
	L^2(L_{\Phi,\infty,j}/Z_{\Phi,\infty,j})} =
	\tfrac{||u_j||^2 ||v_j||^2}{|P_{\gl_{\Phi,\infty,j}}(\gamma_j)|}\,.
$$
Taking the product over $j$ we have (\ref{frob-schur-n-inf}) for
decomposable $u$ and $v$.  Decomposable vectors are dense in 
$\cH_{\pi_{\Phi,\gamma,\infty}}$ so 
(\ref{frob-schur-n-inf}) follows from the decomposable case by continuity.
\end{proof}

Now we continue as in \cite[Sections 3, 4 \& 5]{W2015}.  The first step is the 
rescaling implicit in Theorem \ref{l2-lim}, specifically in 
(\ref{frob-schur-n-inf}), which holds in our situation with only the obvious
changes.
Recall $N_{\Phi,a,b} = L_{\Phi,m_a+1}\dots L_{\Phi,m_b} = L_{\Phi,m_a+1,m_n}$,
and $N_{\Phi,a,\infty} = \varinjlim_b N_{\Phi,a,b}$, so
$N_{\Phi,\infty} = N_{\Phi,n} \ltimes N_{\Phi,n,\infty}$.

\begin{proposition}\label{rescale}
Let $\gamma \in \mathfrak{t}_{\Phi,\infty}^*$ and $\ell > n$ so that
$\gamma_{\ell}|_{\gs_{\Phi,n}} = \gamma_n$.
Then $\pi_{\Phi,\gamma,{\ell}}|_{N_{\Phi,n}}$ is an infinite multiple of 
$\pi_{\Phi,\gamma,n}$.  Split $\mathcal{H}_{\pi_{\Phi,\gamma,{\ell}}} 
= \cH' \widehat{\otimes}\mathcal{H}''$ where $\cH' 
= \mathcal{H}_{\pi_{\Phi,\gamma,n}}$, and where $\cH'' =
\mathcal{H}_{\pi_{\Phi,\lambda,{m_n +1}}} \widehat{\otimes} 
\cdots \widehat{\otimes} \mathcal{H}_{\pi_{\Phi,\lambda,{m_{\ell}}}}$ with
$\gamma_\ell = \lambda_1 + \dots + \lambda_{m_\ell}$.
Choose a $C^\infty$ unit vector $e \in \mathcal{H}''$, so
\begin{equation}\label{unitfactor}
\mathcal{H}_{\pi_{\Phi,\gamma,n}} \hookrightarrow \mathcal{H}_{\pi_{\Phi,\gamma,{\ell}}} 
\text{ by } v \mapsto v\otimes e
\end{equation}
is an $N_n$-equivariant isometric injection that sends $C^\infty$
vectors to $C^\infty$ vectors.
If $u, v \in \mathcal{H}_{\pi_{\Phi,\gamma,n}}$ then
\begin{equation}\label{rescale1}
\begin{aligned}
||f_{\pi_{\Phi,\gamma,{\ell}};\,u\otimes e,v\otimes e}&||^2_
	{L^2(N_{\Phi,\ell}/S_{\Phi,\ell})} \\
        &= \frac{|P_n(\gamma,n)|}{|P_{\ell}(\gamma,{\ell})|}
	||f_{\pi_{\Phi,\gamma,n};\,u\otimes e,v\otimes e}||^2_
        {L^2(N_{\Phi,n}/S_{\Phi,n})}
\end{aligned}
\end{equation}
\end{proposition}

Given $\gamma \in \mathfrak{t}_{\Phi,\infty}^*$ consider the
unitary character $\zeta_\gamma = \exp(2\pi i \gamma)$ on $S_{\Phi,\infty}$,
given by $\zeta_\gamma(\exp(\xi)) = e^{2\pi i \gamma(\xi)}$ for
$\xi \in \mathfrak{s}_{\Phi,\infty}$\,.
The corresponding Hilbert space is 
$$
L^2(N_{\Phi,\infty}/S_{\Phi,\infty};\zeta_\gamma) =
	\varinjlim L^2(N_{\Phi,n}/S_{\Phi,n};\zeta_{\gamma_n})
$$ 
where 
$$
\begin{aligned}
L^2(&N_{\Phi,n}/S_{\Phi,n};\zeta_{\gamma_n}) = \\
& \{f: N_{\Phi,n} \to \C \mid f(gx) = \zeta_\gamma(x)^{-1}f(g)
\text{ and } |f| \in L^2(N_{\Phi,n}/S_{\Phi,n})
\text{ for } x \in S_{\Phi,n}\}.
\end{aligned}
$$
The finite linear combinations of the 
coefficients $f_{\pi_{\Phi,\gamma,n};\,u,v}$\,, 
where $u, v \in \cH_{\pi_{\Phi,\gamma,n}}$\,, are dense in
$L^2(N_{\Phi,n})$.
That gives us a $N_{\Phi,n} \times N_{\Phi,n}$ equivariant
Hilbert space isomorphism
$$
L^2(N_{\Phi,n}/S_{\Phi,n}; \zeta_{\gamma_n})
\cong \cH_{\pi_{\Phi,\gamma,n}} \widehat{\otimes}
\cH^*_{\pi_{\Phi,\gamma,n}}\,.
$$
The stepwise square integrable group $N_{\Phi,n}$ satisfies
$$
L^2(N_{\Phi,n}) 
= \int_{\gamma_n \in \gt_{\Phi,n}^*} \cH_{\pi_{\Phi,\gamma,n}}
\widehat{\otimes}\cH^*_{\pi_{\Phi,\gamma,n}} |P_n(\gamma)|d\gamma_n\,.
$$
That expands functions on $N_{\Phi,\infty} = N_{\Phi,1}N_{\Phi,2}\dots$
that depend only on the first $m_n$ factors.  To increase the number of
factors we must deal the renormalization
implicit in (\ref{rescale1}).  Reformulate (\ref{rescale1}):
\begin{equation}\label{hilbert-proj0}
p_{\gamma,n,\ell}: f_{\pi_{\Phi,\gamma,{\ell}};\,u\otimes u',v\otimes v'}
   \mapsto \langle u',v'\rangle \tfrac{|P_n(\gamma_n)|}{|P_{\ell}(\gamma_{\ell})|}
   f_{\pi_{\Phi,\gamma,n};\,u,v}
\end{equation}
is the orthogonal projection dual to 
$\cH_{\pi_{\Phi,\gamma,n}}\hookrightarrow\cH_{\pi_{\Phi,\gamma,{\ell}}}$\,.
These maps sum over $(\gamma_n,\gamma_{\ell})$ to a Hilbert space projection 
$p_{\ell,n}=\bigl (\int_{\gamma_{\ell}\in \gs_{\Phi,\ell}}p_{\gamma,n,\ell} 
d\gamma' \bigr )$,
\begin{equation}\label{hilbert-proj}
p_{\ell,n}: L^2(N_{\Phi,\ell}) \to L^2(N_{\Phi,n}) \text{ for } \ell \geqq n.
\end{equation}
The maps (\ref{hilbert-proj}) define an inverse system in the category of
Hilbert spaces and partial isometries:
\begin{equation}\label{hilbert-proj-sys}
L^2(N_{\Phi,1}) \overset{p_{2,1}}{\longleftarrow} L^2(N_{\Phi,2})
	\overset{p_{3,2}}{\longleftarrow}  L^2(N_{\Phi,3})
        \overset{p_{4,3}}{\longleftarrow} \,\, ... \,\,\,\longleftarrow\,
        L^2(N_\Phi)
\end{equation}
where the projective limit 
$L^2(N_\Phi) := \varprojlim \{L^2(N_{\Phi,n}),p_{\ell,n}\}$ is taken in that 
category.  We now have the Hilbert space projective limit
\begin{equation}\label{hilbert-limit}
L^2(N_\Phi) := \varprojlim \{L^2(N_{\Phi,n}), p_{\ell,n}\}.
\end{equation}
Because of the renormalizations in (\ref{hilbert-proj0}), the elements of
$L^2(N_\Phi)$ do not have an immediate interpretation as functions on $N_\Phi$.
We address that problem by looking at the Schwartz space.
\medskip

The Schwartz space considerations of \cite[Section 5]{W2015} extend to
our setting with only obvious modifications, so we restrict our discussion to 
the relevant definitions and results.
\medskip

Given $\gamma = (\gamma_n) \in \mathfrak{t}_{\Phi,\infty}^*$ we have the
unitary character $\zeta_\gamma = \exp(2\pi i \gamma)$ on 
$S_{\Phi,\infty} = \varinjlim S_{\Phi,n}$\,.  Express 
$\zeta_\gamma = (\zeta_{\gamma_n}) \in \varprojlim \widehat{S_{\Phi,n}}$\,.
The corresponding
{\em relative Schwartz space} $\cC((N_{\Phi,n}/S_{\Phi,n}); \zeta_{\gamma_n})$
consists of all
functions $f \in C^\infty(N_{\Phi,n})$ such that
\begin{equation}\label{rel-schwartz-n}
\begin{aligned}
&f(xs) = \zeta_{\gamma_n}(s)^{-1}f(x) \text{ for } x \in N_{\Phi,n} \text{ and }
	s \in S_{\Phi,n}\,, \text{ and }
	|q(g)p(D)f| \text{ is bounded }\\
&\text{ on } N_{\Phi,n}/S_{\Phi,n} \text{ for all polynomials }
	p, q \text{ on } N_{\Phi,n}/S_{\Phi,n}
	\text{ and all } D \in \cU(\gn_{\Phi,n}).
\end{aligned}
\end{equation}	
The corresponding {\em limit Schwartz space} 
$\cC((N_{\Phi,\infty}/S_{\Phi,\infty});\zeta_\gamma) = \varprojlim 
\cC(N_{\Phi,n}/S_{\Phi,n}; \zeta_{\gamma_n})$, consisting of all functions
$f \in C^\infty(N_{\Phi,\infty})$ such that
\begin{equation}\label{lim-schwartz-n}
\begin{aligned}
&f(xs) = \zeta_\gamma(s)^{-1}f(x) \text{ for } x \in N_{\Phi,\infty} \text{ and }
        s \in S_{\Phi,\infty}\,, \text{ and }
        |q(g)p(D)f| \text{ is bounded }\\
&\text{ on } N_{\Phi,\infty}/S_{\Phi,\infty} \text{ for all polynomials }
        p, q \text{ on } N_{\Phi,\infty}/S_{\Phi,\infty}
        \text{ and all } D \in \cU(\gn_{\Phi,\infty}).
\end{aligned}
\end{equation}
As expected, $\cC(N_{\Phi,n}/S_{\Phi,n}; \zeta_{\gamma_n})$ is a nuclear 
Fr\' echet space and it is dense in 
$L^2(N_{\Phi,n}/S_{\Phi,n};\zeta_{\gamma_n})$, and \cite[Theorem 5.7]{W2015}
and its corollaries go through in our setting as follows.

\begin{theorem}\label{coef-sch}
Let $\gamma = (\gamma_n) \in \gt^*_{\Phi,\infty}$\,.  Let $n > 0$ and let 
$u$ and $v$ be $C^\infty$ vectors for the stepwise square integrable
representation $\pi_{\Phi,\gamma,n}$ of $N_{\Phi,n}$\,.
Then the coefficient function
$f_{\pi_{\Phi,\gamma,n};\,u,v}$ belongs to the relative Schwartz space
$\cC((N_{\Phi,n}/S_{\Phi,n});\zeta_{\gamma_n})$, and the coefficient
function $f_{\pi_{\Phi,\gamma,\infty};\,u,v}$ belongs to the limit
Schwartz space $\cC(N_{\Phi,\infty}/S_{\Phi,\infty};\zeta_\gamma)$.
\end{theorem}

\begin{corollary}\label{l1-coef1}
Let $\gamma = (\gamma_n) \in \gt^*_{\Phi,\infty}$\,.  Let $n > 0$ and let
$u$ and $v$ be $C^\infty$ vectors for the stepwise square integrable
representation $\pi_{\Phi,\gamma,n}$ of $N_{\Phi,n}$\,.
Then the coefficient function
$f_{\pi_{\Phi,\gamma,n};\,u,v} \in L^1(N_{\Phi,n}/S_{\Phi,n};\zeta_{\gamma_n})$,
and the coefficient function $f_{\pi_{\Phi,\gamma,\infty};\,u,v} \in
\varprojlim L^1(N_{\Phi,n}/S_{\Phi,n};\zeta_{\gamma_n})$.
\end{corollary}

In fact the argument shows

\begin{corollary}\label{l1-coef2}
Let $L$ be a connected simply connected nilpotent Lie group, $Z$ its center,
and $\lambda \in \gz^*$ such that $\pi_\lambda$ is a square integrable 
(mod $Z$) representation of $L$.  
Let $\zeta = e^{2\pi i \lambda} \in \widehat{Z}$ and let $u$ and $v$
be $C^\infty$ vectors for $\pi_\lambda$\,.  Then the coefficient 
$f_{\pi_\lambda;\,u,v} \in L^1(L/Z,\zeta_\lambda)$.
\end{corollary}

A norm $|\xi|$ on $\gn_{\Phi,n}$ corresponds to a norm
$||\exp(\xi)|| := ||\xi||$ on $N_{\Phi,n}$\,.  Thus classical Schwartz space
$\cC(\gn_{\Phi,n})$ on the real vector space $\gn_{\Phi,n}$, corresponds to the
Schwartz space $\cC(N_{\Phi,n})$\,, which thus is defined by seminorms
\begin{equation}\label{seminorm}
\nu_{k,D,n}(f) = {\sup}_{x \in N_{\Phi,n}} |(1 + |x|^2)^k (Df)(x)|.
\end{equation}
Here $k$ is a positive integer, and $D \in \mathcal{U}(\gn_{\Phi,n})$ is
a differential operator acting on the left on $N_{\Phi,n}$\,.
Since $\exp: \mathfrak{n}_{\Phi,n} \to N_{\Phi,n}$ is a polynomial diffeomorphism,
$f \mapsto f\cdot\exp$ is 
a topological isomorphism of $\cC(N_{\Phi,n})$ onto $\cC(\gn_{\Phi,n})$:
\begin{equation}\label{schwartz-n-n}
\cC(N_{\Phi,n}) = \{f \in C^\infty(N_{\Phi,n}) \mid
        f\circ \exp \in \cC(\gn_{\Phi,n})\}.
\end{equation}
We now define the Schwartz space 
\begin{equation}\label{schwartz-n-inf}
\cC(N_{\Phi,\infty}) = 
 \{f \in C^\infty(N_{\Phi,\infty}) \mid f|_{N_{\Phi,n}} \in \cC(N_{\Phi,n})
	\text{ for } n \ggg 0\} = \varprojlim \cC(N_{\Phi,n})
\end{equation}
where the inverse limit is taken in the category of complete locally convex
topological vector spaces and continuous linear maps.  Since $\cC(N_{\Phi,n})$
is defined by the seminorms (\ref{seminorm}), the same follows for $\cC(N_{\Phi,\infty})$.  In other words,
\begin{lemma}\label{lim-schwartz-prop-n}
The Schwartz space $\cC(N_{\Phi,\infty})$ consists of all functions
$f \in C^\infty(N_{\Phi,\infty})$ such that, for all $n > 0$, $\nu_{k,D,n}(f)$
is bounded for all integers $k > 0$ all $D \in \cU(\gn_{\Phi,n})$.
Here the seminorms $\nu_{k,D,n}$ are given by {\rm (\ref{seminorm})}.
\end{lemma}
\medskip

Every $f \in \cC(N_{\Phi,n})$ is a limit in $\cC(N_{\Phi,n})$ of 
finite linear combinations of the functions
$f_{\gamma_n}(x) = \int_{S_{\Phi,n}} f(xs)\zeta_{\gamma_n}(s)ds$
in $\cC((N_{\Phi,n}/S_{\Phi,n}),\zeta_{\gamma_n})$.
Specifically, denote $\varphi_x(\gamma_n) := f_{\gamma_n}(x)$.  Then
$\varphi_x$ is a multiple of the Fourier transform 
$\cF_{\Phi,n}(\ell(x)^{-1}f)|_{S_{\Phi,n}})$.
\medskip

The inverse Fourier transform
$\cF_{\Phi,n}^{-1}(\varphi_x)$ reconstructs $f$ from the $f_{\gamma_n}$.
Since the relative Schwartz space 
$\cC((N_{\Phi,}/S_{\Phi,n});\zeta_{\gamma_n})$
is dense in $L^2(N_{\Phi,n}/S_{\Phi,n},\zeta_{\gamma_n})$
and the the set of finite linear combinations of coefficients 
$f_{\pi_{\Phi,\gamma,n};\,u,v}$ (where $u, v$ are $C^\infty$ vectors) is dense 
in $\cC(N_{\Phi,n}/S_{\Phi,n},\zeta_{\gamma_n})$, now every
$f \in \cC(N_{\Phi,n})$ is a Schwartz wave packet along $\gs_{\Phi,n}^*$ of
coefficients of the various $\pi_{\Phi,\gamma,n}$\,, $u$ and $v$ smooth.
Now we combine the inverse system (\ref{hilbert-proj-sys}) and its Schwartz
space analog.
\begin{equation}\label{schwartz-proj-sys}
\begin{CD}
\cC(N_{\Phi,1}) @< q_{2,1} << \cC(N_{\Phi,2})
        @< q_{3,2} <<  \cC(N_{\Phi,3}) @< q_{4,3} << \dots @ <<< \cC(N_\Phi)
	= \varprojlim \cC(N_{\Phi,n}) \\
    @VV{r_1}V       @VV{r_2}V     @VV{r_3}V        @VVV  @VV{r_\infty}V \\
L^2(N_{\Phi,1}) @< p_{2,1} << L^2(N_{\Phi,2}) 
	@< p_{3,2} <<  L^2(N_{\Phi,3}) @< p_{4,3} << \dots @ <<< L^2(N_\Phi)
	= \varprojlim L^2(N_{\Phi,n})
\end{CD}
\end{equation}
The $r_n: \cC(N_{\Phi,n}) \hookrightarrow L^2(N_{\Phi,n})$ are continuous 
injections with dense image, so 
$r_\infty: \cC(N_{\Phi,\infty}) \hookrightarrow L^2(N_{\Phi,\infty})$ is a 
continuous injection with dense image.  Putting all this together as in the 
minimal parabolic case \cite[Section 5]{W2015}, we have proved

\begin{proposition} \label{dense-embedding}
Assume {\rm (\ref{lim-subgroups1})(a)}, so that $\{N_{\Phi,n}\}$ is a direct
system and $N_{\Phi,\infty} = \varinjlim N_{\Phi,n}$ is well defined.  Define 
$r_\infty: \cC(N_{\Phi,\infty}) \hookrightarrow L^2(N_{\Phi,\infty})$ as in 
the commutative diagram {\rm (\ref{schwartz-proj-sys})}.  Then 
$L^2(N_{\Phi,\infty})$ is a Hilbert space completion of $\cC(N_{\Phi,\infty})$.  
In particular $r_\infty$ defines a pre-Hilbert
space structure on $\cC(N_{\Phi,\infty})$ with completion $L^2(N_{\Phi,\infty})$.
\end{proposition}

As in \cite[Corollary 5.17]{W2015}, we note that  $\cC(N_{\Phi,\infty})$ is 
independent of the choices we made in the construction of 
$L^2(N_{\Phi,\infty})$, so

\begin{corollary}\label{l2-well-def}
The limit Hilbert space 
$L^2(N_{\Phi,\infty}) = \varprojlim \{L^2(N_{\Phi,n}), p_{\ell,n}\}$ of
{\rm (\ref{schwartz-proj-sys})} , and the left/right regular representation
of $N_{\Phi,\infty}\times N_{\Phi,\infty}$ on $L^2(N_{\Phi,\infty})$, are 
independent of the choice of vectors $e$ in {\rm (\ref{unitfactor})}.
\end{corollary}

\noindent Recall the notation

\begin{itemize}
\item $\mathfrak{t}_{\Phi,\infty}^* = \varprojlim \mathfrak{t}^*_{\Phi,n}$ 
consists of all $\gamma = (\gamma_n)$ where each 
$\gamma_n \in \mathfrak{t}_{\Phi,n}^*$ and
if $\ell \geqq n$ then $\gamma_{\ell}|_{\mathfrak{s}_{\Phi,n}} = \gamma_n$\,.
\item given $\gamma = (\gamma_n) \in \mathfrak{t}_{\Phi,\infty}^*$ the limit 
representation $\pi_{\Phi,\gamma,\infty} = \varprojlim \pi_{\Phi,\gamma_n}$ is 
constructed as in Section \ref{sec7},
\item The distribution characters $\Theta_{\pi_{\Phi,\gamma,n}}$ on the 
$N_{\Phi,n}$ are given by (\ref{def-dist-char}), and
\item $\mathcal{C}(N_{\Phi,\infty}) = \varprojlim \mathcal{C}(N_{\Phi,n})$ 
consists of all sets $f = (f_n)$ where each $f_n \in \mathcal{C}(N_{\Phi,n})$, 
and where if $\ell \geqq n$ then $f_{\ell}|_{N_{\Phi,n}} = f_n$\,.
\end{itemize}
Then as in the case of minimal parabolics \cite[Section 6]{W2015}, 
the limit Fourier inversion formula is
\begin{theorem}\label{limit-inversion}
Suppose that $N_{\Phi,\infty} = \varinjlim N_{\Phi,n}$ where $\{N_{\Phi,n}\}$ 
satisfies {\rm (\ref{newsetup})}.  
Let $f = (f_n) \in \mathcal{C}(N_{\Phi,\infty})$ and $x \in N_{\Phi,\infty}$\,.
Then $x \in N_{\Phi,n}$ for some $n$ and
\begin{equation}\label{lim-inv-formula}
f(x) = c_n\int_{\mathfrak{t}_{\Phi,n}^*} \Theta_{\pi_{\Phi,\gamma,n}}(r_xf)
        |\Pf_{\mathfrak{n}_{\Phi,n}}({\gamma_n})|d\gamma_n
\end{equation}
where $c_n = 2^{d_1 + \dots + d_m} d_1! d_2! \dots d_m!$ as in {\rm (\ref{c-d}a)}
and $m$ is the number of factors $L_r$ in $N_{\Phi,n}$.
\end{theorem}

\begin{proof}  Apply Theorem \ref{plancherel-general} to $N_{\Phi,n}$:
$f(x) = f_n(x) = c_n\int_{\gt_{\Phi,n}^*} 
\Theta_{\pi_{\Phi,\gamma,n}}(r_xf) |\Pf_{\gn_{\Phi,n}}({\gamma_n})|\, d\gamma_n$\,.
\end{proof}

\section{Representations of the Limit Groups II: 
	$A_{\Phi,\infty} N_{\Phi,\infty}$}
\label{sec10}
\setcounter{equation}{0}
We extend some of the results of Section \ref{sec9} to the maximal 
exponential (locally) solvable subgroup $A_{\Phi,\infty} N_{\Phi,\infty}$\,.
\medskip

The first step is to locate the $A_{\Phi,\infty}$-stabilizer of a limit
square integrable representation $\pi_\gamma$ of
$N_{\Phi,\infty}$\,.  Following (\ref{a-stab}) we set
\begin{equation}\label{a-stab-inf}
A'_{\Phi,\infty} = \{\exp(\xi) \mid \xi \in \ga_{\Phi,\infty} 
	\text{ and every } \beta_j(\xi) = 0\}.
\end{equation}
\begin{lemma}\label{lim-a-stab}
If $\gamma = (\gamma_n) \in \gt^*_{\Phi,\infty}$ then $A'_{\Phi,\infty}$
is the stabilizer of $\pi_\gamma$ in $A_{\Phi,\infty}$\,.
\end{lemma}
\begin{proof}  Recall the $J''_r$ from Lemma \ref{inter-compl}.  Then
Lemma \ref{semidirect} tells us that, for each $r_0$\,, 
$\gl_{\Phi,r_0}$ has center
$$
\gz_{\Phi,r_0} = {\sum}_{\beta_r|{\ga_\Phi} = \beta_{r_0}|{\ga_\Phi}}
	\left ( \gg_{\beta_r} + {\sum}_{J''_r}\gg_\alpha \right ),
$$ 
and 
Lemma \ref{inter-compl} then says that $\gz_{\Phi,r}$ is an $\ad(\ga_\Phi)$
eigenspace on $\gg$.  Thus the $\ad^*(\ga_\Phi)$-stabilizer of $\gamma$
is given by $\beta_r(\ga_{\Phi,\infty}) = 0$ for all $r$.
\end{proof}

Lemma \ref{traces} shows that our methods cannot yield a
Dixmier-Puk\'anszky operator for $A_{\Phi,\infty} N_{\Phi,\infty}$
nor for $U_{\Phi,\infty}A_{\Phi,\infty} N_{\Phi,\infty}$\,,
but we do have such operators $D_n$ for the $A_{\Phi,n}N_{\Phi,n}$ and
the $U_{\Phi,n}A_{\Phi,n}N_{\Phi,n}$\,.  
\medskip

Let $\gamma = (\gamma_n) \in \gt_{\Phi,\infty}^*$\,.  Then 
$\pi_{\Phi,\gamma,\infty}$ extends from $N_{\Phi,\infty}$  to
a representation $\pi^\dagger_{\Phi,\gamma,\infty}$ of
$A'_{\Phi,\infty} N_{\Phi,\infty}$ with the same representation space, 
because every $\pi_{\Phi,\gamma,n}$ extends that way from 
$N_{\Phi,n}$ to $A'_{\Phi,n}N_{\Phi,n}$.  The
representations of $A'_{\Phi,n}N_{\Phi,n}$ corresponding to $\gamma_n$ are the
$\exp(2\pi i\xi|_{\ga'_{\Phi,n}}) \otimes \pi^\dagger_{\Phi,\gamma,n}$\,.  The 
representation of $A_{\Phi,\infty}N_{\Phi,\infty}$ and the 
$A_{\Phi,n}N_{\Phi,n}$\,, corresponding to $\gamma$ and $\xi = (\xi_n) \in
(\ga'_{\Phi,\infty})^*$, is the 
\begin{equation}\label{repan'}
\pi_{\Phi,\gamma,\xi,\infty} := \varinjlim \pi_{\Phi,\gamma,\xi,n}
	\text{ where } 
\pi_{\Phi,\gamma,\xi,n} =
\Ind_{A'_{\Phi,n}N_{\Phi,n}}^{A_{\Phi,n}N_{\Phi,n}}\,
\bigl (\exp(2\pi i\xi_n) \otimes 
\pi^\dagger_{\Phi,\gamma,n} \bigr ).
\end{equation} 
If $\dim(A_{\Phi,\infty}/A'_{\Phi,\infty}) < \infty$ then
(\ref{repan'}) becomes $\pi_{\Phi,\gamma,\xi,\infty} = 
\Ind_{A'_{\Phi,\infty} N_{\Phi,\infty}}^{A_{\Phi,\infty} N_{\Phi,\infty}}
(\exp(2\pi i\xi)\otimes \pi^\dagger_{\Phi,\gamma,\infty})$\,, 
because then one can integrate
over $A_{\Phi,\infty}/A'_{\Phi,\infty}$.  Or in general one may view   
(\ref{repan'}) as an interpretation of $\pi_{\Phi,\gamma,\xi,\infty} =
\Ind_{A'_{\Phi,\infty} N_{\Phi,\infty}}^{A_{\Phi,\infty} N_{\Phi,\infty}}
(\exp(2\pi i\xi)\otimes \pi^\dagger_{\Phi,\gamma,\infty})$\,.

\begin{lemma}\label{equiv-a-inf}
Let $\gamma, \gamma' \in \gt^*_{\Phi,\infty}$\,.  Then the
representations $\pi_{\Phi,\gamma,\xi,\infty}$ and 
$\pi_{\Phi,\gamma',\xi',\infty}$ are equivalent if and only if
both $\xi' = \xi$ and $\gamma' \in \Ad^*(A_{\Phi,\infty})(\gamma)$.
Express $\gamma = (\gamma_n)$ with 
$\gamma_n = \sum_{j=1}^{m_n} \gamma_{n,j}$ where
$\gamma_{n,j} \in \gz_j^*$\,.  Then $\Ad^*(A_{\Phi,\infty})(\gamma)$ 
consists of all $\left (\sum_{j=1}^{m_n} c_j\gamma_{n,j}\right )$ with 
every $c_j > 0$.
\end{lemma}

\begin{proof}  The Mackey little group method implies 
$\pi_{\Phi,\gamma,\xi,n} \simeq \pi_{\Phi,\gamma',\xi',n}$
just when $\xi' = \xi$ and $\gamma'_n \in \Ad^*(A_{\Phi,n})(\gamma_n)$.
The first assertion follows.  The second is because the action of
$\Ad^*(A_{\Phi,\infty})$ on $\gamma_{n,j}$ is multiplication by an arbitrary
positive real $c_j = \exp(i\beta_j(\alpha))$ for 
$\alpha \in \ga_{\Phi,\infty}$.
\end{proof}

The representation space $\cH_{\pi_{\Phi,\gamma,\xi,\infty}}$ of
$\pi_{\Phi,\gamma,\xi,\infty} \in 
(A_{\Phi,\infty} N_{\Phi,\infty})\,\,\widehat{}$\,\, is the same as that of  
$N_{\Phi,\infty}$, except for the unitary character $\exp(2\pi i\xi)$.
We thus obtain
$$
L^2((A_{\Phi,n}N_{\Phi,n}/A'_{\Phi,n}S_{\Phi,n}); (\exp(2\pi i\xi)\otimes 
\zeta_n)) \cong (\cH_{\pi_{\Phi,\gamma,\xi,n}} \widehat{\otimes}
\cH^*_{\pi_{\Phi,\gamma,\xi,n}})\,.
$$
Summing over $\gt_{\Phi,n}^*$ and $\ga_{\Phi,n}/\ga'_{\Phi,n}$ we 
proceed as in Section \ref{sec9}; then 
$$
L^2(A_{\Phi,n}N_{\Phi,n})
= \int_{\ga_{\Phi,n}/\ga'_{\Phi,n}}\int_{\gt_{\Phi,n}^*} 
  (\cH_{\pi_{\Phi,\gamma,\xi,n}} \widehat{\otimes}
  \cH^*_{\pi_{\Phi,\gamma,\xi,n}}) |P_n(\gamma_n)|d\gamma_n\,d\xi
$$
so, as in (\ref{hilbert-proj-sys}) and (\ref{hilbert-limit}),
$$
L^2(A_{\Phi,\infty}N_{\Phi,\infty}) = \varinjlim 
\int_{\ga_{\Phi,n}/\ga'_{\Phi,n}}\int_{\gt_{\Phi,n}^*} 
  (\cH_{\pi_{\Phi,\gamma,\xi,n}} \widehat{\otimes}
  \cH^*_{\pi_{\Phi,\gamma,\xi,n}}) |P_n(\gamma_n)|d\gamma_n\,d\xi\,.
$$
Since the base spaces of the unitary line bundles
$$
A_{\Phi,n}N_{\Phi,n} \to A_{\Phi,n}N_{\Phi,n}/A'_{\Phi,n}S_{\Phi,n}
\text{ and } N_{\Phi,n} \to N_{\Phi,n}/S_{\Phi,n}
$$
are similar, we modify (\ref{rel-schwartz-n}) for the {\em relative 
Schwartz space} $\cC((A_{\Phi,n}N_{\Phi,n}/A'_{\Phi,n}S_{\Phi,n});
\exp(2\pi i\xi)\otimes \zeta_{\gamma_n})$ to consist of all functions
$f \in C^\infty(A_{\Phi,n}N_{\Phi,n})$ such that
\begin{equation}\label{rel-schwartz-an}
\begin{aligned}
&f(xas) = \exp(-2\pi i\xi(\log a))\zeta_{\gamma_n}(s)^{-1}f(x)\,\,
	(x \in N_{\Phi,n}, a \in A'_{\Phi,n}\,,
        s \in S_{\Phi,n}), \text{ and }
        |q(g)p(D)f| \text{ is}\\
&\text{ bounded on } A_{\Phi,n}N_{\Phi,n}/A'_{\Phi,n}S_{\Phi,n} \text{ for all 
	polynomials }
        p, q \text{ on } N_{\Phi,n}/S_{\Phi,n}
        \text{ and all } D \in \cU(\gn_{\Phi,n}).
\end{aligned}
\end{equation}
The corresponding {\em limit relative Schwartz space} is
$$
\cC((A_{\Phi,\infty}N_{\Phi,\infty}/A'_{\Phi,\infty}S_{\Phi,\infty});
	(\exp(2\pi i\xi)\otimes \zeta_\gamma)) = \varprojlim
\cC((A_{\Phi,n}N_{\Phi,n}/A'_{\Phi,n}S_{\Phi,n}); 
(\exp(2\pi i\xi)\otimes\zeta_{\gamma_n})),
$$ 
consisting of all functions
$f \in C^\infty(A_{\Phi,\infty}N_{\Phi,\infty})$ such that
\begin{equation}\label{lim-schwartz-an}
\begin{aligned}
&f(xas) = \exp(-2\pi i\xi(\log a))\zeta_\gamma(s)^{-1}f(x) \text{ for } 
	x \in N_{\Phi,\infty}, a \in A'_{\Phi,\infty} \text{ and }
        s \in S_{\Phi,\infty}\,, \text{ and }\\
&|q(g)p(D)f| \text{ is bounded } \text{ on } 
	A_{\Phi,\infty}N_{\Phi,\infty}/A'_{\Phi,\infty}S_{\Phi,\infty} \\
&\text{for all polynomials }
        p, q \text{ on } N_{\Phi,\infty}/S_{\Phi,\infty}
        \text{ and all } D \in \cU(\gn_{\Phi,\infty}).
\end{aligned}
\end{equation}

Theorem \ref{coef-sch} and Corollaries \ref{l1-coef1} and \ref{l1-coef2}
hold for our groups $A_{\Phi,\bullet}N_{\Phi,\bullet}$ here with essentially no change, 
so we will not repeat them.
\medskip

We use Casselman's extension \cite[p. 4]{C1989} of the classical definition 
for seminorms and Schwartz space (which we used for $\gn_{\Phi,n}$).
First, we have seminorms on the $\ga_{\Phi,n}+\gn_{\Phi,n}$ as in 
(\ref{seminorm}) as follows.  Fix a continuous norm $||\varphi||$ on
$A_{\Phi,n}N_{\Phi,n}$ such that 
\begin{equation}\label{norm}
||1_{A_{\Phi,n}N_{\Phi,n}}|| = 1, ||x|| = ||x^{-1}|| \geqq 1 \text{ for all }
x, \text{ and } ||x||/||y|| \leqq ||xy|| \leqq ||x||\, ||y|| \text{ for all }
x, y.
\end{equation}
That gives seminorms
\begin{equation}\label{seminorm-an}
\nu_{k,D,n}(f) = {\sup}_{x \in A_{\Phi,n}N_{\Phi,n}}||x||^k|Df(x)|
\text{ where } k > 0 \text{ and } D \in \cU(\ga_{\Phi,n}+\gn_{\Phi,n}).
\end{equation}
That defines the Schwartz space $\cC(A_{\Phi,n}N_{\Phi,n})$ as in 
(\ref{schwartz-n-n}): 
\begin{equation}\label{schwartz-an-n}
\cC(A_{\Phi,n}N_{\Phi,n}) = \{f \in C^\infty(A_{\Phi,n}N_{\Phi,n}) \mid
\nu_{k,D,n}(f) < \infty \text{ for } 
k > 0 \text{ and } D \in \cU(\ga_{\Phi,n}+\gn_{\Phi,n})\}.
\end{equation}
Finally 
we define $\cC(A_{\Phi,\infty}N_{\Phi,\infty})$
to be the inverse limit in the category of locally convex topological
vector spaces and continuous linear maps, as in (\ref{schwartz-n-inf}):
\begin{equation}\label{schwartz-an-inf}
\cC(A_{\Phi,\infty}N_{\Phi,\infty}) = \{f \in C^\infty(A_{\Phi,\infty}N_{\Phi,\infty})
\mid f|_{A_{\Phi,n}N_{\Phi,n}} \in \cC(A_{\Phi,n}N_{\Phi,n})\}
= \varprojlim \cC(A_{\Phi,n}N_{\Phi,n}).
\end{equation}

\begin{lemma}\label{lim-schwartz-prop-an}{\rm (\cite[Proposition 1.1]{C1989})}
The Schwartz space $\cC(A_{\Phi,\infty}N_{\Phi,\infty})$ consists of all 
functions $f \in C^\infty(A_{\Phi,\infty}N_{\Phi,\infty})$ such that, for all
$n > 0$, $\nu_{k,D,n}(f) < \infty$
for all integers $k > 0$ and all 
$D \in \cU(\ga_{\Phi,n}+\gn_{\Phi,n})$.
Here $\nu_{k,D,n}$ is given by {\rm (\ref{seminorm-an})}.
The $\cC(A_{\Phi,n}N_{\Phi,n})$ are nuclear Fr\' echet spaces and
$\cC(A_{\Phi,\infty}N_{\Phi,\infty})$ is an LF space.
The left/right actions of $(A_{\Phi,n}N_{\Phi,n} \times A_{\Phi,n}N_{\Phi,n})$
on $\cC(A_{\Phi,n}N_{\Phi,n})$ and of $(A_{\Phi,\infty}N_{\Phi,\infty}\times
A_{\Phi,\infty}N_{\Phi,\infty})$ on $\cC(A_{\Phi,\infty}N_{\Phi,\infty})$ are
continuous.
\end{lemma}

As for the $N_{\Phi,n}$, if $f \in \cC(A_{\Phi,n}N_{\Phi,n})$ it is a limit in
$\cC(A_{\Phi,n}N_{\Phi,n})$ of finite linear combinations of the 
$f_{\xi,\gamma,n}(x) = \int_{A_{\Phi,n}}\int_{S_{\Phi,n}} 
f(xas)\exp(2\pi i\xi(\log a))
\zeta_{\gamma_n}(s)dsda$ in $\cC((A_{\Phi,n}N_{\Phi,n}/A'_{\Phi,n}S_{\Phi,n}),
(\exp(2\pi i\xi)\zeta_{\gamma_n}))$.
Specifically, denote $\varphi_x(\xi,\gamma_n) := f_{\xi,\gamma_n}(x)$.  Then
$\varphi_x$ is a multiple of the classical Fourier transform
$\cF_{\Phi,n}(\ell(x)^{-1}f)|_{A_{\Phi,n}S_{\Phi,n}})$, and the
inverse Fourier transform $\cF_{\Phi,n}^{-1}(\varphi_x)$ reconstructs $f$ from 
the $f_{\xi,\gamma_n}$. 
$\cC((A_{\Phi,n}N_{\Phi,n}/A'_{\Phi,n}S_{\Phi,n}), (\exp(2\pi i\xi)\zeta_{\gamma_n}))$
is dense 
in $L^2((A_{\Phi,n}N_{\Phi,n}/A'_{\Phi,n}S_{\Phi,n}), (\exp(2\pi i\xi)\zeta_{\gamma_n}))$.
The finite linear combinations of coefficients
of the $\pi_{\Phi,\gamma,\xi,n}$ along $C^\infty$ vectors form dense
subset of \newline
$\cC((A_{\Phi,n}N_{\Phi,n}/A'_{\Phi,n}S_{\Phi,n}),(\exp(2\pi i\xi)\zeta_{\gamma_n}))$.  So every
$f \in \cC(A_{\Phi,n}N_{\Phi,n})$ is a Schwartz wave packet along
$\ga_{\Phi,n}^* + \gs_{\Phi,n}^*$ of coefficients of the various 
$\pi_{\Phi,\gamma,\xi,n}$, and the corresponding inverse systems
fit together as in (\ref{schwartz-proj-sys}):
\begin{equation}\label{schwartz-proj-sys-an}
\begin{CD}
\cC(A_{\Phi,1}N_{\Phi,1}) @< q_{2,1} << \cC(A_{\Phi,2}N_{\Phi,2})
        @< q_{3,2} <<  \dots @ <<< 
	\cC(A_\Phi N_\Phi) = \varprojlim \cC(A_{\Phi,n}N_{\Phi,n}) \\
    @VV{r_1}V       @VV{r_2}V          @VVV  @VV{r_\infty}V \\
L^2(A_{\Phi,1}N_{\Phi,1}) @< p_{2,1} << L^2(A_{\Phi,2}N_{\Phi,2})
        @< p_{3,2} <<  \dots @ <<< 
	L^2(A_\Phi N_\Phi) = \varprojlim L^2(A_{\Phi,n}N_{\Phi,n})
\end{CD}
\end{equation}
Again, the $r_n$ are continuous injections with dense image, so 
$r_\infty: \cC(A_{\Phi,\infty}N_{\Phi,\infty}) \hookrightarrow 
L^2(A_{\Phi,\infty}N_{\Phi,\infty})$ is a continuous injection 
with dense image.  As in Proposition \ref{dense-embedding} we conclude

\begin{proposition} \label{dense-embedding-an}
Assume {\rm (\ref{lim-subgroups1})(a)}, 
so that $\{A_{\Phi,n}N_{\Phi,n}\}$ is a direct system and 
the limit group $A_{\Phi,\infty}N_{\Phi,\infty}$ = 
$\varinjlim A_{\Phi,n}N_{\Phi,n}$ is well defined.  Define
$r_\infty: \cC(A_{\Phi,\infty}N_{\Phi,\infty}) 
\hookrightarrow L^2(A_{\Phi,\infty}N_{\Phi,\infty})$ as in
the commutative diagram {\rm (\ref{schwartz-proj-sys-an})}.  Then
$L^2(A_{\Phi,\infty}N_{\Phi,\infty})$ is a Hilbert space completion of 
$\cC(A_{\Phi,\infty}N_{\Phi,\infty})$.
In particular $r_\infty$ defines a pre-Hilbert
space structure on $\cC(A_{\Phi,\infty}N_{\Phi,\infty})$ with completion 
$L^2(A_{\Phi,\infty}N_{\Phi,\infty})$.
\end{proposition}

\begin{corollary}\label{l2-well-def-an}
The limit Hilbert space $L^2(A_{\Phi,\infty}N_{\Phi,\infty}) = 
\varprojlim \{L^2(A_{\Phi,n}N_{\Phi,n}), p_{\ell,n}\}$ of
{\rm (\ref{schwartz-proj-sys-an})} , and the left/right regular representation
of $(A_{\Phi,\infty}N_{\Phi,\infty})\times (A_{\Phi,\infty}N_{\Phi,\infty})$ 
on $L^2(A_{\Phi,\infty}N_{\Phi,\infty})$, are independent of the choice of 
$C^\infty$ unit vectors $e$ in the inclusions 
$
\cH_{\pi,\gamma,\xi,n} \hookrightarrow \cH_{\pi,\gamma,\xi,\ell}\,, 
\ell \geqq n, \text{ by } v \mapsto v \otimes e.
$
\end{corollary}

The distribution characters $\Theta_{\pi_{\Phi,\gamma,\xi,n}} = 
\exp(2\pi i\xi)\Theta_{\pi_{\Phi,\gamma,n}}$ where 
$\Theta_{\pi_{\Phi,\gamma,n}}$ is given
by (\ref{def-dist-char}). The limit Schwartz space
$\mathcal{C}(A_{\Phi,\infty}N_{\Phi,\infty}) 
= \varprojlim \mathcal{C}(A_{\Phi,n}N_{\Phi,n})$ consists of all 
$f = (f_n)$ where each $f_n~\in~\mathcal{C}(A_{\Phi,n}N_{\Phi,n})$.

As in the case of minimal parabolics \cite[Section 6]{W2015},
the limit Fourier inversion formula is
\begin{theorem}\label{limit-inversion-an}
Suppose that $(\Psi_n \setminus \Phi_n) \subset 
(\Psi_\ell \setminus \Phi_\ell)$ for
$\ell \geqq n$, so that $A_{\Phi,\infty}N_{\Phi,\infty} =
\varinjlim A_{\Phi,n}N_{\Phi,n}$ is well defined.
Let $D_n$ be a Dixmier-Puk\'anszky operator for $A_{\Phi,n}N_{\Phi,n}$\,.
Let $f = (f_n) \in \mathcal{C}(A_{\Phi,n}N_{\Phi,n})$ and 
$x \in A_{\Phi,\infty}N_{\Phi,\infty}$\,.
Then $x \in A_{\Phi,n}N_{\Phi,n}$ for some $n$ and
\begin{equation}\label{lim-inv-formula-an}
f(x) = c_n\int_{\xi \in (\ga'_{\Phi,n})^*}
	\int_{\mathfrak{s}_{\Phi,n}^*/\Ad^*(A_{\Phi,n})} 
	\Theta_{\pi_{\Phi,\gamma,\xi,n}}(D_n(r_xf))
        |\Pf_{\mathfrak{n}_n}({\gamma_n})|d\gamma_n d\xi
\end{equation}
where $c_n = (\tfrac{1}{2\pi})^{\dim \ga'_\Phi/2}\,\,
2^{d_1 + \dots + d_{m_n}} d_1! d_2! \dots d_{m_n}!$ as in 
{\rm (\ref{c-d}a)} and $m_n$ is the number of factors $L_r$ in $N_{\Phi,n}$.
\end{theorem}

\begin{proof}  Apply Theorem \ref{planch-an} to $A_{\Phi,n}N_{\Phi,n}$.  
\end{proof}

\section{Representations of the Limit Groups III: $U_{\Phi,\infty}$}
\label{sec11}
\setcounter{equation}{0}
We are going to study highest weight limit representations of
$U_{\Phi,\infty} = \varinjlim U_{\Phi,n}$.  These are the 
representations for which there is an explicit Peter-Weyl Theorem
\cite[Theorem 4.3]{W2009}.  We restrict our attention to highest
weight representations of $U_{\Phi,\infty}$ for a good reason:
as noted in papers (\cite{SV1975}, \cite{SV1978}, \cite{SV1982}  and
\cite{V1972}) of Str\u atil\u a and Voiculescu, irreducible
unitary representations of $U(\infty)$ and other lim-compact groups 
can be extremely complicated, even of Type III.  This is summarized in 
\cite[Section 9]{W2014a}.
\medskip

Recall from Section \ref{sec8} that $G_\ell$ propagates $G_n$ for 
$\ell \geqq n$.  In particular $\Phi = (\Phi_n)$ where $\Phi_n$ is
the simple root system for $(\gm_{\Phi,n} + \ga_{\Phi,n})_\C$ and
$\Phi_n \subset \Phi_\ell$ for $\ell \geqq n$.  It is implicit that
the maximal compact subgroups $K_m \subset G_m$ satisfy $K_n \subset K_\ell$
for $\ell \geqq n$, so $K := \varinjlim\, K_n$ is a maximal lim--compact
subgroup of $G := \varinjlim\, G_n$\,.  We decompose the Cartan
subalgebras of $\gg_{\Phi,n}$ and $\gg_{\Phi,\infty}$ along the lines of
the proof of Lemma \ref{components}, as follows:
\begin{equation}\label{chop-cartan}
\gh_n = \gc_{\Phi,n} + \gb_{\Phi,n} + \ga_{\Phi,n} \text{ and }
\gh_\infty = \gc_{\Phi,\infty} + \gb_{\Phi,\infty} + \ga_{\Phi,\infty}
\end{equation}
where $\ga_{\Phi,n}$ is as before, $\gc_{\Phi,n} + \gb_{\Phi,n}$ is a
Cartan subalgebra of $\gm_{\Phi,n}$\,, $\gb_{\Phi,n} + \ga_{\Phi,n} = \ga_n$\,, 
and $\gc_{\Phi,n}$ is a Cartan
subalgebra of $\gu_{\Phi,n}$\,.  Then $\ga_{\Phi,\infty} = \varinjlim
\ga_{\Phi,n}$ as before, $\gb_{\Phi,\infty} = \varinjlim \gb_{\Phi,n}$\,,
and $\gc_{\Phi,\infty} = \varinjlim \gc_{\Phi,n}$\,.  Further,
$\ga_n = \gb_{\Phi,n} + \ga_{\Phi,n}$ and $\gc_{\Phi,n} = \gh_n \cap \gk_n$\,.
Notice that $C_{\Phi,n} := \exp(\gc_{\Phi,n})$ is a maximal torus in 
$U_{\Phi,n}^0$\,.
\medskip

We define a simple root system for $\gu_{\Phi,n}$ along the lines of
an idea of Borel and de Siebenthal.  Let $\{\gm^{(i)}_{\Phi,n}\}$ be 
the simple ideals in $\gm_{\Phi,n}$ and let $\{\Phi_n^{(i)}\}$ denote
the corresponding subsets of $\Phi_n$\,.  If every root in $\Phi_n^{(i)}$
is compact we set $\Sigma_n^{(i)} = \Phi_n^{(i)}$\,.  Otherwise 
$\Phi_n^{(i)}$ contains just one noncompact root, say $\alpha_n^{(i)}$\,.
let $\beta_n^{(i)}$ denote the maximal root of $\gm^{(i)}_{\Phi,n}$\,.
If $\alpha_n^{(i)}$ has coefficient $1$ as a summand of $\beta_n^{(i)}$
we set $\Sigma_n^{(i)} = \Phi_n^{(i)}\setminus \{\alpha_n^{(i)}\}$.
If it has coefficient $2$ as a summand of $\beta_n^{(i)}$ we set
$\Sigma_n^{(i)} = (\Phi_n^{(i)}\setminus \{\alpha_n^{(i)}\}) \cup
\{-\beta_n^{(i)}\}$.  Now $\Sigma_n:= \bigcup \Sigma_n^{(i)}$ is a
simple root system $\gu_{\Phi,n}$ and for its semisimple part 
$[\gu_{\Phi,n},\gu_{\Phi,n}]$.
\smallskip

\begin{lemma} \label{u-propagate}
If $\ell \geqq n$ then $\Sigma_{\Phi,n} \subset \Sigma_{\Phi,\ell}$\,.
Thus $\Sigma_\Phi := \bigcup \Sigma_{\Phi,n}$ is a simple root system for 
the semisimple part $[\gu_{\Phi,\infty},\gu_{\Phi,\infty}] := \varinjlim 
[\gu_{\Phi,n},\gu_{\Phi,n}]$ of $\gu_{\Phi,\infty}$\,.
\end{lemma}

\begin{proof} If $\alpha \in \Sigma_{\Phi,n}$ is not simple as a root
of $\gu_{\Phi,n+1}$, then as a linear combination of roots in $\Phi_{n+1}$
it must involve a root from $\Phi_{n+1}\setminus \Phi_n$\,, which 
contradicts 
$\alpha \in \Delta((\gm_{\Phi,n})_\C,(\gc_{\Phi,n}+\gb_{\Phi,n})_\C)$.
\end{proof}

As in Lemma \ref{components} we define $F_n = 
\exp(i\ga_n)\cap K_n$\,.
It is an elementary abelian $2$-subgroup of $U_{\Phi,n}$\,, central in both
$U_{\Phi,n}$ and $M_{\Phi,n}$\,, and has the properties
$$
U_{\Phi,n} = F_n U_{\Phi,n}^0\,, M_{\Phi,n} = 
F_n M_{\Phi,n}^0\,, \text{ and } E_{\Phi,n} = F_n E_{\Phi,n}^0\,.
$$
Further, $F_n C_{\Phi,n}$ is a Cartan subgroup of $U_{\Phi,n}$\,.
Passing to the limit, we define
$$
F = \varinjlim F_n = \exp(i\ga)\cap K\,,
\text{ and } C_{\Phi,\infty} = \varinjlim C_{\Phi,n}
$$
so that
$$
U_{\Phi,\infty} = F U_{\Phi,\infty}^0 \text{\,, and } F C_{\Phi,\infty}
\text{ is a lim-compact Cartan subgroup of } U_{\Phi,\infty}\,.
$$
\begin{definition}\label{dom-int-lim}
{\rm Let $\lambda_n \in \gc_{\Phi,n}^*$\,.  Then $\lambda_n$ is
{\em integral} if $\exp(2\pi i\lambda_n)$ is a well defined unitary character 
on the torus $C_{\Phi,n}$\,, and $\lambda_n$ is {\em dominant integral} 
if it is integral and $\langle \lambda_n,\alpha \rangle \geqq 0$ for every 
$\alpha \in \Sigma_{\Phi,n}$\,.  Write $\Lambda_{\Phi,n}$ for the set of
dominant integral weights in $\gc_{\Phi,n}^*$\,.

Let $\lambda = (\lambda_n) \in \gc_{\Phi,\infty}^*$\,.  Then $\lambda$ is 
{\em integral} if $\exp(2\pi i\lambda)$ is a well defined unitary character on 
the torus $C_{\Phi,\infty}$\,, in other words if each $\lambda_n$ is integral.
And $\lambda$ is {\em dominant integral} if it is integral and 
$\langle \lambda,\alpha \rangle \geqq 0$ for every $\alpha \in \Sigma_\Phi$\,,
in other words if each $\lambda_n$ is dominant integral.  Write 
$\Lambda_{\Phi,\infty}$
for the set of all dominant integral weights in $\gc_{\Phi,\infty}^*$\,.
}\hfill $\diamondsuit$
\end{definition}

Each $\lambda_n \in \Lambda_{\Phi,n}$ is the highest weight of an 
irreducible unitary representation $\mu_{\lambda,n}$ of $U_{\Phi,n}^0$\,.
Let $\cH_{\lambda_n}$ denote the representation space and $u_{\lambda,n}$
a highest weight unit vector.  Now let 
$\lambda = (\lambda_n) \in \Lambda_{\Phi,\infty}$\,.
Then $u_{\lambda,n} \mapsto u_{\lambda,\ell}$ defines a 
$U_{\Phi,n}^0$-equivariant isometric injection $\cH_{\lambda_n}
\hookrightarrow \cH_{\lambda_\ell}$\,.  Thus $\lambda$ defines a direct
limit highest weight unitary representation
$$
\mu_\lambda = \varinjlim \mu_{\lambda,n} \in \widehat{U_\Phi^0}
\text{ with representation space } \cH_\lambda = \varinjlim \cH_{\lambda_n}\,.
$$
Different choices of $\{u_{\lambda,n}\}$ lead to equivalent representations.
Here recall \cite[Theorem 5.10]{OW2014} that if $\ell \geqq n$ then
$\mu_{\lambda,\ell}|_{U_{\Phi,n}^0}$ contains $\mu_{\lambda,n}$ with
multiplicity $1$, so there is no ambiguity (beyond phase changes 
$u_{\lambda,n} \mapsto e^{i\epsilon_n}u_{\lambda,n}$) about the inclusion
$\cH_{\lambda_n} \hookrightarrow \cH_{\lambda_\ell}$\,.
Now denote
\begin{equation}\label{coherent-lim-u}
\begin{aligned}
&\Xi_{\Phi,n} = \left \{\mu_{\lambda,n,\varphi} := 
	\varphi \otimes \mu_{\lambda,n} \left | 
	\varphi \in \widehat{F}\,,
	\lambda_n \in \Lambda_{\Phi,n} \text{ and } 
	\varphi|_{F\cap U_{\Phi,n}^0} 
	= \mu_{\lambda,n}|_{F\cap U_{\Phi,n}^0}\right . \right \} \,,\\
&\Xi_{\Phi,\infty} = \left \{\mu_{\lambda,\varphi} :=
	\varphi \otimes \mu_\lambda \left | 
        \varphi \in \widehat{F}\,, 
	\lambda = (\lambda_n) \in \Lambda_{\Phi,\infty} \text{ and }
	\varphi|_{F\cap U_\Phi^0}
	= \mu_\lambda |_{F\cap U_\Phi^0} \right . \right \} \,.
\end{aligned}
\end{equation}

Lemma \ref{u-propagate} shows that the direct system $\{U_{\Phi,n}^0\}$ is
strict and is
parabolic in the sense of \cite[Eq. 4.2]{W2009}.  Thus we have the Peter-Weyl
Theorem for parabolic direct limits \cite[Theorem 4.3]{W2009}, and it follows
immediately for the system $\{U_{\Phi,n}\}$\,.  Rescaling matrix coefficients 
with the Frobenius-Schur orthogonality relations as in (\ref{rescale1}) and
(\ref{hilbert-proj0}) we obtain Hilbert space projections
$p_{\ell,n}: L^2(U_{\Phi,\ell}) \to L^2(U_{\Phi,n})$ and an inverse system
\begin{equation}\label{hilbert-proj-sys-u}
L^2(U_{\Phi,1}) \overset{p_{2,1}}{\longleftarrow} L^2(U_{\Phi,2})
        \overset{p_{3,2}}{\longleftarrow}  L^2(U_{\Phi,3})
        \overset{p_{4,3}}{\longleftarrow} \,\, ... \,\,\,\longleftarrow\,
        L^2(U_{\Phi,\infty})
\end{equation}
in the category of Hilbert spaces and projections,
where the projective limit
$L^2(U_{\Phi,\infty}) := \varprojlim \{L^2(U_{\Phi,n}),p_{\ell,n}\}$ is taken in that
category.  We now have the Hilbert space projective limit
\begin{equation}\label{hilbert-limit-u}
L^2(N_{\Phi,\infty}) := \varprojlim \{L^2(N_{\Phi,n}), p_{\ell,n}\} = 
{\sum}_{\mu_{\lambda,\varphi} \in\Xi_{\Phi,\infty}} 
\cH_\lambda \widehat{\otimes} \cH_\lambda^* \text{ orthogonal direct sum.}
\end{equation}
The left/right representation of $U_{\Phi,\infty} \times U_{\Phi,\infty}$ on $L^2(U_{\Phi,\infty})$ 
is multiplicity-free, preserves each summand
$\cH_\lambda \widehat{\otimes} \cH_\lambda^*$\,, and acts on
$\cH_\lambda \widehat{\otimes} \cH_\lambda^*$ by the irreducible representation
of highest weight $(\lambda,\lambda^*)$.  The connection with matrix 
coefficients is
\begin{equation}\label{schwartz-proj-sys-u}
\begin{CD}
\cC(U_{\Phi,1}) @< q_{2,1} << \cC(U_{\Phi,2})
        @< q_{3,2} <<  \dots @ <<< 
        \cC(U_{\Phi,\infty}) = \varprojlim \cC(U_{\Phi,n}) \\
    @VV{r_1}V       @VV{r_2}V          @VVV  @VV{r_\infty}V \\
L^2(U_{\Phi,1}) @< p_{2,1} << L^2(U_{\Phi,2})
        @< p_{3,2} <<  \dots @ <<< 
        L^2(U_{\Phi,\infty}) = \varprojlim L^2(U_{\Phi,n})
\end{CD}
\end{equation}
as in (\ref{schwartz-proj-sys}).  As in Proposition \ref{dense-embedding} this
realizes the limit space $L^2(U_{\Phi,\infty})$ as a Hilbert space completion of
the Schwartz space $\cC(U_{\Phi,\infty})$\,, and because of compactness the latter 
in turn is the projective limit of spaces 
$\cC(U_{\Phi,n}) = C^\infty(U_{\Phi,n})$.  The Fourier inversion formula for
$U_{\Phi,\infty}$ is given stepwise as in Theorem \ref{limit-inversion}.

\section{Representations of the Limit Groups IV: 
	$U_{\Phi,\infty} N_{\Phi,\infty}$}
\label{sec12}
\setcounter{equation}{0}
We combine some of the results of Sections \ref{sec9} and \ref{sec11},
extending them to the subgroup $U_{\Phi,\infty} N_{\Phi,\infty}$\,. 
\medskip

In view of the discussion culminating in (\ref{lim-subgroups1}) 
we assume that the direct system
$\{G_n\}$ of real semisimple Lie groups satisfies
\begin{equation}
\begin{aligned}
\text{ if } \ell \geqq n &\text{ then } \Phi_n \subset \Phi_\ell \text{ and }
	(\Psi_n \setminus \Phi_n) \subset (\Psi_\ell \setminus \Phi_\ell) 
	\text{ so that} \\
& U_{\Phi,\infty} := \varinjlim U_{\Phi,n} \text{ and }
U_{\Phi,\infty}N_{\Phi,\infty} := \varinjlim
U_{\Phi,n}N_{\Phi,n} \text{ exist.}
\end{aligned}
\end{equation}
We also extend Definition \ref{invariant}:

\begin{definition}\label{inv-inf}
{\rm $N_{\Phi,\infty} = L_{\Phi,1}L_{\Phi,2}L_{\Phi,3}\dots $
is {\em weakly invariant} if each 
$\Ad(U_{\Phi,\infty})\gz_{\Phi,j} = \gz_{\Phi,j}$\,.
\hfill $\diamondsuit$ }
\end{definition}
We'll need a variation on Lemma \ref{components}.  Recall
the maximal lim-compact subgroup $K = \varinjlim K_n$\,.

\begin{lemma}\label{components-inf}
Let $F_n = \exp(i\ga_{\Phi,n})\cap K_n$ and 
$F = \exp(i\ga_{\Phi,\infty})\cap K$.
Then $F = \varinjlim F_n$ is contained in $U_{\Phi,\infty}$ and is central
in $M_{\Phi,\infty}$; if $x \in F$ then $x^2 = 1$, $U_{\Phi,\infty} =
F U_{\Phi,\infty}^0$\,; and $M_{\Phi,\infty} = F M_{\Phi,\infty}^0$\,.
\end{lemma}

\begin{proof}  Lemma \ref{components} contains the corresponding results
for the $F_n$\,.  It follows that $F$ is a subgroup of $U_{\Phi,\infty}$
central in $M_{\Phi,\infty}$\,, that describes the components as stated, and
in which every element has square $1$.
\end{proof}

\begin{lemma}\label{f-triv-s}
The action of $\Ad(F)$ on $\gs^*_{\Phi,\infty}$ is trivial.
\end{lemma}
\begin{proof} Lemma \ref{F-triv} shows that $\Ad(F_\ell)$ is trivial
on $\gs^*_{\Phi,n}$ whenever $\ell \geqq n$.
\end{proof}

Now suppose that $N_{\Phi,\infty} = L_{\Phi,1}L_{\Phi,2}L_{\Phi,3}\dots $
is weakly invariant.  We continue as in Section \ref{sec6}.  

\begin{equation}\label{regset-inf}
\gr^*_{\Phi,\infty} = \{\gamma = (\gamma_n) \in \gt^*_{\Phi,\infty} \mid
	\text{ for each } n,\,\,\Ad^*(U_{\Phi,n})\gamma_n \text{ is a principal }
	U_{\Phi,n}\text{-orbit on } \gs^*_{\Phi,n}\}.
\end{equation}
It is dense, open and $\Ad^*(U_{\Phi,\infty})$-invariant in 
$\gs^*_{\Phi,\infty}$.
Let $\sigma: \Ad^*(U_{\Phi,\infty})\backslash \gr^*_{\Phi,\infty} \to
\gr^*_{\Phi,\infty}$ be a measurable section to $\gr^*_{\Phi,\infty} \to
\Ad^*(U_{\Phi,\infty})\backslash \gr^*_{\Phi,\infty}$ on whose image all the
isotropy subgroups are the same, denoted
\begin{equation}\label{m-iso-regset-inf}
U'_{\Phi,\infty}: \text{ isotropy subgroup of } U_{\Phi,\infty}
	\text{ at } \sigma(\Ad^*(U_{\Phi,\infty})(\gamma)),
	\text{ independent of } 
	\gamma \in \gr^*_{\Phi,\infty}\,.
\end{equation}
As a bonus, in view of Lemma \ref{lim-a-stab}, 
the isotropy subgroup of $U_{\Phi,\infty}A_{\Phi,\infty}$ at
$\Ad^*(a)\sigma(\Ad^*(U_{\Phi,\infty})(\gamma))$ is 
$U'_{\Phi,\infty}A'_{\Phi,\infty}$\,, independent of
$a \in A_{\Phi,\infty}$ and $\gamma \in \gr^*_{\Phi,\infty}$\,.
Note that 
$
U'_{\Phi,\infty} = \varinjlim U'_{\Phi,n}
$
where $U'_{\Phi,n}$ is the isotropy subgroup of $U_{\Phi,n}$ at
$\sigma(\Ad^*(U_{\Phi,\infty})(\gamma))_n$\,, independent of
$\gamma \in \gr^*_{\Phi,\infty}$\,.  
Given $\mu' \in \widehat{U'_{\Phi,\infty}}$, say $\mu' = \varinjlim \mu'_n$ 
where $\mu'_n \in \widehat{U'_{\Phi,n}}$, and $\gamma$ is in the image of
$\sigma$, we have representations
\begin{equation}\label{rep-inf-un}
\begin{aligned}
\pi_{\Phi,\gamma,\mu',n}: = 
&\Ind_{U'_{\Phi,n}N_{\Phi,n}}^{U_{\Phi,n}N_{\Phi,n}}
(\mu'_n\otimes \pi_{\Phi,\gamma,n}) \text{ and } \\
&\pi_{\Phi,\gamma,\mu',\infty} =
  \Ind_{U'_{\Phi,\infty}N_{\Phi,\infty}}^{U_{\Phi,\infty}N_{\Phi,\infty}}
    (\mu'\otimes \pi_{\Phi,\gamma,\infty})
	:= \varinjlim \pi_{\Phi,\gamma,\mu',n}\,.
\end{aligned}
\end{equation}
To be precise here, $\mu'$ must be a cocycle representation of 
$U'_{\Phi,\infty}$
where the cocycle $\varepsilon$ is the inverse of the Mackey obstruction to 
extending $\pi_{\Phi,\gamma,\infty}$ to a representation of 
$U'_{\Phi,\infty}N_{\Phi,\infty}$\,.
\medskip

As in (\ref{rel-schwartz-an}) we define the {\em relative Schwartz space}
$\cC((U_{\Phi,n}N_{\Phi,n}/U'_{\Phi,n}S_{\Phi,n}),
\mu'_n \otimes \zeta_{\gamma_n})$ to consist of all functions
$f \in C^\infty(U_{\Phi,n}N_{\Phi,n})$ such that
\begin{equation}\label{rel-schwartz-un}
\begin{aligned}
&f(xus) = \mu'_n(u)^{-1}\zeta_{\gamma_n}(s)^{-1}f(x)\,\,
        (x \in N_{\Phi,n}, u \in U'_{\Phi,n}\,,
        s \in S_{\Phi,n}), \text{ and }
        |q(g)p(D)f| \text{ is bounded}\\
&\text{ on } U_{\Phi,n}N_{\Phi,n}/U'_{\Phi,n}S_{\Phi,n} \text{ for all
        polynomials }
        p, q \text{ on } N_{\Phi,n}/S_{\Phi,n}
        \text{ and all } D \in \cU(\gu_{\Phi,n}+\gn_{\Phi,n}).
\end{aligned}
\end{equation}
The corresponding {\em limit relative Schwartz space} is
$$
\cC((U_{\Phi,\infty}N_{\Phi,\infty}/U'_{\Phi,\infty}S_{\Phi,\infty}),
        (\mu'\otimes \zeta_\gamma)) = \varprojlim
\cC((U_{\Phi,n}N_{\Phi,n}/U'_{\Phi,n}S_{\Phi,n}), 
	(\mu'_n\otimes\zeta_{\gamma_n})),
$$ 
consisting of all functions
$f \in C^\infty(U_{\Phi,\infty}N_{\Phi,\infty})$ such that
\begin{equation}\label{lim-schwartz-un}
\begin{aligned}
&f(xus) = \mu'(u)^{-1}\zeta_\gamma(s)^{-1}f(x) \text{ for }
        x \in N_{\Phi,\infty}, u \in U'_{\Phi,\infty} \text{ and }
        s \in S_{\Phi,\infty}\,, \text{ and }\\
&|q(g)p(D)f| \text{ is bounded } \text{ on }
        U_{\Phi,\infty}N_{\Phi,\infty}/U'_{\Phi,\infty}S_{\Phi,\infty} \\
&\text{for all polynomials }
        p, q \text{ on } N_{\Phi,\infty}/S_{\Phi,\infty}
        \text{ and all } D \in \cU(\gu_{\Phi,\infty}+\gn_{\Phi,\infty}).
\end{aligned}
\end{equation}

Theorem \ref{coef-sch} and Corollaries \ref{l1-coef1} and \ref{l1-coef2}
hold for our groups $U_{\Phi,n} N_{\Phi,n} $ here with essentially no change, 
so we will not repeat them.
\medskip

Following the discussion in Section \ref{sec10} for 
$\cC(A_{\Phi,n}N_{\Phi,n})$ and $\cC(A_{\Phi,\infty}N_{\Phi,\infty})$
we define seminorms
\begin{equation}\label{seminorm-un}
\nu_{k,D,n}(f) = {\sup}_{x \in U_{\Phi,n}N_{\Phi,n}}||x||^k|Df(x)|
\text{ where } k > 0 \text{ and } D \in \cU(\gu_{\Phi,n}+\gn_{\Phi,n}).
\end{equation}
That defines the Schwartz space $\cC(U_{\Phi,n}N_{\Phi,n})$ as in
(\ref{schwartz-n-n}):
\begin{equation}\label{schwartz-un-n}
\cC(U_{\Phi,n}N_{\Phi,n}) = 
	\left \{f \in C^\infty(U_{\Phi,n}N_{\Phi,n}) \left |
	\nu_{k,D,n}(f) < \infty \text{ for } k > 0 \text{ and } 
	D \in \cU(\gu_{\Phi,n}+\gn_{\Phi,n})\right . \right \}.
\end{equation}
Finally
we define $\cC(U_{\Phi,\infty}N_{\Phi,\infty})$
to be the inverse limit in the category of locally convex topological
vector spaces and continuous linear maps, as in (\ref{schwartz-n-inf}):
\begin{equation}\label{schwartz-un-inf}
\cC(U_{\Phi,\infty}N_{\Phi,\infty}) = 
	\left \{f \in C^\infty(U_{\Phi,\infty}N_{\Phi,\infty})
	\left | f|_{U_{\Phi,n}N_{\Phi,n}} \in \cC(U_{\Phi,n}N_{\Phi,n})
	\right . \right \}
= \varprojlim \cC(U_{\Phi,n}N_{\Phi,n}).
\end{equation}
Then we have
\begin{lemma}\label{lim-schwartz-prop-un}
The Schwartz space $\cC(U_{\Phi,\infty}N_{\Phi,\infty})$ consists of all
functions $f \in C^\infty(U_{\Phi,\infty}N_{\Phi,\infty})$ such that
$\nu_{k,D,n}(f) < \infty$ for all integers $k > 0$ and all
$D \in \cU(\gu_{\Phi,\infty}+\gn_{\Phi,\infty})$.
Here $\nu_{k,D,n}$ is given by {\rm (\ref{seminorm-un})}.
The $\cC(U_{\Phi,n}N_{\Phi,n})$ are nuclear Fr\' echet spaces and
$\cC(U_{\Phi,\infty}N_{\Phi,\infty})$ is an LF space.
The left/right actions of $(U_{\Phi,n}N_{\Phi,n}) \times (U_{\Phi,n}N_{\Phi,n})$
on $\cC(U_{\Phi,n}N_{\Phi,n})$ and of $(U_{\Phi,\infty}N_{\Phi,\infty})\times
(U_{\Phi,\infty}N_{\Phi,\infty})$ on $\cC(U_{\Phi,\infty}N_{\Phi,\infty})$ are
continuous.
\end{lemma}
\medskip

We construct $L^2(U_{\Phi,\infty} N_{\Phi,\infty}) := 
\varprojlim L^2(U_{\Phi,n}N_{\Phi,n})$ along the lines of Section \ref{sec9}.  
Let $\gamma = (\gamma_n) \in \gr^*_{\Phi,\infty}$
such that $\gamma$ is in the image of $\sigma$.
Consider $\mu' = \varinjlim \mu'_n \in \widehat{U'_{\Phi,\infty}}$ where
(i) $\mu'_n \in \widehat{U'_{\Phi,n}}$ and 
(ii) $\cH_{\mu'_n} \subset \cH_{\mu'_\ell}$ from a map $u_n \mapsto u_\ell$
of highest weight unit vectors, for $\ell \geqq n$.  
\medskip

Every $f \in \cC(U_{\Phi,n}N_{\Phi,n})$ is a Schwartz wave packet along
$(\gu'_{\Phi,n})^* + \gs_{\Phi,n}^*$ of coefficients of the various
$\pi_{\Phi,\gamma,\mu',n}$, and the corresponding inverse systems
fit together as in (\ref{schwartz-proj-sys}):
{\footnotesize
\begin{equation}\label{schwartz-proj-sys-un}
\begin{CD}
\cC(U_{\Phi,1}N_{\Phi,1}) @< q_{2,1} << \cC(U_{\Phi,2}N_{\Phi,2})
        @< q_{3,2} <<  \dots @ <<<
        \cC(U_{\Phi,\infty} N_{\Phi,\infty}) = \varprojlim \cC(U_{\Phi,n}N_{\Phi,n}) \\
    @VV{r_1}V       @VV{r_2}V          @VVV  @VV{r_\infty}V \\
L^2(U_{\Phi,1}N_{\Phi,1}) @< p_{2,1} << L^2(U_{\Phi,2}N_{\Phi,2})
        @< p_{3,2} <<  \dots @ <<<
        L^2(U_{\Phi,\infty} N_{\Phi,\infty}) = \varprojlim L^2(U_{\Phi,n}N_{\Phi,n})
\end{CD}
\end{equation}
}
Again, the $r_n$ are continuous injections with dense image, so
$r_\infty: \cC(U_{\Phi,\infty}N_{\Phi,\infty}) \hookrightarrow
L^2(U_{\Phi,\infty}N_{\Phi,\infty})$ is a continuous injection with dense image.
As in Proposition \ref{dense-embedding} we conclude

\begin{proposition} \label{dense-embedding-un}
Assume {\rm (\ref{lim-subgroups1})(d)}, so that 
$U_{\Phi,\infty}N_{\Phi,\infty} = \varinjlim U_{\Phi,n}N_{\Phi,n}$ is well 
defined.  Define $r_\infty: \cC(U_{\Phi,\infty}N_{\Phi,\infty}) 
\hookrightarrow L^2(U_{\Phi,\infty}N_{\Phi,\infty})$ as in 
the commutative diagram {\rm (\ref{schwartz-proj-sys-un})}.  Then 
$L^2(U_{\Phi,\infty}N_{\Phi,\infty})$ is a Hilbert space completion of 
$\cC(U_{\Phi,\infty} N_{\Phi,\infty})$.  
In particular $r_\infty$ defines a pre-Hilbert
space structure on $\cC(U_{\Phi,\infty}N_{\Phi,\infty})$ with completion 
$L^2(U_{\Phi,\infty}N_{\Phi,\infty})$.
\end{proposition}

As in \cite[Corollary 5.17]{W2015} $\cC(U_{\Phi,\infty}N_{\Phi,\infty})$ is 
independent of the choices we made in the construction of 
$L^2(U_{\Phi,\infty}N_{\Phi,\infty})$, so

\begin{corollary}\label{l2-well-def-un}
The limit Hilbert space $L^2(U_{\Phi,\infty}N_{\Phi,\infty}) = 
\varprojlim \{L^2(U_{\Phi,n}N_{\Phi,n}), p_{\ell,n}\}$ of
{\rm (\ref{schwartz-proj-sys-un})} , and the left/right regular representation
of $(U_{\Phi,\infty}N_{\Phi,\infty})\times (U_{\Phi,\infty}N_{\Phi,\infty})$ 
on $L^2(U_{\Phi,\infty}N_{\Phi,\infty})$, are 
independent of the choice of vectors $\{e\}$ in {\rm (\ref{unitfactor})}and 
highest weight unit vectors $\{u_n\}$.
\end{corollary}

The limit Fourier inversion formula is
\begin{theorem}\label{limit-inversion-un}
Given $\pi_{\Phi,\gamma,\mu_{\lambda,\varphi,n}} \in 
\widehat{U_{\Phi,n}N_{\Phi,n}}$ let 
$\Theta_{\pi_{\Phi,\gamma,\mu_{\lambda,\varphi,n}}}$
denote its distribution character.  Then
$\Theta_{\pi_{\Phi,\gamma,\mu_{\lambda,\varphi,n}}}$ is a tempered
distribution.  Let $f \in \cC(U_{\Phi,\infty}N_{\Phi,\infty})$ and
$x \in U_{\Phi,\infty}N_{\Phi,\infty}$\,.
Then $x \in U_{\Phi,n}N_{\Phi,n}$ for some $n$ and
\begin{equation}\label{lim-inv-formula-un}
f(x) = c_n\int_{\gamma_n\in\mathfrak{t}_{\Phi,n}^*}
        {\sum}_{\mu'_n \in \widehat{U'_{\Phi,n}}}\,
        \Theta_{\pi_{\Phi,\gamma,\mu',n}}(r_xf)
        \deg(\mu')|\Pf_{\mathfrak{n}_n}({\gamma_n})|d\gamma_n
\end{equation}
where $c_n = 2^{d_1 + \dots + d_{m_n}} d_1! d_2! \dots d_{m_n}!$ as in 
{\rm (\ref{c-d}a)}
and $m_n$ is the number of factors $L_r$ in $N_{\Phi,n}$.
\end{theorem}

\begin{proof}  We compute along the lines of \cite[Theorem 2.7]{LW1982}.
Let $h = r_xf$.  From \cite[Theorem 3.2]{KL1972},
$$
\begin{aligned}
\tr \pi_{\Phi,\gamma,\mu',n}(h) 
	&={\int}_{x \in U_{\Phi,n}/U'_{\Phi,n}}
		\tr \int_{yu \in N_{\Phi,n}U'_{\Phi,n}}h(x^{-1}yux)\cdot
		(\pi_{\Phi,\gamma,n}\otimes \mu'_n)
		   (yu)\,dy\, du\, dx\\
&= {\int}_{x \in U_{\Phi,n}/U'_{\Phi,n}}
	\tr \int_{N_{\Phi,n}U'_{\Phi,n}}h(yx^{-1}ux)\cdot
		(\pi_{\Phi,\gamma,n}\otimes \mu'_n)
		(xyx^{-1}u)\,
		dy\, du\, dx.
\end{aligned}
$$
Now
$$
\begin{aligned}
{\sum}_{_{\widehat{U_{\Phi,n}'}}} 
  &\tr \pi_{\Phi,\gamma,\mu',n}(h) 
	\deg\mu'_n \\
&= {\sum}_{_{_{\widehat{U'_{\Phi,n}}}}} \int_{x \in U_{\Phi,n}/U'_{\Phi,n}} 
        \tr \int_{N_{\Phi,n}U'_{\Phi,n}}h(yx^{-1}ux)
                (\pi_{\Phi,\gamma,\mu',n})(xyx^{-1}u)\,
                dy\, du\, dx\, \deg\mu'_n \\
&= \int_{x \in U_{\Phi,n}/U'_{\Phi,n}} {\sum}_{_{_{ \widehat{U'_{\Phi,n}}}}}
	\tr \int_{N_{\Phi,n}U'_{\Phi,n}}h(yx^{-1}ux)
                (\pi_{\Phi,\gamma,\mu',n})(xyx^{-1}u)\,
		dy\, du\, \deg\mu'_n\, dx\\
&= \int_{x \in U_{\Phi,n}/U'_{\Phi,n}} \tr \int_{N_{\Phi,n}}h(y) 
	\pi_{\Phi,\gamma,\mu',n}(xyx^{-1}) dy\, dx \\
&= \int_{x \in U_{\Phi,n}/U'_{\Phi,n}} \tr \int_{N_{\Phi,n}}h(y)
        (x^{-1}\cdot \pi_{\Phi,\gamma,\mu',n})(y) dy\, dx \\
&= \int_{x \in U_{\Phi,n}/U'_{\Phi,n}} \tr ((x^{-1}\cdot
	\pi_{\Phi,\gamma,\mu',n})(h))\, dx\\
&= \int_{\Ad^*(U_{\Phi,n})\gamma} 
	\tr \pi_{\Phi,\gamma,\mu'_n}(h) 
	|\Pf(\gamma_n)|d\gamma_n .
\end{aligned}
$$
Summing over the the space of $U_{\Phi,n}$-orbits on $\gs_{\Phi,n}^*$ 
we now have
$$
\begin{aligned}
\int_{U_{\Phi,n}\backslash \gs_{\Phi,n}^*}
	&{\sum}_{_{\widehat{U'_{\Phi,n}}}}
	\tr \pi_{\Phi,\gamma,\mu',n}(h) 
		\deg\mu'_n |\Pf(\gamma_n)|d\gamma_n \\
& = \int_{U_{\Phi,n}\backslash \gs_{\Phi,n}^*} 
	\tr \pi_{\Phi,\gamma,\mu',n}(h) 
		|\Pf(\gamma_n)|d\gamma_n \\
& = \int_{\gs_{\Phi,n}^*} \tr \pi_{\Phi,\gamma,n}(h) |\Pf(\gamma_n)|d\gamma_n 
	= h(1) = f(x)\,.
\end{aligned}
$$
That completes the proof.
\end{proof}

\section{Representations of the Limit Groups V:
        $U_{\Phi,\infty} A_{\Phi,\infty} N_{\Phi,\infty}$}
\label{sec13}
\setcounter{equation}{0}
We extend some of the results of Sections \ref{sec10} and \ref{sec12} 
to the maximal amenable subgroups $E_{\Phi,\infty}:= 
U_{\Phi,\infty} A_{\Phi,\infty} N_{\Phi,\infty}$ of $G$.  Here we
are using amenability of the $E_{\Phi,n} := U_{\Phi,n} A_{\Phi,n} N_{\Phi,n}$.
\medskip

As in Definition \ref{invariant} the decomposition
$N_{\Phi,\infty} = L_{\Phi,1}L_{\Phi,2}\dots$ is {\em invariant} if 
each $\ad(\gm_{\Phi,\infty})\gz_{\Phi,\infty,j} = \gz_{\Phi,\infty,j}$\,, 
in other words
if each $N_{\Phi,n} = L_{\Phi,1}L_{\Phi,2}\dots L_{\Phi,m_n}$ is
invariant.  Similarly $N_{\Phi,\infty} = L_{\Phi,1}L_{\Phi,2}\dots$ is 
{\em weakly invariant} if each $\ad(\gu_{\Phi,\infty})\gz_{\Phi,\infty,j} = 
\gz_{\Phi,\infty,j}$\,,
i.e. if each $N_{\Phi,n} = L_{\Phi,1}L_{\Phi,2}\dots L_{\Phi,m_n}$ is
weakly invariant.  
\medskip

Recall the principal orbit set $\gr^*_{\Phi,\infty}$ from (\ref{regset-inf})
and the measurable section 
$\sigma:\Ad^*(U_{\Phi,\infty})\backslash \gr^*_{\Phi,\infty} \to 
\gr^*_{\Phi,\infty}$ on whose image all the isotropy subgroups of 
$\Ad^*(U_{\Phi,\infty})$ are the same.  Note that $\sigma$ is
$\Ad^*(A_{\Phi,\infty})$-equivariant, so we may view it as a section
to $\gr^*_{\Phi,\infty} \to \Ad^*(U_{\Phi,\infty}A_{\Phi,\infty})
\backslash \gr^*_{\Phi,\infty}$ on whose image all the isotropy subgroups of
$\Ad^*(U_{\Phi,\infty}A_{\Phi,\infty})$ are the same.
Following (\ref{ma-iso-regset}), (\ref{a-stab-inf})
and (\ref{m-iso-regset-inf}), and as remarked just after
(\ref{m-iso-regset-inf}), that common isotropy subgroup is 
\begin{equation}\label{ua-iso-regset-inf}
U'_{\Phi,\infty}A'_{\Phi,\infty}: \text{ isotropy of } 
U_{\Phi,\infty}A_{\Phi,\infty} \text{ at } 
\sigma(\Ad^*(U_{\Phi,\infty}A_{\Phi,\infty}))(\gamma), \text{ independent of }
\gamma \in \gr^*_{\Phi,\infty}\,.
\end{equation}

Let $\gamma \in \gt_{\Phi,\infty}^*$ be in the image of $\sigma$.  Then
$\pi_{\Phi,\gamma\infty}$ extends to a representation 
$\pi^\dagger_{\Phi,\gamma,\infty}$ of 
$U'_{\Phi,\infty}A'_{\Phi,\infty}N_{\Phi,\infty}$
on the same representation space $\cH_{\pi_{\Phi,\gamma\infty}}$\,.
Given $\mu' \in \widehat{U'_{\Phi,\infty}}$ and $\xi' = (\xi'_n) \in 
(\ga'_{\Phi,\infty})^*$
the corresponding representation of $E_{\Phi,\infty}:= 
U_{\Phi,\infty} A_{\Phi,\infty} N_{\Phi,\infty}$ is induced from 
$E'_{\Phi,\infty}:= U'_{\Phi,\infty}A'_{\Phi,\infty}N_{\Phi,\infty}$
as follows.
\begin{equation}\label{repuan'} 
\begin{aligned}
\pi_{\Phi,\gamma,\xi',\mu',\infty}& = \varinjlim \pi_{\Phi,\gamma,\xi',\mu',n}
\text{ where } \pi_{\Phi,\gamma,\xi',\mu',n} = \Ind_{E'_{\Phi,n}}^{E_{\Phi,n}}
\left (\mu_n'\otimes \exp(2\pi i \xi'_n)\otimes \pi^\dagger_{\Phi,\gamma,n}\right ),\\
&\text{ in other words }\pi_{\Phi,\gamma,\xi',\mu',n} =
\Ind_{U'_{\Phi,n}A'_{\Phi,n}N_{\Phi,n}}^{U_{\Phi,n} 
A_{\Phi,n} N_{\Phi,n}}\,
\left (\mu_n'\otimes \exp(2\pi i \xi'_n)\otimes \pi^\dagger_{\Phi,\gamma,n}\right ).
\end{aligned}
\end{equation}

As in Section \ref{sec10} the {\em relative Schwartz space}
$\cC((U_{\Phi,n}A_{\Phi,n}N_{\Phi,n}/U'_{\Phi,n}A'_{\Phi,n}S_{\Phi,n}),
(\mu_n'\otimes \exp(2\pi i \xi'_n)\otimes\zeta_{\gamma_n}))$ consists
of all functions $f \in C^\infty(U_{\Phi,n}A_{\Phi,n}N_{\Phi,n})$ such that
\begin{equation}\label{rel-schwartz-uan-1}
f(xuas) = \mu'_n(u)^{-1}\exp(-2\pi i\xi'(\log a))\zeta_{\gamma_n}(s)^{-1}f(x)\,
        (x \in N_{\Phi,n}, u \in U'_{\Phi,n}, a \in A'_{\Phi,n}\,,
        s \in S_{\Phi,n}), 
\end{equation}
and, for all polynomials $p, q$ on $A_{\Phi,n}N_{\Phi,n}/A'_{\Phi,n}S_{\Phi,n}$
and all $D \in \cU(\gu_{\Phi,n}\ga_{\Phi,n}\gn_{\Phi,n})$,
\begin{equation}\label{rel-schwartz-uan-2}
|q(g)p(D)f| \text{ is bounded on } 
U_{\Phi,n}A_{\Phi,n}N_{\Phi,n}/U'_{\Phi,n}A'_{\Phi,n}S_{\Phi,n}\,. 
\end{equation}
The corresponding {\em limit relative Schwartz space} is
\begin{equation}\label{rel-schwartz-uan-3}
\begin{aligned}
\cC((U_{\Phi,\infty}&A_{\Phi,\infty}N_{\Phi,\infty}/U'_{\Phi,\infty}A'_{\Phi,\infty}
S_{\Phi,\infty}),
(\mu'\otimes \exp(2\pi i \xi')\otimes\zeta_{\gamma}))\\
&= \varprojlim \cC((U_{\Phi,n}A_{\Phi,n}N_{\Phi,n}/U'_{\Phi,n}A'_{\Phi,n}
S_{\Phi,n}), (\mu_n'\otimes \exp(2\pi i \xi'_n)\otimes\zeta_{\gamma_n})).
\end{aligned}
\end{equation}

Again, Theorem \ref{coef-sch} and Corollaries \ref{l1-coef1} and 
\ref{l1-coef2} hold {\em mutatis mutandis} for the groups $E_{\Phi,n}$ so
we won't repeat them.  We extend the definition (\ref{seminorm-un}) of
seminorms on $U_{\Phi,n}N_{\Phi,n}$ to $E_{\Phi,n} = 
U_{\Phi,n}A_{\Phi,n}N_{\Phi,n}$:
\begin{equation}\label{seminorm-uan}
\nu_{k,D,n}(f) = {\sup}_{x \in E_{\Phi,n}} ||x||^k|Df(x)| \text{ where }
	k > 0 \text{ and } D \in \ge_{\Phi,n} \text{ for }
	f \in C^\infty(E_{\Phi,n}).
\end{equation}
That defines the Schwartz space $\cC(E_{\Phi,n})$:
\begin{equation}\label{schwartz-uan-n}
\cC(E_{\Phi,n}) = 
	\left \{f \in C^\infty(E_{\Phi,n}) \left |
	\nu_{k,D,n}(f) < \infty \text{ for } k > 0 \text{ and } 
	D \in \cU(\ge_{\Phi,n})\right . \right \}.
\end{equation}
Finally
we define $\cC(E_{\Phi,\infty})$
to be the inverse limit in the category of locally convex topological
vector spaces and continuous linear maps, 
\begin{equation}\label{schwartz-uan-inf}
\cC(E_{\Phi,\infty}) = 
	\left \{f \in C^\infty(E_{\Phi,\infty})
	\left | f|_{E_{\Phi,n}} \in \cC(E_{\Phi,n})
	\right . \right \}
= \varprojlim \cC(E_{\Phi,n}).
\end{equation}
As before
\begin{lemma}\label{lim-schwartz-prop-e}
The Schwartz space $\cC(E_{\Phi,\infty})$ consists of all
functions $f \in C^\infty(E_{\Phi,\infty})$ such that
$\nu_{k,D,n}(f) < \infty$ for all integers $k > 0$ and all
$D \in \cU(\ge_{\Phi,\infty})$.
Here $\nu_{k,D,n}$ is given by {\rm (\ref{seminorm-uan})}.
The $\cC(E_{\Phi,n})$ are nuclear Fr\' echet spaces and
$\cC(E_{\Phi,\infty})$ is an LF space.
The left/right actions of $(E_{\Phi,n}) \times (E_{\Phi,n})$
on $\cC(E_{\Phi,n})$ and of $(E_{\Phi,\infty})\times
(E_{\Phi,\infty})$ on $\cC(E_{\Phi,\infty})$ are
continuous.
\end{lemma}
\medskip

We construct $L^2(E_{\Phi,\infty}) := 
\varprojlim L^2(E_{\Phi,n})$ as before.
Let $\gamma = (\gamma_n) \in \gr^*_{\Phi,\infty}$
such that $\gamma$ is in the image of $\sigma$.
Consider $\mu' = (\mu'_n) \in \widehat{U'_{\Phi,\infty}}$ 
and $\xi = (\xi_n) \in \ga_{\Phi,n}$\,, For $\ell \geqq n$ we consider 
the maps on representation spaces given by
$\cH_{\pi_{\Phi,\gamma,\xi,\mu',n}} \subset \cH_{\pi_{\Phi,\gamma,\xi,\mu',\ell}}$ 
from maps $u_n \mapsto u_\ell$ of highest weight unit vectors.
\medskip

Every $f \in \cC(E_{\Phi,n})$ is a Schwartz wave packet along
$(\gu'_{\Phi,n})^* \ga'_{\Phi,n} + \gs_{\Phi,n}^*$ of coefficients of the various
$\pi_{\Phi,\gamma,\xi,\mu',n}$\,.  The corresponding inverse systems
fit together as in (\ref{schwartz-proj-sys}):
{\footnotesize
\begin{equation}\label{schwartz-proj-sys-uan}
\begin{CD}
\cC(E_{\Phi,1}) @< q_{2,1} << \cC(E_{\Phi,2})
        @< q_{3,2} <<  \dots @ <<<
        \cC(E_{\Phi,\infty}) = \varprojlim \cC(E_{\Phi,n}) \\
    @VV{r_1}V       @VV{r_2}V          @VVV  @VV{r_\infty}V \\
L^2(E_{\Phi,1}) @< p_{2,1} << L^2(E_{\Phi,2})
        @< p_{3,2} <<  \dots @ <<<
        L^2(E_{\Phi,\infty}) = \varprojlim L^2(E_{\Phi,n})
\end{CD}
\end{equation}
}
Again, the $r_n$ are continuous injections with dense image, so
$r_\infty: \cC(E_{\Phi,\infty}) \hookrightarrow
L^2(E_{\Phi,\infty})$ is a continuous injection with dense image.
As in Proposition \ref{dense-embedding} we conclude

\begin{proposition} \label{dense-embedding-uan}
Assume {\rm (\ref{lim-subgroups1})(d)}, so that 
$E_{\Phi,\infty} = \varinjlim E_{\Phi,n}$ is well 
defined.  Define $r_\infty: \cC(E_{\Phi,\infty}) 
\hookrightarrow L^2(E_{\Phi,\infty})$ as in {\rm (\ref{schwartz-proj-sys-uan})}.  
Then $L^2(E_{\Phi,\infty})$ is a Hilbert space completion of 
$\cC(E_{\Phi,\infty})$.  
In particular $r_\infty$ defines a pre-Hilbert
space structure on $\cC(E_{\Phi,\infty})$ with completion 
$L^2(E_{\Phi,\infty})$.
\end{proposition}

As in \cite[Corollary 5.17]{W2015} $\cC(E_{\Phi,\infty})$ is 
independent of the choices we made in the construction of 
$L^2(E_{\Phi,\infty})$, so

\begin{corollary}\label{l2-well-def-uan}
The limit Hilbert space $L^2(E_{\Phi,\infty}) = 
\varprojlim \{L^2(E_{\Phi,n}), p_{\ell,n}\}$ of
{\rm (\ref{schwartz-proj-sys-uan})} , and the left/right regular representation
of $E_{\Phi,\infty}\times E_{\Phi,\infty}$ on $L^2(E_{\Phi,\infty})$, are 
independent of the choice of vectors $\{e\}$ in {\rm (\ref{unitfactor})}and 
highest weight unit vectors $\{u_n\}$.
\end{corollary}

The limit Fourier inversion formula is
\begin{theorem}\label{limit-inversion-uan}
Given $\pi_{\Phi,\gamma,\xi,\mu_{\lambda,\varphi,n}} \in 
\widehat{E_{\Phi,n}}$ let 
$\Theta_{\pi_{\Phi,\gamma,\xi,\mu_{\lambda,\varphi,n}}}$
denote its distribution character.  Then
$\Theta_{\pi_{\Phi,\gamma,\xi,\mu_{\lambda,\varphi,n}}}$ is a tempered
distribution.  Let $f \in \cC(E_{\Phi,\infty})$ and
$x \in E_{\Phi,\infty}$\,.
Then $x \in E_{\Phi,n}$ for some $n$ and
\begin{equation}\label{lim-inv-formula-uan}
f(x) = c_n\int_{\gamma_n\in\mathfrak{t}_{\Phi,n}^*}
	\int_{\xi \in \ga'_{\Phi,n}}
        {\sum}_{\mu'_n \in \widehat{U'_{\Phi,n}}}\,
        \Theta_{\pi_{\Phi,\gamma,\xi,\mu',n}}(r_xf)
        \deg(\mu')|\Pf_{\mathfrak{n}_n}({\gamma_n})|d\xi\,d\gamma_n
\end{equation}
where $c_n = (\tfrac{1}{2\pi})^{\dim \ga'_\Phi/2}\,\,
2^{d_1 + \dots + d_{m_n}} d_1! d_2! \dots d_{m_n}!$ 
and $m_n$ is the number of factors $L_r$ in $N_{\Phi,n}$.
\end{theorem}

\begin{proof}  We combine the ideas in the proofs of Theorems \ref{planch-an}
and \ref{limit-inversion-un}. In an attempt to keep the notation
under control we write $U''_n$ for
$U_{\Phi,n}/U'_{\Phi,n}$ and $A''_n$ for $A_{\Phi,n}/A'_{\Phi,n}$\,,
and more generally we drop the subscript $\Phi$.  We write $\delta$ for the 
modular function of $Q_\Phi$\,.
Let $h = r_xf$.  Using \cite[Theorem 3.2]{KL1972},
$$
\begin{aligned}
\tr \pi_{\gamma,\xi,\mu',n}(Dh) 
	&={\int}_{x \in U''_nA''_n}
		\delta^{-1}(x)\,\, \tr \int_{yau \in N_nA'_nU'_n}(Dh)(x^{-1}yaux)
		\times \\
	&\phantom{XXXXXXXXXXXXXXXXXXXX}\times 
	(\pi^\dagger_{\gamma,n}\otimes \exp(2\pi i\xi)\mu'_n)
		   (yau)\,dy\, da\, du\, dx\, \\
&\hskip -0.7cm = {\int}_{x \in U''A''}
	\tr \int_{N_nA'_nU'_n}(Dh)(yx^{-1}aux)\cdot
		(\pi^\dagger_{\gamma,n}\otimes \exp(2\pi i\xi)\mu'_n)
		(xyx^{-1}u)\,
		dy\, da\, du\, dx.
\end{aligned}
$$
Now
$$
\begin{aligned}
{\sum}_{\widehat{U'_n}}\int_{\widehat{A'_n}}
  &\tr \pi_{\gamma,\xi,\mu',n}(Dh) d\xi\deg\mu'_n \\
&= {\sum}_{\widehat{U'_n}} \int_{\widehat{A'_n}}\int_{x \in U''_nA''_n} 
        \tr \int_{N_nU'_nA'_n}(Dh)(yx^{-1}aux)
                \times \\ 
        &\phantom{XXXXXXXXXXXXXXX}\times
                (\pi^\dagger_{\gamma,n}\otimes \exp(2\pi i\xi)\mu'_n)(xyx^{-1}au)\,
                dy\, da\, du\, dx\, d\xi\, \deg\mu'_n \\
&= \int_{x \in U''_nA''_n}{\sum}_{\widehat{U'_n}}\int_{\widehat{A'_n}}
	\tr \int_{N_nU'_n}(Dh)(yx^{-1}aux)
                \times \\ 
        &\phantom{XXXXXXXXXXXXXXX}\times
                (\pi^\dagger_{\gamma,n}\otimes \exp(2\pi i\xi)\mu'_n)(xyx^{-1}au)\,
		dy\, da\, du\, d\xi\, \deg\mu'_n\, dx\\
&= \int_{x \in U''_nA''_n} \tr \int_{N_n}(Dh)(y) 
	\pi^\dagger_{\gamma,n}(xyx^{-1}) dy\, dx \\
&= \int_{x \in U''_nA''_n} \tr \int_{N_n}(Dh)(y)
        (\Ad^*(x)^{-1}\cdot \pi^\dagger_{\gamma,n})(y) dy\, dx \\
&= \int_{x \in U''_nA''_n} \tr ((\Ad^*(x)^{-1}\cdot
	\pi^\dagger_{\gamma,n})(Dh))\, dx\\
&= \int_{x \in U''_nA''_n} (\Ad^*(x)^{-1}\cdot \pi_{\gamma,n})_*(D)\,\,
	\tr(\Ad(x)^{-1}\cdot \pi^\dagger_{\gamma,n})(h)\, dx\\
&= \int_{x \in U''_nA''_n} (\Ad^*(x)D)\,\,
        \tr(\Ad(x)^{-1}\cdot \pi^\dagger_{\gamma,n})(h)\, dx\\
&= \int_{x \in U''_nA''_n} \delta(x)
        \tr(\Ad(x)^{-1}\cdot \pi^\dagger_{\gamma,n})(h)\, dx
= \int_{\gamma'_n \in \Ad^*(U_nA_n)\gamma_n} 
	\tr \pi^\dagger_{\gamma',n}(h) 
	|\Pf(\gamma'_n)|d\gamma'_n .
\end{aligned}
$$
Summing over the the space of $U_nA_n$-orbits on $\gs_n^*$ 
we now have
$$
\begin{aligned}
\int_{\gamma_n \in\Ad^*(U_nA_n)\backslash \gs_n^*}
	& \left ( {\sum}_{\widehat{U'_n}}\int_{\widehat{A'_n}} 
	\tr \pi_{\gamma,\xi,\mu',n}(Dh)d\xi \deg\mu'_n\right ) d\gamma_n \\
& = \int_{\gamma_n \in\Ad^*(U_nA_n)\backslash \gs_n^*} \left (
	\int_{\gamma'_n \in \Ad^*(U_nA_n)\gamma_n}
	\tr \pi^\dagger_{\gamma'_n}(h)  
		|\Pf(\gamma'_n)|d\gamma'_n \right ) d\gamma_n \\
& = \int_{\gamma_n \in \gs_n^*} \tr \pi_{\gamma_n}(h) |\Pf(\gamma_n)|d\gamma_n 
	= h(1) = f(x)\,.
\end{aligned}
$$
That completes the proof.
\end{proof}

\medskip
\noindent Department of Mathematics \hfill\newline
\noindent University of California\hfill\newline
\noindent Berkeley, California 94720-3840, USA\hfill\newline
\smallskip
\noindent {\tt jawolf@math.berkeley.edu}

\enddocument
\end